\documentclass[12pt,a4paper]{amsart}
\usepackage[english]{babel}
\usepackage{color}
\usepackage[T1]{fontenc}
\usepackage{amsmath}
\usepackage{amscd}
\usepackage{amstext}
\usepackage{amsbsy}
\usepackage{amsopn}
\usepackage{amsthm}
\usepackage{amsxtra}
\usepackage{upref}
\usepackage{amsfonts}
\usepackage{amssymb}
\usepackage{mathrsfs}
\usepackage{euscript}
\usepackage{array}
\usepackage{stmaryrd}
\usepackage{verbatim}

\theoremstyle{plain}
      \newtheorem{theorem}{Theorem}
      \newtheorem{lemma}[theorem]{Lemma}
      \newtheorem{corollary}[theorem]{Corollary}
      \newtheorem{proposition}[theorem]{Proposition}
     
      \theoremstyle{definition}

     \theoremstyle{remark}
     \newtheorem{remark}[theorem]{Remark}
     \theoremstyle{Fact}
     
     \theoremstyle{Example}
     
\theoremstyle{notation}

\newcounter{Step}
\newenvironment{step}[0]{\bigskip\addtocounter{Step}{1}\noindent\textbf{Step \theStep :} }{\
  \begin{flushright} \end{flushright}}

\def\N{\mbox{I\hspace{-.15em}N} }
\def\R{\mbox{I\hspace{-.15em}R} }
\def\Q{\mbox{l\hspace{-.47em}Q} }
\def\C{\hspace{.17em}\mbox{l\hspace{-.47em}C} }

\def\Z{\mbox{Z\hspace{-.30em}Z} }

\def\ot{\otimes}

\def\E{\mathscr{E}}

\begin{document}

\title[Model theories for von Neumann algebras]{Continuous model theories for von Neumann algebras}
\author[Y. Dabrowski]{Yoann Dabrowski}\address{ 
Universit\'e de Lyon\\
Universit\'e Lyon 1\\
Institut Camille Jordan\\
43 blvd. du 11 novembre 1918\\
F-69622 Villeurbanne cedex\\
France} 
\email{dabrowski@math.univ-lyon1.fr}
\thanks{Research partially supported by ANR grant NEUMANN}

\date{}
\maketitle

\begin{abstract}
We axiomatize in (first order finitary) continuous logic for metric structures $\sigma$-finite $W^*$-probability spaces and preduals of von Neumann algebras jointly with a weak-* dense $C^*$-algebra of its dual. This corresponds to the Ocneanu ultrapower and the Groh ultrapower of ($\sigma$-finite in the first case) von Neumann algebras. We give various axiomatizability results corresponding to recent results of Ando and Haagerup including axiomatizability of $III_\lambda$ factors for $0<\lambda\leq 1$ fixed and their preduals. We also strengthen the concrete Groh theory to an axiomatization result for preduals of von Neumann algebras in the language of tracial matrix-ordered operator spaces, a natural language for preduals of dual operator systems.   We give an application to the isomorphism of ultrapowers of factors of type $III$ and $II_\infty$ for different ultrafilters.
\end{abstract}

\section*{Introduction}

The model theory of metric structures (see \cite{BenYBHU}) was recently applied to
analyze ultrapowers of $C^*$-algebras and tracial von Neumann algebras \cite{FarahI,FarahII,FarahIII}. From an operator algebraic viewpoint, the relations between the various ultraproducts of von Neumann algebras (or $\sigma$-finite von Neumann algebras) was recently clarified in \cite{HaagerupAndo}. It is thus expected that some continuous model theory should enable to study those ultraproducts, beyond the tracial case. 

This is the goal of this paper to give axiomatizations in continuous model theory of various classes of von Neumann algebras and recover as model theoretic ultraproducts the two main ultrapowers : the Ocneanu ultrapower $(M,\varphi)^\omega$ of a $\sigma$-finite von Neumann algebra jointly with a faithful normal state $\varphi$ (as usual, we will call the pair $(M,\varphi)$  a $\sigma$-finite $W^*$-probability space) and Groh's ultrapower $\prod^\omega M$ of a von Neumann algebra $M$. In the second case, our model theory will rather be a model theory of the class of preduals of von Neumann algebras, giving an axiomatization of these preduals in continuous model theory. This is not surprising since the Groh ultraproduct is by definition the dual of the metric ultraproduct of preduals. All our axiomatizations will be in the continuous model theory setting for operator algebras from \cite{FarahII}, a multidomain variant of the first order (metric) continuous logic from \cite{BenYBHU}. We want to mention that Ilijas Farah and Bradd Hart checked in an unpublished work that general von Neumann algebras form a ``compact abstract theory" in the sense of \cite{BenYCAT}. This is crucial for our purposes that we restrict to $\sigma$-finite von Neumann algebras where a faithful normal state is available to get an axiomatization in the better behaved metric setting with more convenient syntactic counterpart of the semantic. However, for preduals of von Neumann algebras, we do axiomatize without any $\sigma$-finiteness assumption (a property that wouldn't be axiomatizable in our language for preduals anyway.)

Let us point out that our axiomatizations will be often explicit, but sometimes, as for preduals of von Neumann algebras (theorem \ref{GrohThPredual}), we will obtain the existence of an axiomatization in some explicit language by using a standard model-theoretic result  \cite[Prop 5.14]{BenYBHU} and proving only stability by ultraproducts and ultraroots of a class of models. In this case, the stability by ultraproducts for ultrafilters on $\N$ is always contained in \cite{HaagerupAndo} (sometimes with another language) but we give a model theoretic proof of the general ultraproduct case. This enables us to prove axiomatization results for the natural classes found to be stable by countable ultraproducts in \cite{HaagerupAndo}, such as $III_\lambda$ factors for a fixed $0<\lambda\leq 1$. Note that since the model theoretic result on axiomatizability is based on Keisler-Shelah theorem characterizing elementary equivalence, considering ultraproducts for ultrafilters on $\N$ is a priori (in absence of the continuum hypothesis) not enough  even in the separable case (there are structure for which ultraproducts for ultrafilters on $\N$ are known to be not enough to characterize elementary equivalence, at least consistently with ZFC, see \cite{Shelah} and \cite[Question after Rmk 4.2]{FarahIII}). Our extra-work  with uncountable ultraproducts is thus necessary to use the available model theoretic results. 

We also want to emphasize that, even though it is inspired from \cite{HaagerupAndo}, we give an alternative construction of the Ocneanu ultraproduct that does not use the relation to the Groh ultraproduct to prove we have a von Neumann algebra structure. At the end, we use their full result only to identify our ultraproduct with the usual Ocneaunu ultraproduct.

We have several motivations behind our study. First, even for the study of tracial von Neumann algebras $M$, various $II_\infty$ von Neumann algebras appear naturally, not only $M\otimes B(H)$ but also basic constructions $\langle M,e_B\rangle$  for subfactors $B\subset M$ in the infinite index case. Our theorem \ref{MoreAxiom} contains nice axiomatizability results for this  and other classes that we think will be important for the finite case since $(M\otimes B(H))^\omega\simeq M^\omega\otimes B(H)$ for Ocneanu ultrapower and the theory of $M\otimes B(H)$ with some way of identifying $B(H)$ (axiomatized by $T_{\sigma W^*geom}$ in theorem \ref{MoreAxiom}) is thus an alternative theory for $M$ (not only up to compression since the matrix unit put in the theory recovers $M$ in $M\otimes B(H)$). This theory should deserve more investigation. We also give a related theory of $B(H)$, with a matrix unit as constants, having quantifier elimination.

Of course, as a second motivation, having a model theoretic axiomatization enables to use interesting continuous model theoretic tools to study ultraproducts. Let us give a first consequence in the next theorem for isomorphism of factorial ultraproducts and leave for further investigation the study of stability of ($\sigma$-finite) von Neumann algebras parallel to \cite{FarahI}.

\begin{theorem}\label{stability} Let $M$ a von Neumann algebra with separable predual and $\varphi$ a faithful normal state on $M$.
\begin{enumerate}
\item If the Continuum Hypothesis holds, then for any nonprincipal ultrafilters  $\mathcal{U},\mathcal{V}$ on $\N$, we have isomorphisms  of the Groh ultrapowers  $\prod^\mathcal{U}M\simeq \prod^\mathcal{V}M$ and the Ocneaunu ultrapowers $(M,\varphi)^\mathcal{U}\simeq(M,\varphi)^\mathcal{V}$.
\item If the Continuum Hypothesis fails and $M$ is a factor which is not of type $III_0$, then $M$  is not of type $I$ if and only if there exist nonprincipal ultrafilters $\mathcal{U},\mathcal{V}$  on $\N$ such that $(M,\varphi)^\mathcal{U}\not\simeq(M,\varphi)^\mathcal{V}$ as von Neumann algebras.
\end{enumerate}
\end{theorem}

We give the proof in the last subsection. It uses in an intrinsically linked way various strong model theoretic results available thanks to our various axiomatizations and general structure theory of factors well-known to operator algebraist. The reader familiar with the finite case and/or some structure theory of type $III$ factors can probably read it right away without reading all the axiomatization details of the general case. The second point partially generalizes \cite[Th 4.7]{FarahI} (in the factor case). We conjecture that it is also valid for type $III_0$ factors (as a motivation the closest available theory for their discrete decomposition described in the language of section 4 is easily seen to have the order property in a similar way as in our proof but this is a priori not enough to contradict a von Neumann algebraic isomorphism of all ultrapowers) and probably for non-factors but since ultraproducts of type $III_0$ factors are usually not factors, we stick here to the statement for factorial ultrapowers which are easier to deal with. We also emphasize that even though Ocneanu theory is a theory of $W^*$-probability spaces that strongly depends on the state $\varphi$ we put on $M$, thanks to general structure theory in the factor case and axiomatizability of certain kinds of states (such as periodic states in the $III_\lambda$, $0<\lambda<1$ case) the final non-isomorphism above of ultrapowers is at the von Neumann algebraic level and not only at the $W^*$-probability space/model theoretic level (as in an earlier preprint version of this paper).

Let us now summarize the main ideas and results. Of course, the main problem in the general von Neumann algebra case is to deal with Tomita-Takesaki modular theory.  Section 1 is mainly concerned with producing a theory corresponding to the Ocneanu ultraproduct.
Of course, since the modular group is well behaved for this ultraproduct \cite[Th 4.1]{HaagerupAndo}, it is natural to put it in the theory.

 Subsection 1.1 proves elementary lemmas with the goal of identifying the language and several crucial properties of the theory that will enable us to characterize most of the pieces of this language (including the modular group) in first order continuous logic. Its goal is to give operator algebraic background for model theorist while showing to operator algebraist the way to look for useful operator algebraic results for our axiomatization purposes.  Note that the KMS condition does not seem easy to express in this way and we have to prefer explicit integral formulas for unbounded operators. This makes the axiomatization much trickier than for tracial von Neumann algebras or $C^*$-algebras.  In the non-tracial case, the choice of the topology turns out to be crucial. Since a model theoretic ultraproduct is always a quotient of bounded sequences, we have to consider a topology so that the Ocneanu ultraproduct will be given by a quotient of a set of bounded sequences and not some multiplier algebra as in the original definition. The inspiration comes from such a quotient description in \cite[Proposition 3.14]{HaagerupAndo}. Instead of the *-strong topology, we use the topology for which a net $x_n\to 0$ if $x_n=y_n+z_n$ with $y_n$ converging strongly to 0 and $z_n^*$ converging strongly to 0. It turns out that this topology comes from a norm with a nice modular theory formula (lemma \ref{NormG}) and an explicit unusual lipschitzianity estimate in time for the modular group (lemma \ref{ContinuityModular}) enabling to describe lots of integral formulas for various objects produced from the modular group as Riemann integrals having easy continuous logic definitions. This explicit uniform continuity of the modular group for the relevant metric makes much less surprising the result of \cite{HaagerupAndo} concerning its preservation by ultraproduct. The second most critical problem is the identification of the product which is not uniformly continuous for neither the strong topology nor the topology described above on the unit ball of $M$. The solution is standard in modular theory, spectral algebras have the nice continuity properties fixing this issue, but unfortunately, they are not in general stable by ultraproducts. Thus we follow the concrete way nicely commuting with ultraproducts that Ando and Haagerup used to produce elements in spectral algebras involving Fejer maps $F_N^\varphi$ obtained by integrating Fejer's kernel with the modular group. We will only put in the theory smeared products $m_{N,M}(x,y)=F_N^\varphi(x).F_M^\varphi(y)$. Combined with standard estimates on spectral algebras recalled in lemma \ref{ContinuityProduct}, this will give all the necessary pieces of data for the theory : the state, the adjoint, the metric, the modular group and the smeared products. However, to obtain a universal axiomatization, we also use various other data obtained from the modular group  by Riemann integrals coming from standard spectral theory results that we recall in lemma \ref{EquationForms}. This especially gives us the generator of the modular group $\Delta$ as a sesquilinear form $\mathscr{E}_1$ that we characterize using as intermediate steps the forms $\mathscr{E}_\alpha$ for fractional powers $\Delta^\alpha,$ $\alpha\in ]0,1[.$ This is in this way that we characterize the modular group of the theory as the one coming from $\varphi$.
With these preliminaries at hand, we can write down our axiomatization in subsection 1.2 and produce, in the proof of theorem \ref{OcneanuTh}, a von Neumann algebra from a model of the theory by a GNS construction starting from an algebra generated by the various $F_N^\varphi(x)$ between which the product is already defined. This especially does not use any relation to the Groh ultraproduct or any other non-$\sigma$-finite von Neumann algebra. Of course, the theory reduces to the tracial theory when the state is a trace, all the supplementary data being trivial, for instance all the smeared products $m_{N,M}$ are identically the usual multiplication map.

Subsection 1.3 then gives various supplementary axiomatization results for natural classes of von Neumann algebras in theorem \ref{MoreAxiom}. The reader should notice that the non-explicit axiomatization for type $III_\lambda$ factors $\lambda\in(0,1]$ uses for its proof various generalisations of results of \cite{HaagerupAndo} obtained later in  
 sections 3 and 4. We also advertise in Rmk \ref{NonAxiom} various non-axiomatizability results in our language for $W^*$-probability spaces straightforwardly deduced from results in \cite{HaagerupAndo}. Here the dependence on the state of the Ocneanu theory is crucial.
 
 Subsection 1.4, suggested by questions of Ilijas Farah and Ita\"i Ben Yaacov, gives a more minimal language where an axiomatization of $\sigma$-finite $W^*$-probability spaces is available, at the cost of loosing the universal explicit axiomatization. The main point is to check a definability in the sense of \cite{BenYBHU} of the modular group in a minimal enough language using some more technical (but standard) spectral theory.
 
Section 2 shows the axiomatizability of preduals of von Neumann algebras in a natural language giving an ultraproduct corresponding  to the Groh ultraproduct by taking duals. Unfortunately, neither Groh's construction nor the related construction of Raynaud (of ultraproducts of standard forms) considered instead in \cite{HaagerupAndo} gives an insight in a possible language for such a theory. Looking at Groh's construction rather suggests a theory for a pair of a predual $X$ of a von Neumann algebra and a weak-* dense $C^*$-subalgebra of $X^*$. And of course, the $C^*$ algebra structure is strongly used to show $X^*$ is a von Neumann algebra. We thus look for an implicit axiomatization. The strategy is to use Groh's idea to obtain stability by ultraproducts. But to identify the language, we look for a natural language  containing enough information to obtain stability by ultraroot. This right language becomes clear in the proof of theorem \ref{GrohThPredual}. As should be expected, it contains the structure of the predual $X$ as an operator space, some orders (as in the commutative case) but in each $M_n(X)$ and a dualization of a unit in $X^*$ which is usually called Haagerup trace. This theory is chosen in order to use a result of Choi and Effros giving a $C^*$ algebra structure to an operator system, image of a completely positive unital projection on a subspace of a $C^*$ algebra that will be an ultrapower of the subspace we want to put a $C^*$-algebra structure on. The language is thus natural for a predual of a dual operator system. Then once identified the theory with a language rich enough to obtain stability by ultraroot, we prove (and this is our starting point and the content of theorem \ref{GrohTh}) that Groh's expansion with  a $C^*$ algebra can be used to strengthen his result and give stability by ultraproduct in this stronger but still quite natural language. Note once again that section 2 leaves open the question of an explicit theory for preduals of von Neumann algebras. It would be interesting if such a theory followed the suggestion of Effros-Ruan \cite[p 303]{EffrosRuan} and could be based on looking at the subspaces completely isometric to trace class operators. However our paper gives a first answer to the related question of finding such an axiomatization of preduals using operator space/system theory and not involving the dual $C^*$ algebra structure (even in the form of a coproduct). Since the stability by ultraproducts in \cite{HaagerupAndo} of type $III_\lambda$ factors comes from properties of their preduals, the source of the corresponding stability properties is in this section and has an analogue for preduals. 

Of course, to be used to study the Ocneanu theory, one needs a relation between the Groh and Ocneanu theories. To made our theory easier to read, we start by an axiomatization of standard forms in subsection 2.3. This is a natural expansion of Groh's theory. The theory gathering Groh and Ocneanu theories is suggested from the corresponding relations of ultraproducts in \cite{HaagerupAndo} and written down explicitly in section 3. That's why we call it the Ando-Haagerup theory.

Finally, to get the lack of stability by ultrapowers of type $III_0$ factors, one needs a stability property of their discrete decomposition that we obtain in the case of non-countable ultraproducts in section 4 (the countable case coming from \cite{HaagerupAndo} again). This is a quite tricky explicit theory using the theory of section 3, an explicit theory for $II_\infty$ von Neumann algebras with a geometric state and a unitary implementing a crossproduct of $\Z$ with such a $II_\infty$ $W^*$ probability space. Note that the theory depends on parameters so that we don't find a single axiomatizable class containing discrete decompositions of $III_0$-factors. Of course, the union of all our classes contain all such 
 discrete decompositions.
 
 Let us finish by pointing out that, following the operator algebraic tradition, $\omega$ will always be a non-principal ultrafilter on a set $I$ (maybe uncountable).
\subsection*{Acknowledgments} The author is grateful to Ita\"i Ben Yaacov for helpful discussions  that motivated this investigation and helped improve its clarity. 
He also thanks Ilijas Farah for useful comments that suggested a more self-contained statement of theorem 1. He thanks both for asking questions leading to subsection 1.4.

\section{The Ocneanu Theory for $\sigma$-finite von Neumann algebras}\label{Ocneanu}
 \subsection{Setting and preliminaries}
We treat von Neumann algebras with a fixed state $\varphi$ (intended to be a faithful normal state) which will have one sort $U$ with domains of quantification $D_n$ for the operator norm ball of radius $n$. The metric $d$ will be related to $\varphi$ bellow in a way reducing to the usual $L^2$ norm when $\varphi$ is a trace, thus reducing to \cite{FarahII} in this special case.

Let us recall several norms related to $\varphi$ :
\[||x||_\varphi^2=\varphi(x^*x),\ \ \ ||x||_\varphi^{\#}=\sqrt{||x||_\varphi^2+||x^*||_\varphi^2}\] so that it is well-known that $||.||_\varphi$ defines the strong operator topology and $||.||_\varphi^{\#}$ the strong-* operator topology on the unit ball of $M$ in the $\sigma$-finite case with $\varphi$ faithful (these results are of course not obvious and will be explained later, they depend on the modular theory).

For our purposes, another less usual norm will be much more important :
\[ ||x||_\varphi^{*}=\inf_{y\in M}\left[\sqrt{\varphi(y^*y)+\varphi((x-y)(x-y)^*)}\right]
.\]

We want to take $d(x,y)=||x-y||_\varphi^{*}$. The motivation for this is the description of the Ocneaunu ultraproduct as a quotient vector space in \cite[Proposition 3.14]{HaagerupAndo}. By Ocneanu theory, we mean that we want a theory such that model theoretic ultraproduct recovers the Ocneanu ultraproduct.

Thus we have to check that all common operations are uniformly continuous, and specify modulus of continuity.
Note first that \[||x^*||_\varphi^{*}=||x||_\varphi^{*},\] and  $|\varphi(x)|\leq |\varphi(y)+\varphi(x-y)|\leq \sqrt{2}\sqrt{\varphi(y^*y)+\varphi((x-y)(x-y)^*)}$ for any $y$ by Cauchy-Schwarz and thus \[|\varphi(x)|\leq \sqrt{2}||x||_\varphi^{*}.\] Unfortunately, there is no uniform continuity bound for product.
To deal with that, we will use modular theory. We refer to \cite{TakesakiBook} for general results or \cite{HaagerupAndo} for some more specific properties. The state $\varphi$ on $M$ is by now always faithful and normal.

Only recall that if $\xi_\varphi$ denotes the GNS vector for $\varphi$, $S^0_\varphi(x\xi_\varphi)=x^*\xi_\varphi$ defines a densely defined closable operator with closure $S_\varphi$ such that $\Delta_\varphi=S_\varphi^*S_\varphi$ and the polar decomposition $S_\varphi=J_\varphi\Delta_\varphi^{1/2}$. $J_\varphi$ is called the modular conjugation operator and $\Delta_\varphi$ the modular operator.  The modular automorphism group is then defined by : \[\sigma_t^\varphi(x)=\Delta_\varphi^{it}x\Delta_\varphi^{-it}\]
Tomita's fundamental Theorem states that $\sigma_t^\varphi$ leave $M$ invariant and even defines a one parameter automorphism group of $M.$ Especially, $\sigma_t^\varphi$ preserves adjoint and $\varphi\circ \sigma_t^\varphi=\varphi.$
We will need Arveson's spectral theory (see \cite[Section XI.1]{TakesakiBook}).
For $f\in L^1(\R)$, one can then define for $x\in M$ \[\sigma_f^\varphi(x):=\int_{\R}dt f(t)\sigma_t^\varphi(x)\in M,\ \ \  \hat{f}(y)=\int_{\R}dt f(t)e^{ity}\]
so that we have a relation between our Fourier transform and functional calculus: \[\sigma_f^\varphi(x)\xi_\phi=\hat{f}(\ln(\Delta))(x\xi_\phi).\] Note that $\sigma_f^\varphi\circ \sigma_g^\varphi=\sigma_g^\varphi\circ \sigma_f^\varphi.$
Then the spectrum of $x$ is better understood by describing its complement as support usually is :
\[[\mathrm{Spec}_{\sigma^\varphi}(x)]^c=\{t\in \R \ : \  \exists f\in L^1(\R),\ \hat{f}(t)\neq 0\ \mathrm{and}\ \sigma_f(x)=0\}.\]
Conversely, from \cite[Lemma XI.1.3]{TakesakiBook}, if $x\neq 0$, $\mathrm{Spec}_{\sigma^\varphi}(x)\neq \emptyset$ and if $\text{supp}(\hat{f})\subset [\mathrm{Spec}_{\sigma^\varphi}(x)]^c$ then $\sigma_f(x)=0.$
The crucial definition for us is the spectral subspace of a subset $E\subset \R$ :
\[M(\sigma^\varphi,E)=\{x\in M : \mathrm{Spec}_{\sigma^\varphi}(x)\subset E\}.\]

$M(\sigma^\varphi,\{0\})$ is called the centralizer of $\varphi$ and from \cite[Corol XI.1.8]{TakesakiBook}, we have \[M(\sigma^\varphi,E)^*=M(\sigma^\varphi,-E);\ \ \ M(\sigma^\varphi,E)M(\sigma^\varphi,F)\subset M(\sigma^\varphi,\overline{E+F}).\]

We will also use the Arveson spectra defined by its complement:\[[\mathrm{Sp}(\sigma^\varphi)]^c=\{t\in \R\ :\  \exists f\in L^1(\R),  \  \hat{f}(t)\neq 0\ \mathrm{and}\ \sigma_f=0\}=\{t\in \R\ :\ \exists \epsilon >0,\  M(\sigma^\varphi,[t-\epsilon,t+\epsilon])=\{0\}\}\]

The following result is deduced from an old result of Haagerup (cf the proof of \cite[lemma 4.13]{HaagerupAndo})

\begin{lemma}\label{ContinuityProduct}
For any $a>0$, $x\in M(\sigma^\varphi,[-a,a]),y\in M$, we have, with $C_a=2e^{a}+e^{a/2}$:\[||(xy)^*||_\varphi\leq C_a||x||\ ||y^*||_\varphi,\ \ \ ||xy||_\varphi^\#\leq C_a||x||\ ||y||_\varphi^\#, \ \ \ ||xy||_\varphi^*\leq C_a||x||\ ||y||_\varphi^*.\]
\end{lemma}
Thus product will be uniformly continuous on balls of $M(\sigma^\varphi,[-a,a])$. We could try taking those balls as domain of quantification $E_{a,n}, a,n-1\in \N$ of another sort $V,$ but they are in general not stable by ultraproduct and cannot be put in the theory. 
However, we will use the next completeness result:

\begin{lemma}\label{CompleteBalls}
The operator norm unit balls $(M)_1$ and $(M(\sigma^\varphi,[-K,K]))_1$, for  $K\geq 0,$ are complete for $d$.
\end{lemma}
\begin{proof}
For, take a Cauchy net $a_n$ say in the unit ball $M_1$, then take a  decomposition $a_n=b_n+c_n, b_n,c_n\in M$, $b_n$ Cauchy for $||.||_\varphi$ and $c_n$ Cauchy for $||(.)^*||_\varphi$ as is possible from the definition of $d$ given in (12). From the proof of the next lemma, one can take $||b_n||_\varphi\leq 2||x||, ||(c_n)^*||_\varphi\leq 2||x||.$ From the completeness of $L^2(M,\varphi)$ there is $b,c^*\in L^2(M,\varphi)$ such that $||b_n-b||_\varphi\to 0, ||c_n^*-c^*||_\varphi\to 0$. Let us call $a=b+c$ which is also the weak-* limit of $a_n=b_n+c_n$, which is thus in the operator norm unit ball $(M)_1$. Finally using the alternative infemum describing $d$ in the proof of the next lemma (since $b_n-b=a_n-a-(c_n-c)\in D(\Delta^{1/2})$):\[d(a_n-a,0)^2\leq||b_n-b||_\varphi^2+||c_n^*-c^*||_\varphi^2\to 0,\] thus $(M)_1$ is indeed complete. 
Take any $g\in L^1(\R)$ with $\text{supp}(\hat{g})\subset [-K,K]^c$, then if $a_n\in M(\sigma^\varphi,[-K,K])$ as above, $\sigma_g^\varphi(a_n)=0$ and, using the formula obtained in the next lemma, \[ ||\sigma_g^\varphi(a_n-a)||_\varphi^*=||\sigma_g^\varphi G_0^\varphi(a_n-a)||_\varphi^\#\leq || G_0^\varphi(a_n-a)||_\varphi^\#=||(a_n-a)||_\varphi^*\to 0\] and thus $\sigma_g^\varphi(a)=0$ and since this is for all $g$ as above, $a\in M(\sigma^\varphi,[-K,K]).$
\end{proof}

We also record the following useful spectral theory result and deduce the modular theory formula for our distance. We already needed it in our previous lemma and we will use it crucially later.
\begin{lemma}\label{NormG}
If $g_s(t)=\frac{2e^{-ist}}{e^{\pi t}+e^{-\pi t}}$ then $||g_s||_{L^1(\R)}=1$ and \[\sigma_{g_s}^\varphi(x)\xi_\varphi=2e^{s/2}\Delta^{1/2}(\Delta+e^s)^{-1}(x\xi_\varphi),\]
and, if we call $G_s^\varphi=\sigma_{g_s}^\varphi$ we have the equality, for any $x\in M$ :
\[2||x||_\varphi^*=||G_0^\varphi(x)||_\varphi^\#.\]
\end{lemma}
\begin{proof}
This first equality is 
\cite[Lemma VI.1.21]{TakesakiBook}.
A completeness argument and a computation shows \begin{align*} (||x||_\varphi^{*})^2&=\inf_{y\in M}\left[||y\xi_\varphi||^2+||\Delta^{1/2}x\xi_\varphi||^2+||\Delta^{1/2}(y\xi_\varphi)||^2-2\Re\langle \Delta^{1/2}(y\xi_\varphi),\Delta^{1/2}(x\xi_\varphi)\rangle\right]\\&=\inf_{\eta\in D(\Delta^{1/2})}\left[||\eta||^2+||\Delta^{1/2}x\xi_\varphi||^2+||\Delta^{1/2}(\eta)||^2-2\Re \langle \Delta^{1/2}(\eta),\Delta^{1/2}(x\xi_\varphi)\rangle\right]
.\end{align*}
By Lax-Milgram lemma (see e.g. \cite[Corol V.8]{Brezis} with $a(u,v)=\langle (1+\Delta) u,v\rangle$)
, the infemum is easily reached at $\eta=\eta_0=\Delta(1+\Delta)^{-1}(x\xi_\varphi)$. Indeed, the minimization problem is equivalent to finding the $\inf$ of $a(\eta,\eta)-2\Re [a(\eta,\eta_0)]+a(\eta_0,\eta_0)=a(\eta-\eta_0,\eta-\eta_0)$ which is obviously minimal at $\eta=\eta_0.$

Since $x\xi_\varphi-\Delta(1+\Delta)^{-1}(x\xi_\varphi)=(1+\Delta)^{-1}(x\xi_\varphi)$ we obtain by an easy computation : \[(||x||_\varphi^{*})^2=||\Delta(1+\Delta)^{-1}(x\xi_\varphi)||^2+||\Delta^{1/2}(1+\Delta)^{-1}(x\xi_\varphi)||^2=||\Delta^{1/2}(1+\Delta)^{-1/2}(x\xi_\varphi)||^2.\]
On the other hand, a similar easy computation gives 
\[\frac{(||G_0^\varphi(x)||_\varphi^\#)^2}{4}=||\Delta^{1/2}(\Delta+1)^{-1}(x\xi_\varphi)||^2+||\Delta(\Delta+1)^{-1}(x\xi_\varphi)||^2=||\Delta^{1/2}(1+\Delta)^{-1/2}(x\xi_\varphi)||^2.\]
\end{proof}

From this we also deduce an explicit (uniform) continuity bound for the modular group in the distance $d$, and deduce from this an explicit bound giving an approximation formula for certain $\sigma_f^\varphi$ that will be used in our axiomatization.

\begin{lemma}\label{ContinuityModular}
For any $x\in M$ we have $||\sigma_t^\varphi(x)-x||_\varphi^*\leq 2t||x||_\varphi^\# $ and $||\sigma_t^\varphi(x)||_\varphi^*\leq  ||f||_{L^1(\R)} ||x||_\varphi^*.$

As a consequence, for $x\in M$ and $f\in L^1(\R)\cap C^1_b(\R)$, then  \[\left\|\sigma^\varphi_f(x)-\frac{1}{n^2}\sum_{k=-n^3}^{n^3-1}f(\frac{k}{n^2})\sigma^\varphi_{k/n^2}(x)\right\|_\varphi^*\leq ||1_{]-\infty, n]\cup [n,\infty[}f||_1||x||_\varphi^*+\frac{2||f||_1}{n^2}||x||_\varphi^\#+\frac{||f'||_\infty}{n}||x||_\varphi^*\]
\end{lemma}
We will be especially interested in this result for Fejer's kernel
$f_m:\R\to \R,m>0$ defined by \[f_m(t)=\frac{m}{2\pi}1_{\{t=0\}}+1_{\{t\neq 0\}}\frac{1-\cos(mt)}{\pi m t^2}\geq 0\] of Fourier transform $\widehat{f_m}(t)=\max(0,1-|t[/m).$ Note that $||f_m||_1=1$. We will write \[F_m^\varphi=\sigma_{f_m}^\varphi.\] Note also that $||f_m'||_\infty\leq\frac{m^2}{\pi}$ 
 and $||1_{]-\infty, n]\cup [n,\infty[}f_m||_1\leq \frac{8}{\pi m n^3}.$

More generally, we will write $f_{m,l}(t)= f_m(t)e^{itl}$ the variant with translated Fourier transform, $F_{m,l}^\varphi=\sigma_{f_{m,l}}^\varphi$ and $||f_{m,l}'||_\infty\leq\frac{m^2}{\pi}+\frac{lm}{2\pi}.$

\begin{proof}From the proof of our previous lemma, bounding $||\sigma_t^\varphi(x)-x||_\varphi^*$ corresponds to bounding:
\[||\Delta^{1/2}(1+\Delta)^{-1/2}([\sigma_t^\varphi(x)-x]\xi_\varphi)||^2=\langle\frac{|\Delta^{it}-1|^2}{\Delta+\Delta^{-1}}(\Delta^2+1)(1+\Delta)^{-1}x\xi_\varphi,x\xi_\varphi \rangle\leq (2t)^2\langle(1+\Delta)x\xi_\varphi,x \xi_\varphi\rangle \]
where we used spectral theory and the elementary bound on $\R^2$, $\frac{|e^{ith}-1|^2}{e^h+e^{-h}}\leq t^2h^2e^{-|h|}\leq (2t)^2e^{-2}$ (from the maximum achieved at $h=2$) and well-known bounds $||(1+\Delta)^{-1}||\leq 1,||\Delta(1+\Delta)^{-1}||\leq 1.$ The last bound is nothing but the expected one.

For the second estimate, it suffices to decompose naively (and bound as naively) :
\begin{align*}\int_{\R}dt f(t)\sigma_t^\varphi(x)&=\sum_{k=-n^3}^{n^3-1}\int_{k/n^2}^{(k+1)/n^2}dt f(t)(\sigma_t^\varphi(x)-\sigma_{k/n^2}^\varphi(x))+(f(t)-f(k/n^2))\sigma_{k/n^2}^\varphi(x)\\&+\int_{]-\infty, n]\cup [n,\infty[}dt f(t)\sigma_t^\varphi(x)+\frac{1}{n^2}\sum_{k=-n^3}^{n^3-1}f(\frac{k}{n^2})\sigma^\varphi_{k/n^2}(x).\end{align*}

\end{proof}


We will also need a way to identify the spectral algebras. This of course works in a more general setting of covariant systems $(M,\R,\sigma)$ over $\R$ (cf. e.g. \cite[section XI.1]{TakesakiBook}). This is really standard and we only include a proof for the reader's convenience since we will use this quite often.

\begin{lemma}\label{SpectralAlgebras}
Let $x\in M$ and $K\in \N$. Then  $x\in M(\sigma,[-K,K])$ if and only if for any $L\geq 2K,L\in \N$ $\sigma_{f_{K,\pm L}}(x)=0.$
\end{lemma}
\begin{proof}
For $x \in M(\sigma,[-K,K])$, 
note that $[-K,K]\cap \text{supp}(\widehat{f_{K-\epsilon,\pm L}})=[-K,K]\cap[-K+\epsilon\mp L,K-\epsilon\mp L])=\emptyset$ thus by \cite[lemma XI.1.3]{TakesakiBook} $\sigma_{f_{K-\epsilon,\pm L}}(x)=0.$ The limit $\epsilon\to 0$ concludes since $||f_{K-\epsilon,\pm L}-f_{K,\pm L}||_1\to0$ and thus $||\sigma_{f_{K-\epsilon,\pm L}}(x)-\sigma_{f_{K,\pm L}}(x)||\to 0$. Conversely, if $\sigma_{f_{K,\pm L}}(x)=0$ $\sigma_{\tau_{\pm L}f_K}(x)=0$ where $f_{K,\pm L}(x)={\tau_{\pm L}f_K}(x)=f_K(x\mp L)$ is the translation of ${f_K}$ by $\pm L$ which is non-zero in $]-K\pm L,K\pm L[$ and those sets cover $[-K,K]^c$ thus any point outside  $[-K,K]^c$ is not in the spectrum of $x$.
\end{proof}

We will finally need another standard fact of spectral theory  to compute the form $\E_{\alpha}(x,y)=\langle \Delta^{\alpha}(x\xi_\varphi),(y\xi_\varphi)\rangle$
in the case (we could but won't restrict to) $\alpha=1/3,2/3,1.$ We will use crucially that $M\subset D(\Delta^{1/2})$ in the case above. We give explicit bound for integral formulas in order to get first order continuous logic axioms characterizing those forms.
\begin{lemma}\label{EquationForms} 
For any positive closed densely defined operator as $\Delta$, and $0<\alpha<1, \epsilon>0$, we have \[(\Delta+\epsilon)^{-\alpha}=\frac{\sin(\alpha \pi)}{\pi}\int_0^\infty s^{-\alpha}(\Delta+s+\epsilon)^{-1}ds.\] If moreover $M\subset D(\Delta^{1/2}),$ and if we write $G_s=2e^{s/2}\Delta^{1/2}(\Delta+e^s)^{-1} $, then for $\alpha\in]0,1/2[$ and  any $x,y\in M$, $\beta\in [0,1-\alpha[$ (or if $\alpha+\beta=1$, for $x,y\in D(\Delta^{3/4})$) we have
\begin{equation}\label{formeq}\E_{\alpha+\beta}(x,y)=\frac{\cos(\alpha \pi)}{2\pi}\int_{-\infty}^\infty dt e^{\alpha t}\E_{\beta}(G_{t}(x),y).\end{equation}
Especially, for $\hat{\beta}=\min( 1/2,\beta), \delta=\min(1/2,1-\beta)$, if $\alpha+\beta<1$ we have : \begin{align*}&\left|\E_{\alpha+\beta}(x,y)-\frac{1}{n^2}\frac{\cos(\alpha \pi)}{2\pi}\sum_{k=-n^3}^{n^3-1}e^{\alpha k/n^2}\E_{\beta}(G_{k/n^2}(x),y)\right|\leq (\frac{2e^{-n(\alpha +\hat{\beta})}}{\alpha +\hat{\beta}} +\frac{2e^{-n(\delta-\alpha)}}{\delta-\alpha}\\&+|(1-e^{1/n^2})|\frac{4(1-e^{-n(\alpha +\hat{\beta})})}{\alpha +\hat{\beta}} +|(1-e^{1/n^2})|\frac{2(1+e^{(\delta-\alpha)/n^2})(1-e^{-n(\delta-\alpha)})}{\delta-\alpha})\frac{||x||_\varphi^\#||y||_\varphi^\#}{2\pi} \end{align*}
and for $\alpha+\beta=1$, we have
\begin{align*}&\left|\E_{1}(x,y)-\frac{1}{n^2}\frac{\cos(\alpha \pi)}{2\pi}\sum_{k=-n^3}^{n^3-1}e^{\alpha k/n^2}\E_{\beta}(G_{k/n^2}(x),y)\right|\leq (\frac{2e^{-n(\alpha +1/2)}}{\alpha +1/2} +\frac{2e^{-n(1/2-\alpha)}}{1/2-\alpha}\\&+|(1-e^{1/n^2})|\frac{4(1-e^{-n(\alpha +1/2)})}{\alpha +1/2} +|(1-e^{1/n^2})|\frac{2(1+e^{(1/2-\alpha)/n^2})(1-e^{-n(1/2-\alpha)})}{1/2-\alpha})\\&\ \ \ \ \ \ \ \ \ \ \ \ \ \times \frac{\sqrt{||x||_\varphi^2+||\Delta^{3/4}(x\xi_\varphi)||_\varphi^2}\sqrt{||y||_\varphi^2+||\Delta^{3/4}(y\xi_\varphi)||_\varphi^2}}{2\pi}. \end{align*}

\end{lemma}

\begin{proof}
The first result is well-known, see e.g. \cite[Rmk V.3.50]{Kato},  the integral is absolutely converging in bounded operators since $||s^{-\alpha}(\Delta+s+\epsilon)^{-1}||\leq s^{-\alpha}(s+\epsilon)^{-1}$. 
What remains to prove is a definition as form and is standard e.g. similar to the arguments in Kato's book. We first show the integrals are absolutely converging. Note that $||(\Delta+e^s)^{-1}||\leq e^{-s},  ||\Delta(\Delta+e^s)^{-1}||=||1-e^s(\Delta+e^s)^{-1}||\leq 1$, (and for $\beta\leq 1/2$ $||\Delta^{\beta}(y\xi_\varphi)||^2\leq ||\Delta^{1/2}(y\xi_\varphi)||^2+||y\xi_\varphi||^2$ since $\Delta^{2\beta}\leq 1+\Delta$)
. Note also that by definition $|e^{\alpha t}\E_{\beta}(G_{t}(x),y)|=2e^{(\alpha+1/2)t}|\langle \Delta^{\beta+1/2}(\Delta+e^t)^{-1}(x\xi_\varphi),(y\xi_\varphi)\rangle|$  so that one gets
\[|e^{\alpha t}\E_{\beta}(G_{t}(x),y)|\leq 2e^{(\alpha+1/2)t}e^{(\beta-3/2)t}\|\Delta^{1/2}(x\xi_\varphi)\|\|\Delta^{1/2}(y\xi_\varphi)\|\ \ \ \mathrm{if} \ 1-\alpha>\beta>1/2,t>0,\]
\[|e^{\alpha t}\E_{\beta}(G_{t}(x),y)|\leq 2e^{(\alpha-1/2)t}\|\Delta^{1/2}(x\xi_\varphi)\|\|\Delta^{\beta}(y\xi_\varphi)\|\ \ \ \mathrm{if} \ \beta\leq1/2,t>0,\]
\[|e^{\alpha t}\E_{\beta}(G_{t}(x),y)|\leq 2e^{(\alpha+1/2)t}\|\Delta^{\beta-1/2}(x\xi_\varphi)\|\|y\xi_\varphi\|\ \ \ \mathrm{if} \  \beta>1/2,t<0,\]
\[|e^{\alpha t}\E_{\beta}(G_{t}(x),y)|\leq 2e^{(\alpha+1/2)t}e^{(\beta-1/2)t}\|(x\xi_\varphi)\|\|(y\xi_\varphi)\|\ \ \ \mathrm{if} \ \beta\leq1/2,t<0.\]
This gives the expected integrability in all cases but the case $\alpha+\beta=1$ in which case we use a bound similar to the second bound (since for $\beta\leq 1$, $||\Delta^{\beta/2+1/4}(y\xi_\varphi)||^2\leq ||\Delta^{3/4}(y\xi_\varphi)||^2+||y\xi_\varphi||^2$ using $\Delta^{\beta+1/2}\leq 1+\Delta^{3/2}$):
\[|e^{\alpha t}\E_{\beta}(G_{t}(x),y)|\leq 2e^{(\alpha-1/2)t}\|\Delta^{\beta/2+1/4}(x\xi_\varphi)\|\|\Delta^{\beta/2+1/4}(y\xi_\varphi)\|\ \ \ \mathrm{if} \ 1-\alpha=\beta>1/2,t>0\]

Applying our first formula to $\gamma=1/2-\alpha$ instead of $\alpha$ and with a change of variable $s=e^u$, one gets 
\[(\Delta+\epsilon)^{\alpha-1/2}=\frac{\cos(\alpha \pi)}{\pi}\int_{-\infty}^\infty du e^{(\alpha+1/2) u}(\Delta+e^u+\epsilon)^{-1}\] so that, first on the domain of $\Delta^{1/4+\beta/2}$ and then on $M$ by density and with $h(u,\epsilon)=\ln(e^u+\epsilon)$, we obtain:
\[\E_{\beta+1/2}((\Delta+\epsilon)^{\alpha-1/2}x,y)=\frac{\cos(\alpha \pi)}{\pi}\int_{-\infty}^\infty du e^{(\alpha+1/2) u}e^{-h(u,\epsilon)/2}\E_{\beta}(G_{h(u,\epsilon)}(x\xi_\varphi),(y\xi_\varphi))\]
By dominated convergence theorem with bounds similar to those above (and for $a,\epsilon\geq 0$, $(e^{-t}\epsilon+1)^{-a}\leq 1$), we obtain the result at the limit $\epsilon\to 0.$

By the resolvent equation we have \begin{align*}e^{\alpha t}G_{t}(x)-e^{\alpha k/n^2}G_{k/n^2}(x)&=2e^{(1/2+\alpha)t}\Delta^{1/2}(\Delta+e^{t})^{-1}(1-e^{t-k/n^2})e^{k/n^2}(\Delta+e^{k/n^2})^{-1}(x)\\&+2(e^{(1/2+\alpha)(t-k/n^2)}-1)e^{(1/2+\alpha)k/n^2}\Delta^{1/2}(\Delta+e^{k/n^2})^{-1}(x)\end{align*} so that for $t\geq k/n^2$ (note $\alpha+1/2\leq 1)$
\begin{align*}|e^{\alpha t}&\E_{\beta}(G_{t}(x),y)-e^{\alpha k/n^2}\E_{\beta}(G_{k/n^2}(x),y)|\\&\leq |(1-e^{(t-k/n^2)})|\ \left(|e^{\alpha t}\E_{\beta}(G_{t}(e^{k/n^2}(\Delta+e^{k/n^2})^{-1}(x),y)|+|e^{\alpha k/n^2}\E_{\beta}(G_{k/n^2}(x),y)|\right)\end{align*}
The two 
 inequalities follow by standard Riemann integration. 



\end{proof}

 \subsection{Axiomatization}
Recall we have one sort $U$ with domains of quantification $D_n$ for the operator norm ball of radius $n$ of $M$. 
The language will be composed of 
\begin{itemize}
\item[$\bullet$]The constant 0 which will be in $D_1$. 
\item[$\bullet$] For every $\lambda\in \C$ a unary function symbol also denoted $\lambda$ to be interpreted as scalar
multiplication. For simplicity we shall write $\lambda x$ instead of $\lambda (x)$.
\item[$\bullet$] A unary function symbol $*$ for involution on $U$, leaving stable all domains.
\item[$\bullet$] A binary function symbol $+:D_n\times D_m\to D_{n+m}$. and for $K,L\in \N^*$ $m_{K,L}:D_{n}\times D_{m}\to D_{nm}$ (interpreted as $F_K^\varphi(.)F_L^\varphi(.)$ with modulus of continuity as obtained in lemma \ref{ContinuityProduct}).
\item[$\bullet$]The constant 1 in $D_1$.
\item[$\bullet$] Two unary relation symbols $\varphi_r$ and $\varphi_i$ for the real and imaginary parts of the state $\varphi$, on $U$. We will often just write $\varphi$ and assume that the expression can be decomposed
into the real and imaginary parts.
\item[$\bullet$] For each $t\in \Q$, unary function symbols $\sigma_t:D_n\to D_n$ (for the modular group), $G_t:D_n\to D_n$ (for $G_t^\varphi$) and for $(m,l)\in\Q^*,m> 0$ $F_{m,l}:D_n\to  D_n, F_{N,0}=F_N$ (for Fejer's map $F_{m,l}^\varphi$).$H_{K}=(K+1)F_{K+1}-KF_{K}.$
\item[$\bullet$] A function symbol $\tau_{p,\lambda,N}$ (meaning $p(\sum_{i=1}^n\lambda_iF_{N_i}(x))$) for every $*$-polynomial in one variable $p$, any $N=(N_1,...,N_n)\in (\N^*)^n$, any $\lambda=(\lambda_1,...\lambda_n)\in(\Q\cap[0,1])^n \sum\lambda_i=1$. If we write \[m(n,p)=\lceil\sup \{||p(a)||, a\in C, C \ C^*-\mathrm{algebra\ and}\ ||a||\leq n\}\rceil,\] then we require $\tau_{p,\lambda,N}:D_{n}\to D_{m(n,p)}$ and we want it to have same modulus of continuity as $p^.(\sum_{i=1}^n\lambda_iF_{N_i}(.))$ (we will use this as notation for products obtained with maps $m_{(N,K)}$ in a way we will explain bellow).
\item[$\bullet$] For each $\alpha\in\Q\cap [0,1[,$ binary relation symbols $\E_{\alpha,N,M}$ on $D_{n}$ (as for $\varphi$ above, formally decomposed into imaginary and real part and meaning $\E_{\alpha}(F_N^\varphi(.),F_M^\varphi(.))$.
\end{itemize}
 $H_{K}$ is inspired by the so-called \emph{De la vallée poussin Kernel}  $h_{K}=(K+1)f_{K+1}-Kf_{K}$ with Fourier transform \[\widehat{h_{K}}(t)
=\max(0,\min(K+1-|t|,1)).\]
which is equal to $1$ on $[-K,K]$ and with support in $[-K-1,K+1]$. Thus $H_{K}^\varphi$ is identity on $M(\sigma^\varphi,[-K+\epsilon,K-\epsilon]),\epsilon>0$ and we will use that to encode rough analogues of relations of spectrum and product and also associativity. We will also use as a shorter notation for what is supposed to be $H_{K+1}^\varphi(a).H_{L+1}^\varphi(b)$ namely:
\begin{align*}M_{(K,L)}(a,b)&=(K+1)(L+1)m_{(K+1,L+1)}(a,b)+(K+2)(L+2)m_{(K+2,L+2)}(a,b)\\&-(K+1)(L+2)m_{(K+1,L+2)}(a,b)-(K+2)(L+1)m_{(K+2,L+1)}(a,b).\end{align*}
With this notation we can define for a monomial: $p=x^{\epsilon_1}...x^{\epsilon_k}, \epsilon_i\in\{1,*\}$ the expression used above for $\lambda_i\in[0,1]$ and then extend by linearity to a more general polynomial \begin{align*}p^.(&\sum_{i=1}^n\lambda_iF_{N_i}(x))=\sum_{i_1,...,i_k=1}^n \lambda_{i_1}...\lambda_{i_k} \\&M_{N_{i_1},N_{i_{2}}+...+N_{i_k}}(F_{N_{i_{1}}}(x^{\epsilon_{1}}),...M_{N_{i_{k-2}},N_{i_{k-1}}+N_{i_k}}(F_{N_{i_{k-2}}}(x^{\epsilon_{k-2}}),m_{N_{i_{k-1}},N_{i_k}}(x^{\epsilon_{k-1}},x^{\epsilon_{k}})...)\end{align*}
This will correspond in the step 2 of the next theorem to a product in a well-defined associative product.

Let us comment on this language from a model theoretic viewpoint. First, it is easy to see that all the data depends only on $(M,\varphi)$, the pair of a von Neumann algebra and a faithful normal state. Indeed, it is well known that this is the case for $\sigma_t^\varphi$ that is fixed by $t,\varphi$ and all the other data as been defined from the modular group, the product, adjoint and the state in the previous subsection. More precisely, the modular group is the standard  way to encode the unbounded operator $\Delta$ which is defined only at Hilbert space level on $L^2(M,\varphi)$ (as unbounded operator) and thus does not fit well with the model theoretic setting. This unbounded operator can only appear as a sesquilinear form giving the relation $\E_1(x,y)=\langle \Delta^\varphi(x\xi_\varphi),y\xi_\varphi\rangle$. As already mentioned in the introduction, we won't use the KMS condition to check that the automorphism group we put in the theory is indeed the modular group, we will rather use spectral theory that will give explicit formulas that already appeared in the previous section and which are better suited for model theory. For, in step 4 of the proof bellow, we will compute the form corresponding to the modular group given as relation on $M^2$ $Q_t(x,y)=\varphi(\sigma_t(x^*)y)$ which of course determines by density $\sigma_t$ and its Hilbert space extension $\Delta_\sigma^{it}$ uniquely computed from and determining the generator $\Delta_\sigma$ by functional calculus. We will then compute the form $q_\sigma(x,y)=\langle \Delta_\sigma(x\xi_\varphi),y\xi_\varphi\rangle$ (smeared by some $F_K$) and thus deduce that it has the expected value $q_\sigma=\E_1$ so that the (exponentiated) generator $\Delta_\sigma=\Delta^\varphi$ and thus the modular group that encodes it in the theory will also be equal to its expected value $\sigma_t=\sigma_t^\varphi.$ We will see in subsection 1.4 that all this data (including the modular group) is not strictly speaking necessary in a theory of $\sigma$-finite $W^*$ probability spaces. It enables us to get a universal axiomatization (some of the data as the forms $\mathscr{E}_\alpha$ are even only here to get short and readable enough axioms). In order to remove this main extra piece of data, the modular group, we will need some more technical but standard spectral theory. Since having an explicit universal axiomatization is interesting in its own right, we thus postpone the quest of minimality in the language to this supplementary subsection 1.4.

Our models will thus be models of $\sigma$-finite von Neumann algebras having such a faithful normal state \cite[Prop 3.19]{TakesakiBook} or more precisely of $\sigma$-finite $W^*$ probability spaces since the theory will depend on the state $\varphi$ in a non-trivial way.

As in \cite{FarahII}, we now write down axioms satisfied by any $\sigma$-finite  $W^*$ probability space (either obvious or coming from the preliminary subsection 1.1) :
\begin{enumerate}\item[(1)] $x+(y +z) = (x+y)+z, x+0 = x, x+(-x) = 0$ (where $-x$ is the scalar $-1$ acting
on $x$), $x + y = y + x, \lambda(\mu x) = (\lambda \mu)x, \lambda(x + y) = \lambda x + \lambda y, (\lambda + \mu)x = \lambda x + \mu x$, $1(x)=x$. 
\item[(2)] For $K,L,K_1,K_2,K_3\in \N^*$ $\lambda m_{K,L}(x,y)+ m_{K,L}(x,z)=m_{K,L}(x,\lambda y+z),$
\[m_{K_1,L}(F_{K_2}(x),y)=m_{K_2,L}(F_{K_1}(x),y),\]
\[H_{K+L}(F_K(x))=F_K(x),\]
\[H_{K_1+K_2+L}(m_{K_1,K_2}(x_1,x_2))=
m_{K_1,K_2}(x_1,x_2),\]
\begin{align*}&(K_1+K_2+2)m_{K_1+K_2+2,K_3}(m_{K_1,K_2}(x_1,x_2),x_3)-(K_1+K_2+1)m_{K_1+K_2+1,K_3}(m_{K_1,K_2}(x_1,x_2),x_3)\\&=(K_3+K_2+2)m_{K_1,K_2+2+K_3}(x_1,m_{K_2,K_3}(x_2,x_3))-(K_3+K_2+1)m_{K_1,K_2+1+K_3}(x_1,m_{K_2,K_3}(x_2,x_3)).\end{align*}
\item[(3)] $(x^*)^* = x, (x + y)^* = x^* + y^*, (\lambda x)^* = \overline{\lambda}x^*$ . 
\item[(4)] For $K,L\in \N^*$, $[m_{K,L}(x,y)]^*=m_{(L,K)}(y^*,x^*)$, $F_N(x^*)=[F_N(x)]^*$. $d_U(x,0)=d_U(x^*,0).$
\item[(5)] $d_U(x, y) = d_U(x - y, 0)$,  we write $||x||_\varphi^*=d_U(x,0).$
\item[(6)] For $1$ the constant symbol $1\in D_{1},$ $F_N(1)=1$,  $m_{K,N}(1,x)=F_N(x)=m_{N,K}(x,1)$.  
\item[(7)] $\varphi(x + y) = \varphi(x) + \varphi(y)$. 
\item[(8)] $\varphi(x^*) = \overline{\varphi(x)}$, $\varphi(\lambda x) = \lambda  \varphi(x)$, $\varphi(1) = 1$.
\item[(9)] $\max(0, -\sum_{i,j=1}^n\overline{\lambda_i}\lambda_j\varphi(m_{(K_i,K_j)}(x_i^*,x_j))=0.$
\item[(10)] For every $n,m,K,K_1,...,K_n \in \N^*$,
\begin{align*}&\sup_{a\in D_{n}}\sup_{x_j\in D_m}\max(0,\sum_{i,j=1}^n\overline{\lambda_i}\lambda_j[\varphi(M_{(K+K_i,K+K_j)}([m_{K,K_j}(a,x_i)]^*,m_{K,K_j}(a,x_j))-n^2m_{(K_i,K_j)}(x_i^*,x_j))]) =0\end{align*}
\item[(11)]$\tau_{p,\lambda,N}(x) = p^.(\sum_{i=1}^n\lambda_iF_{N_i}(x))$ for every $*$-polynomial $p$ in one variable $x$, $\lambda_i\in\Q\cap [0,1] with \sum_{i=1}^n\lambda_i=1$, $N=(N_1,...,N_n)\in (\N^*)^n$.
\item[(12)] For $K,K_i,m,l\in (\N^*)$, \[\sup_{x\in D_m}\sup_{y_{i}\in D_{l}}\max(0,\left\|F_{K}(x)+\sum_{i=1}^nm_{K_{i},K}(y_{i},x)\right\|_\varphi^*-3me^K\left\|1+\sum_{i=1}^nF_{K_{i}}(y_{i})\right\|_\varphi^*).\]
\end{enumerate}

These axioms are really similar to \cite{FarahII}. The next-to-last axiom mimic the trick they found to identify the operator norm unit ball in a universal axiomatization (rather than a $\forall\exists$ one).  
We now need to specify the metric to coincide with the Ocneanu ultraproduct and to deal with the modular theory. 
We of course find our inspiration in our previous section and require first the modular group relations (including the continuity obtained in lemma \ref{ContinuityModular}):
\begin{enumerate}
\item[(13)]$\sigma_t(\sigma_s(x))=\sigma_{t+s}(x),\sigma_t(\lambda x+y)=\lambda \sigma_t(x)+\sigma_t(y),\sigma_0(x)=x,\sigma_t(x^*)=(\sigma_t(x))^*,\varphi(\sigma_t(x))=\varphi(x)$.
\item[(14)]$\sigma_t([m_{K,L}(x,y)])=[m_{K,L}(\sigma_{t}(x),\sigma_{t}(y))]$, $\sigma_t(F_N(x))=F_N(\sigma_t(x))$.
\item[(15)]For every $n \in \N$,
\[\sup_{x\in D_{n}}\max(0, d_U(\sigma_t(x),x) -4tn)=0.\]
\end{enumerate}

We also need the relations between $\sigma_t,G_t,F_m$ from lemma \ref{NormG} and lemma \ref{ContinuityModular}:
\begin{enumerate}
\item[(16)]For $s\in \Q,m\in\N$  \[\sup_{x\in D_{m}}\max(0,d_U\left(G_s(x),\frac{1}{n^2}\sum_{k=-n^3}^{n^3-1}\frac{2e^{-is\frac{k}{n^2}}}{e^{\pi \frac{k}{n^2}}+e^{-\pi \frac{k}{n^2}}}\sigma_{k/n^2}(x)\right)- \frac{8e^{-\pi n}m}{\pi}-\frac{4m}{n^2}-\frac{4(\pi+s)m}{n})=0.\] 
\item[(17)] For $\lambda_i\in \C,K_i\in \N^*$ \[4(||\sum_{i=1}^n\lambda_iF_{K_i}(x_i)||_\varphi^*)^2=\sum_{i,j=1}^n\overline{\lambda_i}\lambda_j\varphi(m_{K_i,K_j}(G_0(x_i)^*,G_0(x_j)))+\varphi(m_{K_j,K_i}(G_0(x_j),G_0(x_i)^*)).\]
\item[(18)]For $N\in\Q\cap]0,\infty[, l\in\Q^*,m\in\N^*, K,L\in\N, L\geq 2K$  \begin{align*}&\sup_{x\in D_{m}}\max(0,-\frac{4m}{n^2}-\frac{16m}{\pi Nn^3}-\frac{|l|Nm}{n\pi}-\frac{2mN^2}{\pi n}+\\&d_U\left(F_{N,l}(x),\frac{N}{2\pi n^2}x+\frac{1}{n^2}\sum_{k=-n^3,k\neq0}^{n^3-1}e^{il\frac{k}{n^2}}\frac{1-\cos(N\frac{k}{n^2})}{\pi N \frac{k^2}{n^4}}\sigma_{k/n^2}(x)\right))=0,\end{align*}
\end{enumerate}
We finally have the relations defining our forms from lemma \ref{EquationForms}
\begin{enumerate}
\item[(19)]$\E_{0,K,L}(x,y)=\varphi(m_{K,L}(x^*,y))$ and for $\alpha,\beta\in \Q\cap[0,1[,0<\alpha<1/2,\alpha+\beta<1,\epsilon= \min(1/2-\alpha,1-\beta-\alpha),\delta= \min(\epsilon,\alpha) ,m,K,L,n\in\N^*$  \begin{align*}\sup_{(x,y)\in D_{m}^2}\max(0&,\left|\E_{\alpha+\beta,K,L}(x,y)-\frac{1}{n^2}\frac{\cos(\alpha \pi)}{2\pi}\sum_{k=-n^3}^{n^3-1}e^{\alpha k/n^2}\E_{\beta,K,L}(G_{k/n^2}(x),y)\right|\\&- \frac{4e^{-n\delta}m^2}{\pi\delta}-|(1-e^{1/n^2})|\frac{2(3+e^{\epsilon/n^2})m^2}{\pi\delta})=0 \end{align*}
\item[(20)]For $\alpha,\beta\in \Q\cap[0,1[,0<\alpha<1/2,\mu=1/2-\alpha,m,K,L,n\in\N^*, L\leq K$, for $\alpha+\beta=1$, we have
\begin{align*}\sup_{(x,y)\in D_{m}^2}\max(0&,\left|\varphi(m_{L,K}(y,x^*))-\frac{1}{n^2}\frac{\cos(\alpha \pi)}{2\pi}\sum_{k=-n^3}^{n^3-1}e^{\alpha k/n^2}\E_{\beta,K,L}(G_{k/n^2}(x),y)\right|\\&-\frac{2e^{-n\mu}(1+e^{3K/2})m^2}{\pi\mu}-|(1-e^{1/n^2})|\frac{(3+e^{\mu/n^2})(1+e^{3K/2})m^2}{\pi\mu})=0 \end{align*}
\end{enumerate}

Recall that a (structure) model of a theory will be a metric space 
 with each domain of quantification (for us balls) complete in the metric and with all the symbols having the specified uniform continuity functions. Recall that an axiomatization of a category $\mathcal{C}$ will be as in \cite{FarahII}, a functor $\mathcal{M}$ from $\mathcal{C}$ to models of the theory $T$ such that $\mathcal{M}(A)$ is determined up to isomorphism for any $A\in \mathcal{C}$, for any model $M$ of $T$ there is $A\in \mathcal{C}$ such that $M$ is isomorphic to $\mathcal{M}(A)$ and for every $A,B\in \mathcal{C}$, there is a bijection $Hom_{\mathcal{C}}(A,B)\simeq Hom(\mathcal{M}(A),\mathcal{M}(B)).$
 
The category of ($\sigma$-finite) $W^*$ probability spaces may not have the most expected morphisms. We will consider as morphism only those state preserving $*$-homomorphisms having an image admitting a state preserving conditional expectation. Recall that by $\sigma$-finite $W^*$ probability spaces, we mean a pair $(M,\varphi)$ of a  $\sigma$-finite von Neumann algebra $M$ having a fixed faithful normal state $\varphi$. Since we put in the structure the modular group and we want our morphism to correspond to model-theoretic morphisms and thus commute with the modular group, the image of a morphism will thus be left invariant by the modular group of the target  state, and by a result of Takesaki (cf e.g. \cite[Th IX.4.2]{TakesakiBook}) this is equivalent to the existence of such a conditional expectation. The above category thus seems a nice semantic formulation of our only choice.
 
We are ready to obtain our axiomatization with a compatibility with ultraproducts.

\begin{theorem}\label{OcneanuTh}
The class of $\sigma$-finite $W^*$ probability spaces (with morphisms as described above) is  axiomatizable by the Ocneanu theory $T_{\sigma W^*}$, theory consisting of axioms (1)-(20) above. Moreover, if $(M_n,\varphi_n)$ are  $W^*$ probability spaces of this type, then for any non-principal ultrafilter $\omega$, the model of the Ocneanu ultraproduct of \cite{HaagerupAndo} is given by the model-theoretic ultraproduct :\[\mathcal{M}((M_n,\varphi_n)^\omega)=[\mathcal{M}(M_n,\varphi_n)]^\omega.\] 
\end{theorem}
\begin{remark}
The model theoretic ultraproduct thus gives a construction of a von Neumann algebra for which the ultraproduct of modular groups is the modular group of the ultraproduct state. This is the same result as in \cite[Th 4.1]{HaagerupAndo} without any use of the Groh-Raynaud ultraproduct. However, for simplicity, to check that this von Neumann algebra structure coincides with the Ocneanu ultraproduct of \cite{HaagerupAndo}, we will use all their results even those using the Groh-Raynaud ultraproduct. Our construction does not really provide a new proof that the standard Ocneanu construction is a von Neumann algebra with the right modular theory. But it provides an alternative root to the same object and the proof of its main properties without using any non-$\sigma$ finite von Neumann algebra. Note also that of course (1)-(20) means all the axioms from (1) to (20). We would write (1),(20) for singling out the two axioms (1) and (20).
\end{remark}
 
\begin{proof}We already noticed that any $(M,\varphi)$ gives a model of $T_{\sigma W^*}$ using the lemmas of the previous section. Obviously, in using Takesaki's theorem \cite[Th IX.4.2]{TakesakiBook}, a (state preserving) *-homomorphism of von Neumann algebras having a state preserving conditional expectation on its range gives a map between the corresponding models preserving the structure and a fortiori vice versa.
Let us say a supplementary word on that for the reader's convenience. By injectivity of a (state preserving) *-homomorphism of von Neumann algebras, saying that such a morphism preserve the modular group boils down to the easy remark that the modular group of a subalgebra $N$ of $(M,\varphi)$ computed with the restricted state is the restricted modular group, as soon as this is possible, namely if the subalgebra is left invariant by the huge modular group, which is equivalent by  \cite[Th IX.4.2]{TakesakiBook} to the existence of  the stated conditional expectation. For, the restriction of the modular group in this case is an automorphism group and satisfies the KMS modular condition with respect to the restricted state (thus the uniqueness in e.g. \cite[Th VIII.1.2]{TakesakiBook} concludes). 
 We explained in the previous section how the extra-data is computed from the modular group and thus commutation of the morphism and this data is deduced from the one with the modular group. For the converse, the only piece of data that a structure preserving morphism does not preserve by definition is the product, since it only preserves smeared products, but the limiting description of the product obtained bellow from smeared products gives the homomorphism property for the product too. Since the structure preserving morphism preserves the modular group, the converse in \cite[Th IX.4.2]{TakesakiBook} gives existence of a state preserving conditional expectation on its image. We thus  checked from our choice of category and  as expected for a nice axiomatization the bijection $Hom_{\mathcal{C}}(A,B)\simeq Hom(\mathcal{M}(A),\mathcal{M}(B)).$


Assume $M$ satisfies $T_{\sigma W^*}$. We want to see that in the sort $U$, the set $M$ gives a $W^*$ probability space, having the expected modular theory and balls. This is of course the most technical part and we divide this into several steps.

\begin{step}
First properties of modular theoretic maps.
\end{step}

 From (13)-(15) $\sigma_t$ is a continuous one parameter group of linear state preserving maps on $M$.  We can extend by continuity for $d$, $\sigma_t$ to $t\in \R$.
 
 From (16) one then deduces first that $G_s(x)$ is the limit (in metric $d$) of the Riemann sums written, and by the bounds in lemma \ref{ContinuityModular} we thus deduce for any $x\in D_n(M)$ \[G_s(x)=\int_{\R}dt \frac{2e^{-ist}}{e^{\pi t}+e^{-\pi t}}\sigma_t(x)\in D_n(M).\]
 By dominated convergence theorem, one also deduces $G_s$ is strongly continuous in $s$ for $d$ and extends with the same formula to $G_s,s\in\R.$
 
We can now repeat the reasoning with (18) to get that for $N\in \Q^*, N>0,l\in \Q$ for $x\in D_n(M)$:
\[F_{N,l}(x)=\int_{\R}dt f_{N,l}(t)\sigma_t(x)\in D_n(M).\]
We will first use only the case $F_N=F_{N,0},N\in\N^*$. Let us show that \[A:=\mathrm{Vect}\{F_N(x), x\in M, N\in \N^*\}\] is dense in $M$ for $d$ and that even $A\cap D_n(M)$ is dense in $D_n(M).$ 

Note that $x=\int_{\R}dt f_N(t)x$ so that $F_N(x)-x=\int_{\R}dt f_N(t)(\sigma_t(x)-x)$ and thus \[d(F_N(x),x)\leq \int_{\R}dt f_N(t)d(\sigma_t(x),x).\]
Since $f_m$ is an approximation of a Dirac mass $\delta_0$ and  $d(\sigma_t(x),x)$ is continuous with value $0$ at $0$, one deduces 
 $d(F_N(x),x)\to_{N\to \infty} 0$ proving the expected density. Note also that $\sigma_s(F_N(x))=F_N(\sigma_s(x)), G_s(F_N(x))=F_N(G_s(x)), x\in M$ and some other properties written for $F_n$ are actually consequences of the integral formula and those for $\sigma_t$.
 
\begin{step}
Building the algebra structure on $A$ and its faithful representation.
\end{step}

We first build the algebra structure on $A=\mathrm{Vect}\{F_N(x), x\in M, N\in \N^*\},$ which is already stable by adjoint and modular group by (4),(14).

We want to extend by bilinearity the product \[[F_K(x)].[F_L(y)]=m_{K,L}(x,y)=H_{K+L+1}(m_{K,L}(x,y))\in A\] by (2). If $\mathcal{A}$ is the abstract direct sum of $\N^*$ copies of $M$ with $x_n$ the  $n$-th copy of $x\in M$ there is a map of $f:\mathcal{A}\to A$ sending $x_N\mapsto F_N(x)$. Obviously the product $\mathcal{A}\times \mathcal{A}\to A$ defined by $x_K.y_L=m_{K,L}(x,y)$ is well defined and it suffices to see it vanishes on $Ker(f)\times \mathcal{A}+ \mathcal{A}\times Ker(f)$ to see that this induces a well-defined (expected) product on $A$. Axiom (12) (and a symmetric variant insured using adjoints and (4)) exactly guaranties this, since $f(a)=0$ iff $||f(a)||_\varphi^*=0.$

We can now see that $A$ is in this way a $*$-algebra.

Indeed the last part of the axiom (2) can be rewritten \[[H_{K_1+K_2+1}(m_{K_1,K_2}(x_1,x_2))].F_{K_3}(x_3)=F_{K_1}(x_1).[H_{K_3+K_2+1}(m_{K_2,K_3}(x_2,x_3))]\]
But using the definition and the second relation in (2) this is nothing but the associativity relation $([F_{K_1}(x_1)].[F_{K_2}(x_2)])[F_{K_3}(x_3)]=[F_{K_1}(x_1)].([F_{K_2}(x_2)].[F_{K_3}(x_3)]).$ The $*$ algebra relation $[F_{K_1}(x_1)].[F_{K_2}(x_2)])^*= [F_{K_2}(x_2)]^*.[F_{K_1}(x_1)]^*$ is obtained from the first part of (4). Note that similarly, the expression appearing in (10) can be interpreted using the product as :\[M_{(K+K_i,K+K_j)}([m_{K,K_j}(a,x_i)]^*,m_{K,K_j}(a,x_j))=([F_{K}(a)].[F_{K_i}(x_i)])^*.([F_{K}(a)].[F_{K_j}(x_j)]).\]

Moreover, from (7)-(9) $A$ is a complex
pre-Hilbert space with inner product given by $\langle y,x\rangle=\varphi(y^*x)$. Left multiplication by $a \in A$ is a
linear operator on $A$ and axiom (10) guarantees that for $x\in D_{n}(M)$, $F_N(x)$ is bounded of norm less than $n$. The operation $*$ is the
adjoint because for all $x$ and $y$ we have \[\langle ax, y\rangle = \varphi((ax)^*y) = \varphi(x^*a^*y) =  \langle x, a^*y\rangle.\] Thus $A$
is represented (by left multiplication) as a *-algebra of Hilbert space (bounded) operators (on the completion $L^2(A,\varphi)$ with cyclic vector $\xi_\varphi$). 

We now want to check that the representation is faithful in showing that for any $x\in A$:
\begin{equation}\label{MetricStrongInequality}\tag{17b}4d_U(x,0)^2\leq ||x\xi_\varphi||^2+||x^*\xi_\varphi||^2.\end{equation}
We of course use (17). We first note that using the kernel for $G_0$ is a probability and Cauchy-Schwartz inequality : \begin{align*}&\sum_{i,j=1}^n\overline{\lambda_i}\lambda_j\varphi(m_{K_i,K_j}(G_0((x_i)^*),G_0(x_j)))\\&=\int_{\R}dt \frac{2}{e^{\pi t}+e^{-\pi t}}\int_{\R}ds \frac{2}{e^{\pi s}+e^{-\pi s}}\sum_{i,j=1}^n\overline{\lambda_i}\lambda_j\varphi(m_{K_i,K_j}(\sigma_t((x_i)^*),\sigma_s(x_j)))
\\& \leq \sup_{s,t\in \R}|\sum_{i,j=1}^n\overline{\lambda_i}\lambda_j\varphi(m_{K_i,K_j}(\sigma_t((x_i)^*),\sigma_s(x_j)))|\\&\leq \sup_{s,t\in \R}|\sum_{i,j=1}^n\overline{\lambda_i}\lambda_j\varphi(m_{K_i,K_j}(\sigma_t((x_i)^*),\sigma_t(x_j)))|^{1/2}|\sum_{i,j=1}^n\overline{\lambda_i}\lambda_j\varphi(m_{K_i,K_j}(\sigma_s((x_i)^*),\sigma_s(x_j)))|^{1/2}\\&=|\sum_{i,j=1}^n\overline{\lambda_i}\lambda_j\varphi(m_{K_i,K_j}((x_i)^*,x_j))|=\|[\sum_{i=1}^n\lambda_iF_{K_i}(x_i)]\xi_\varphi\|^2\end{align*}
where we finally used (13)-(14). Thus (17b) now follows immediately from (17). Faithfulness also follows, since if the operator $a$ is $0$, we have $||a\xi_\varphi||=||a^*\xi_\varphi||=0$ and thus $d(a,0)=0$, i.e $a=0$ in $A\subset M$ since $M$ is a metric space.  

\begin{step}
Obtaining the von Neumann algebra structure on $M$.
\end{step}

We have now to use the modular theory to be able to represent $M$ (and not only the dense $A$) on $L^2(A,\varphi)$.
Note that (12) gives a Lipschitz bound $||a.F_N(x)||_\varphi^*\leq C_{N,||x||}||a||_\varphi^*$ valid for any $a\in A.$ Thus, by the uniform continuity bound for $\varphi$ in the setting \[|\langle F_K(y)\xi_\varphi,a.F_N(x)\xi_\varphi\rangle|\leq \sqrt{2}||F_K(y)^*.(a.F_N(x))||\leq \sqrt{2}C_{K,||y||}C_{N,||x||}||a||_\varphi^*.\]

Thus, using boundedness of the action of $A$ and density of $A\subset L^2(A,\varphi)$, for any $x\in A$, $a\mapsto ax\xi_\varphi$ is uniformly continuous on $A\cap D_n(M)$ with $d$ to the weak topology (uniform structure). Thus, since $L^2(A,\varphi)$ is complete and Hausdorff for the weak topology and we checked in the first step that $A\cap D_n(M)\subset D_n(M)$, it follows (e.g. \cite[\S 5.4.(4)]{Kothe}) that the map extends to a (uniformly) continuous map $D_n(M)\to L^2(A,\varphi).$ This gives an action of $M$ on $L^2(A,\varphi)$ by bounded operators, $D_n(M)$ acting by operators of norm less than $n$ (this is used to extend the action from $A\xi_\varphi$ to $L^2(A,\varphi)$).
 
 Note that $(F_N(x)-x)\xi_\varphi,(F_N(x)^*-x^*)\xi_\varphi\to 0$  weakly (from the metric to weak continuity), thus passing to convex combination, a usual consequence of Hahn-Banach theorem says a net of convex combinations of $(F_N(x)\xi_\varphi,F_N(x)^*\xi_\varphi,F_N(x))$ converges to $(x\xi_\varphi,x^*\xi_\varphi,x)$ in $L^2(A,\varphi)^2\times M$. Applying the above inequality \eqref{MetricStrongInequality} to this net, one gets for any $x\in M$: \[4d_U(x,0)^2\leq ||x\xi_\varphi||^2+||x^*\xi_\varphi||^2.\]
 We also call $U_n(x)$ the net above of convex combination of $F_N(x)$ that we can even assume to converge in the $*$-strong operator topology to $x$ (replacing $\xi_\varphi$ by $x\xi_\varphi, x\in A$).
 
 The map $i:M\to B(L^2(A,\varphi))$ is thus continuous on balls with the metric $d$ to the weak operator topology. From the inequality above this gives a faithful action. Let us see that $i(M)=A''$ so that the image $i(M)$ will be a von Neumann algebra (isomorphic to $M$ and will especially induce a product extending the one of $A$). The density and metric to weak operator topology continuity gives that the image $i(M)$ is included in the weak closure $A''$. Conversely, by Kaplansky density theorem, take a net $i(a_n)\to a\in A''$ $||a||\leq m$, $a_n\in D_m(M)\cap A$ with convergence in the strong-* operator topology. From the inequality above, $d_U(a_n,a_m)^2\leq ||i(a_n-a_m)\xi_\varphi||^2+||i(a_n-a_m)^*\xi_\varphi||^2$ so that $a_n$ is a Cauchy net for the metric $d$ in $D_m(M)$, thus by completeness of the balls of the model, it converges to $A\in D_m(M)$ and it remains to check $i(A)=a$. But by the metric to weak operator topology convergence $i(a_n)\to i(A)$ in the weak operator topology, and since 
 $i(a_n)\to a$ in this separated topology, this concludes.
 
We thus have a von Neumann algebra structure on $M$ with $j:A\to M$ a state preserving $*$-homomorphism with dense range. Since for $x\in D_m(M)$, there is a net $U_n(x)\in D_m(M)\cap A$ (of convex combinations of $F_N(x)$ as above)  converging $*$-strongly to $x$ so that any commutative polynomial $p(U_N(x))$ has the right norm by (11) and tends to $p(x)$, we deduce that it has the right operator norm. Thus arguing as in \cite{FarahII}, one gets $D_m(M)$ is operator norm $m$ ball of $M$.
 
 \begin{step}
Identifying the modular theory $\sigma_t=\sigma_t^\varphi$.
\end{step}

We now want to identify $\sigma_t=\sigma_t^\varphi.$
First note that $\sigma_t$ is an automorphism of $M$ by extension of the property for $A$ and that $t\mapsto \sigma_t(x)$ is weakly continuous for $x\in M$ from the continuity in metric, and since it is bounded, it is also $\sigma$-weakly continuous (see e.g. \cite[lemma II.2.5]{TakesakiBook}) and thus $(M,\R,\sigma)$ defines a covariant system by \cite[proposition X.1.2]{TakesakiBook}. 

Note that $\sigma_t$ induces a one parameter group of isometries on $L^2(A,\varphi)$, which is strongly continuous from the continuity on $A$ extended by density, we will write $\Delta^{it}$ this semigroup, with $\Delta$ an unbounded non-singular densely defined operator on $L^2(A,\varphi).$ Then $\Delta^{it}(x\xi_\varphi)=\sigma_t(x)\xi_\varphi$ for $x\in A$. 

Note $\widehat{h_{R+1}}$ is continuous compactly supported in $[-R-2,R+2]$ equal to $1$ on $[-R-1,R+1]$ and we have $h_{R+1}\in L^1(\R).$ If $x\in M(\sigma,[-R,R])$, we know that $\sigma_{h_{R+1}}(x)=x$ by \cite[lemma XI.1.3]{TakesakiBook} and $x\xi_\varphi=\int_{-\infty}^\infty dt h_R(t)\Delta^{it}(x\xi_\varphi)=[\widehat{h_R}(\ln(\Delta))](x\xi_\varphi)$ by spectral calculus (since $\int_{-K}^K dt h_R(t)e^{itx}$ converges uniformly on $\R$ to $\widehat{h_R}$ and $\int_{-K}^K dt h_R(t)\Delta^{it}$ (as uniform limit of Riemann sums) thus converges in norm when $K\to \infty$ to $[\widehat{h_R}(\ln(\Delta))]$ by \cite[p326 and Th 13.30]{Rudin} since the unbounded selfadjoint case is developed from the unitary case). Thus by functional calculus again $||\Delta\alpha(\alpha+\Delta)^{-1}(x\xi_\varphi)||=||\Delta\alpha(\alpha+\Delta)^{-1}\widehat{h_{R+1}}(\ln(\Delta))(x\xi_\varphi)||\leq e^{(R+2)}||(x\xi_\varphi)||$ and since this is independent of $\alpha$ one gets $x\xi_\varphi\in D(\Delta)$ and at the limit $\alpha\to \infty$ : $||\Delta(x\xi_\varphi)||\leq e^{(R+2)}||(x\xi_\varphi)||$ for any $x\in M(\sigma,[-R,R]).$

Note that  the argument above on spectral algebras for $\sigma$ suffices to get $A\xi_\varphi\subset D(\Delta).$ Indeed, it suffices to use the equation $\sigma_{h_{R+1}}(x)=H_{R+1}(x)=x$ from axiom $(2)$ for $x=F_R(y).$

As in lemma \ref{NormG}, one deduces $G_s(x)\xi_\varphi=e^{s/2}2\Delta^{1/2}(\Delta+e^s)^{-1}(x\xi_\varphi).$ 
Let us write for $\beta\in ]0,1[$ $\E_{\beta}(x,y)=\langle\Delta^{\beta}x\xi_\varphi,y\xi_\varphi \rangle$ and note that $\E_{0}(F_{K}(x),F_L(y))=\E_{0,K,L}(x,y).$
From (19) and some arguments in lemma \ref{EquationForms}, one thus deduces for $x,y\in M$ \[\E_{\alpha+\beta,K,L}(x,y)=\frac{\cos(\alpha \pi)}{2\pi}\int_{-\infty}^\infty dt e^{\alpha t}\E_{\beta,K,L}(G_{t}(x),y)
.\]
 
For $\beta=0$, we have  $\E_{0,K,L}(G_{s}(x),y)=\langle e^{s/2}2\Delta^{1/2}(\Delta+e^s)^{-1}(F_K(x)\xi_\varphi),F_L(y)\xi_\varphi\rangle$ and from the formula above and from the same formula for $\Delta$, one deduces, $\E_{\alpha,K,L}(x,y)=\langle \Delta^{\alpha}(F_K(x)\xi_\varphi),(F_L(y)\xi_\varphi)\rangle, \alpha<1/2.$ Replacing $\beta=0$ by any $\beta<1/2$ one thus gets for any $\alpha+\beta<1$, $\E_{\alpha+\beta,K,L}(x,y)=\langle \Delta^{\alpha+\beta}(F_K(x)\xi_\varphi),(F_L(y)\xi_\varphi)\rangle$.

Applying now similarly (20) one gets for $\alpha+\beta=1,x,y\in A$
\[\varphi(F_L(y)[F_K(x)]^*)=\frac{\cos(\alpha \pi)}{2\pi}\int_{-\infty}^\infty dt e^{\alpha t}\E_{\beta}(G_{t}^\varphi(x),y)=\langle \Delta(F_K(x)\xi_\varphi),(F_L(y)\xi_\varphi)\rangle.\]

Thus by sesquilinearity, we even have for any $x,y\in A$ :
\[\varphi(yx^*)=\langle \Delta(x\xi_\varphi),(y\xi_\varphi)\rangle.\]
Since $\Delta^{1/2}$ is a closed bounded operator, taking $x_n\to x,$ (*-strongly) $x_n\in A,x\in M$  a bounded net one deduces from the formula above that $||\Delta^{1/2}(x_n\xi_\varphi)||\leq ||x_n||$ is bounded thus $(\Delta^{1/2}(x_n\xi_\varphi))$ has a weakly converging subnet and a normwise converging convex combination so that by closability, $x\xi_\varphi\in D(\Delta^{1/2})\supset M\xi_\varphi$ and the equality above is thus extended to $M$ in the form \[\varphi(yx^*)=\langle y^*\xi_\varphi,x^*\xi_\varphi\rangle=\langle \Delta^{1/2}(x\xi_\varphi), \Delta^{1/2}(y\xi_\varphi)\rangle, \ \ x,y\in M.\]

But using again the equality in (17) now extended  to $M$ (by strong-* density and continuity on $L^2(M,\varphi)$ of $\Delta^{1/2}(1+\Delta)^{-1}=\Delta^{1/2}(1+\Delta)^{-1/2}(1+\Delta)^{-1/2}$) \[4d_U(x,0)^2=||G_0(x)\xi_\varphi||^2+||G_0(x)^*\xi_\varphi||^2=||\Delta^{1/2}(1+\Delta)^{-1}(x\xi_\varphi)||^2+||\Delta(1+\Delta)^{-1}(x\xi_\varphi)||^2\leq 2||x\xi_\varphi||^2.\]
Thus this implies $\varphi$ faithful on $M$ and thus has a modular group and the relation above implies $\Delta=\Delta_\varphi$ and thus $\sigma_t=\sigma_t^\varphi.$ All the previous computations shows that the data of our model coincide to the one produced from  $(M,\varphi)$ but maybe $m_{K,L}$ for which we give an extra argument :

$F_K(x).F_L(y)$ computed in $M$ is obtained by taking the limit of \[\sum_{i,j}\lambda_i \mu_j m_{O_i,P_j}(F_K(x),F_L(y))=\sum_{i}\lambda_i  m_{O_i,L}(F_K(x),\sum_{j}\mu_j F_{P_j}(y))\] from the definition of step 3 for some convex combinations with $O_i,P_j$ large enough and from an equality in (2). Now, from the inequality in (12), one gets :\[||\sum_{i}\lambda_i  m_{O_i,L}(F_K(x),\sum_{j}\mu_j F_{P_j}(y))-m_{K,L}(x,\sum_{j}\mu_j F_{P_j}(y))||_\varphi^*\leq e^L||y|| ||\sum_{i}\lambda_i F_{O_i}(F_K(x))-F_K(x)||_\varphi^*\]
and this net has been chosen so that this tends to $0$. Similarly $m_{K,L}(x,\sum_{j}\mu_j F_{P_j}(y))\to m_{K,L}(x,y)$ and one thus obtains $F_K(x).F_L(y)=m_{K,L}(x,y)$ as expected.

 We have thus proved our first axiomatization result.

 \begin{step}
Identifying the ultraproduct.
\end{step}

First, with our original formula for $||x||_\varphi^*$ equivalent to the one in (17) from lemma \ref{NormG} a bounded sequence $(x_n)$ with $d(x_n,0)\to_{n\to \omega} 0$ can be decomposed in $x_n=y_n+z_n$ with $||y_n||_\varphi\to_{n\to \omega} 0,||z_n^*||_\varphi\to_{n\to \omega} 0$, i.e. with the notation of \cite{HaagerupAndo} $(y_n)\in\mathcal{L}_\omega, (z_n)\in\mathcal{L}_\omega^*$. But by their proposition 3.14, $(M_n,\varphi_n)^\omega=\ell^\infty(\N, M_n)/(\mathcal{L}_\omega+\mathcal{L}_\omega^*)$ as vector spaces thus this is the same set as the model theoretic ultraproduct for our structure. 

$(\varphi_n)^\omega$ is then defined in the same way and we have to see that the canonical map above gives a state preserving $*$-homomorphism. The only non obvious part is the identification of the product. But if one uses all the results of \cite[Th 4.1, lemma 4.13,4.14]{HaagerupAndo}, it is obvious that all our ultraproduct data is the data taken in their ultraproduct, even the multiplication map $m_{K,L}(x_n,y_n)=F_K(x_n).F_L(y_n)$ since $(F_K(x_n)))\in\mathcal{M}^\omega$ and the product is defined as the sequence of products on those sequences. Thus the von Neumann algebra structure has to coincide with the model theoretic one, since we have just checked the model of the Ocneanu ultraproduct (which is known to be a $W^*$-probablity space) is indeed the model theoretic ultraproduct of models (and since we checked in our previous steps that the model determines the von Neumann algebra structure).



\end{proof}
\subsection{More axiomatization results}

We now consider extra properties enabling to axiomatize explicitly interesting classes. We fix a closed discrete set $\Gamma\subset \R$ that will contain $Sp(\sigma^\varphi)\subset \Gamma$. The next axiom says $\sigma^\varphi$ has the appropriate spectrum with respect to $\Gamma$. 
\begin{enumerate}
\item[(21)]For $N\in\Q\cap]0,\infty[,l\in\Q^*,m\in\N^*$ with $]l-N,l+N[\subset \Gamma^c$  
\[ \sup_{x\in D_{m}}d_U(F_{N,l}(x)),0)=0.\]
\end{enumerate}
Up to now, all axioms form a $\forall$-axiomatizable theory.

Our last axioms in the case $Sp(\sigma^\varphi)= \log(\lambda)\Z,\lambda\in]0,1[$ will enable to identify a $III_\lambda$-factor with a periodic state. We will use notations similar to those in \cite{FarahII} to use their characterization of $II_1$ factors on the centralizer:
\[\xi_{K}(x)=\sqrt{\varphi(M_{(K,K)}(x^*,x))-|\varphi(x)|^2}\]
\[Com_{K}(a,b)=[m_{K,K}(b,a)-m_{K,K}(a,b)],\]
\[\eta_{K}(x)=\sup_{y\in D_1}\varphi\left[M_{(2K,2K)}([Com_{K}(x,y)]^*,[Com_{K}(x,y)])\right],\]
\[Proj_{K}(a)=[m_{K,K}(a,a^*)-m_{2K,2K}(m_{K,K}(a,a^*),[m_{K,K}(a,a^*)]^*)],\]
Then our next two axioms are expressed as follows :
\begin{enumerate}
\item[(22)] 
For $N\in\Q\cap]0,\infty[,l\in\Q^*,m\in\N^*$ with $]l-N,l+N[\cap \Gamma =\{n\log(\lambda)\}, n\in \N^*$ and $n\log(\lambda)\in \Gamma\cap ]l-N/2,l+N/2[$  
\[ \inf_{x\in D_{1}}\max\left[|d_U(M_{\lceil|l|+N\rceil+1,\lceil|l|+N\rceil+1}(x^*,x),1),d_U(2F_{N,l}(x)-F_{N/2,l}(x),x)\right]=0.\]
\item[(23)]For $N\in\Q\cap]0,\infty[,N<|\log(\lambda)|,m\in\N^*,K\geq \lceil N\rceil$ 
\[ \sup_{x\in D_1}\max(0,\xi_{K}(F_N(x))-\eta_{K}(F_N(x)))=0,\]
\[ \inf_{x\in D_1}(\varphi(m_{4K,4K}((Proj_{K}(F_N(x)))^*,Proj_{K}(F_N(x))))+|\varphi(M_{(K,K)}(F_N(x),F_N(x^*)))-1/\pi|)=0.\]
\end{enumerate}

We will see that a result of \cite{HaagerupAndo} will state that $II_1$-factors are not axiomatizable in our language for $\sigma$-finite von Neumann algebras. Of course, tracial $W^*$-probability spaces are axiomatizable in requiring $\varphi$ to be a trace. Similarly, even though $II_\infty$ factors won't be axiomatizable, we may be interested in having a canonical model in choosing a specific state $\varphi=tr\ot \varphi|_{B(H)}$ on $M=N\ot B(H)$, with $N$ a tracial von Neumann algebra. This is the purpose of the next axiom that uses extra constant symbols $w_{i,j}\in D_1,i,j\in\N$ for a matrix unit and also that $\psi=\varphi|_{B(H)}$ is given by $\psi(w_{j,j})=2^{-j-1}$, $\psi(w_{j,k})=0$. We will say $\varphi$ is a geometric state (for this matrix unit). Note that this implies (e.g. by \cite[Th 2.11]{TakesakiBook}) that $\sigma_t^{\varphi}(w_{j,k})=2^{(k-j)it}w_{j,k}$ and thus $w_{k,j}\in M(\sigma^\varphi,\{(j-k)\ln(2)\}).$

\begin{enumerate}
\item[(24)] $\varphi(w_{j,j})=2^{-j-1}$, $\varphi(w_{j,k})=0$, $w_{k,j}^*=w_{j,k}$, $H_{\lceil|j-k|\ln(2)\rceil+1}(w_{j,k})=w_{j,k},$  \[M_{(\lceil|j-k|\ln(2)\rceil,\lceil|l-m|\ln(2)\rceil)}(w_{j,k},w_{l,m})=\delta_{k,l}w_{l,m},\]
\item[(25)]$H_1(w_{0,0}F_N(x)w_{0,0})=w_{0,0}F_N(x)w_{0,0},$ with $w_{k,k}F_N(x)w_{j,j}:=M_{N,0}(M_{0,N}(w_{k,k},F_N(x)), w_{j,j})$ \[\varphi(M_{(0,0)}(w_{0,0}xw_{0,0},w_{0,0}yw_{0,0}))=\varphi(M_{(0,0)}(w_{0,0}yw_{0,0},w_{0,0}xw_{0,0})).\]
\item[(26)]
\[ \sup_{x\in D_1}\max(0,\xi_{1}(w_{0,0}xw_{0,0})-\eta_{1}(w_{0,0}xw_{0,0}))=0,\]
\[ \inf_{x\in D_1}(\varphi(m_{4K,4K}((Proj_{1}(w_{0,0}xw_{0,0})^*,Proj_{1}(w_{0,0}xw_{0,0})))+|\varphi(M_{(1,1)}(w_{0,0}xw_{0,0},w_{0,0}x^*w_{0,0}))-1/\pi|)=0.\]
\item[(27)]
$\varphi(M_{(0,0)}(w_{0,0}xw_{0,0},w_{0,0}yw_{0,0}))=2\varphi(w_{0,0}xw_{0,0})\varphi(w_{0,0}yw_{0,0})$
\end{enumerate}

\begin{theorem}\label{MoreAxiom}
The following classes are also axiomatizable in the same language as $\sigma$-finite $W^*$-probability spaces :
\begin{enumerate}
\item $\sigma$-finite $W^*$ probability spaces with $Sp(\sigma^\varphi)\subset \Gamma$ for a discrete set $\Gamma$, by the theory $T_{\sigma W^*}(Sp\subset\Gamma)$ consisting of axioms (1)-(21).
\item $III_\lambda$ factors with a periodic state for  $0<\lambda< 1$,  by the theory $T_{\sigma III_\lambda}$ consisting of axioms (1)-(23) with $\Gamma=\log(\lambda)\Z$.
\item $III_\lambda$ factors with some faithful state for  $0<\lambda\leq 1.$
\end{enumerate}
The following classes are axiomatizable in the corresponding expansions described above in the definition of each formula:
\begin{enumerate}
\item[(iv)]  $\sigma$-finite  $W^*$ probability spaces  of the form $N\ot B(H)$ for $N$ tracial and 
with $\varphi$ a geometric state, by $T_{\sigma W^* geom}$ consisting of $T_{\sigma W^*}$ and (24)-(25).
\item[(v)] $\sigma$-finite $II_\infty$ factors 
with $\varphi$ a geometric state, by $T_{\sigma W^*II_\infty geom}$ consisting of $T_{\sigma W^*}$ and (24)-(26).
\item[(vi)]  $\sigma$-finite  type $I_\infty$ factors
with $\varphi$ a geometric state, by $T_{\sigma W^*I_\infty geom}$ consisting of $T_{\sigma W^* geom}$ and (27).
\end{enumerate}
\end{theorem}

\begin{remark}
The reader should note that $T_{\sigma W^*}(Sp\subset\Gamma),T_{\sigma W^* geom},T_{\sigma W^*I_\infty geom}$ are universal theories while $T_{\sigma W^*II_\infty geom},T_{\sigma III_\lambda}$ are $\forall\exists$ theories as was the theory of tracial $W^*$ probability spaces which are $II_1$ factors.
\end{remark}

\begin{proof}
Condition (21) says explicitly $Sp(\sigma_\varphi)\subset \Gamma$ in saying $M(\sigma, E)=\{0\}$  for any closed set $E\subset ]l-N,l+N[\subset \Gamma^c$. 

For a $III_\lambda$ factor with a periodic state Takesaki's theorem \cite[Th 1.27]{Takesaki}, there is an isometry $u$ (thus $u^*u=1$) in $M(\sigma^\varphi,\{\log(\lambda)\}).$
By discreteness of $\Gamma$ for  $\gamma=\log(\lambda)\in\Gamma$, one can find $l,N$ as in (22). 
Since $2\widehat{f_{N,l}}-\widehat{f_{N/2,L}}$ is supported in $[l-N,l+N]$ and equal to $1$ on $[l-N/2,l+N/2]$ thus a neighborhood of $\{\gamma\}$ so that $2F_{N,l}(u)-F_{N/2,L}(u)=u$ and similarly,  by \cite[lemma XI.1.3]{TakesakiBook} again, $M_{\lceil|l|+N\rceil+1,\lceil|l|+N\rceil+1}(u^*,u)=u^*u=1$ and thus (22) is satisfied.

Moreover, by the same theorem of Takesaki, the centralizer is a $II_1$ factor. Since the state is lacunary (as soon as (21) holds), the argument in \cite[lemma 2.3]{HaagerupStormer} gives that $F_N$ with $N$ as in (23) is the state preserving projection on the centralizer. (23) thus says that the centralizer satisfies the axioms (16)-(17) in \cite{FarahII} and thus is not only a tracial von Neumann algebra but a $II_1$ factor. Thus it is indeed satisfied if $M$ is a $III_\lambda$ factor with a periodic state.

Conversely, assuming not only (21) but also (23), we know that the centralizer is a $II_1$ factor and one argues in the spirit of \cite[Rmk 6.12]{HaagerupAndo}, the center of $M$ is in the centralizer, thus in the center of the centralizer, thus $M$ is a factor. Then by general results, the log of Connes' $S$-invariant is $\log(S(M)-\{0\})=Sp(\sigma^\varphi)$. We then claim that (22) implies that $Sp(\sigma^\varphi)=\log(\lambda) \Z$ (and not only included as (21) implies, and this gives by definition that $M$ is a type $III_\lambda$ factor and $\varphi$ a periodic state). Indeed assume $M(\sigma^\varphi,\{log(\lambda)\} )=\emptyset.$ This means for any $x$, $F_{N,l}(x)=F_{N/2}(x,l)=0$ but if (22) holds, there is thus a $x=x(\epsilon)$ with $||x||_\varphi^*\leq \epsilon$ and $||(M_{\lceil|l|+N\rceil+1,\lceil|l|+N\rceil+1}(x^*,x)-1||_\varphi^*\leq \epsilon$
But $||M_{K,L}(x^*,x)||_\varphi^*\leq 4(K+1)(L+1)C_{K+1}||x|| \ ||x||_\varphi^*$ which is as small as one wants if $\epsilon$ small enough, contradicting $|\  ||(M_{\lceil|l|+N\rceil+1,\lceil|l|+N\rceil+1}(x^*,x)||_\varphi^*-1|\leq \epsilon.$ Thus (22) implies 
$M(\sigma^\varphi,\{log(\lambda)\} )\neq\emptyset.$ 
The same reasoning applies to $M(\sigma^\varphi,\{n\  log(\lambda)\} )\neq\emptyset$ for $n\in \N$ and then by adjoint for $n\in \Z.$ 

This concludes the axiomatization of $III_\lambda$ factors with periodic states.

For $II_\infty$ factors (algebras) with geometric state, we checked that a geometric state satisfies (24) and conversely, $w_{i,j}\in A$ and is thus a matrix unit (for $\sum_{i}w_{i,i}=1$ since $w_{i,i}$ are orthogonal projections, $||1- \sum_{i\leq n}w_{i,i}||_\varphi^\#=\varphi(1- \sum_{i=0}^nw_{i,i})=2^{-(n+1)}\to 0. $ This gives the strong convergence of $\sum_{i\leq n}w_{i,i}$ to $1$). Moreover if $e=w_{0,0},$ $\varphi$ restricted to $eMe$ is tracial by (25) and for instance \cite[Prop IV.1.8]{TakesakiBook} gives $M\simeq eMe\ot B(H).$ Using \cite{FarahII}, axiom (26) then says $eMe$ is a $II_1$ factor. (27) says $eMe=\C$ since $2\varphi(w_{0,0}.w_{0,0})$ is a state on $eMe$ with $e$ as unit.

Finally, the axiomatizability statement for $III_\lambda$ factors, $\lambda\in]0,1]$ will follow from \cite[Prop 5.14]{BenYBHU}. We will only check that variants of the results of \cite{HaagerupAndo} implies their stability by ultraproducts and ultraroots. Note that actually, it suffices to check stability by ultraroots for countably incomplete ultrafilters (since one can always take an ultrafilter of that type in Keisler-Shelah theorem). 

 The stability by ultraproducts is \cite[Th 6.11]{HaagerupAndo} and our slight extension corollary \ref{AHlambda} bellow. 
  Recall for a fixed state $(M,\varphi)^\omega=M^\omega$ does not depend on $\varphi$. Of course, if $M^\omega$ is a $III_\lambda$-factor, the center of $M$ is included in the center of $M^\omega$ thus $M$ is a factor. Of course, the ultrapower of a type $I_n,I_\infty,II_1$ factor is of the same type (see e.g. \cite[Prop 6.1]{HaagerupAndo} for references, for type $I_\infty$ factors one can use our axiomatization $T_{\sigma W^*I_\infty geom}$ with a geometric state), the ultrapower of a type $III_\lambda$, $\lambda\in]0,1]$ is of the same type as recalled, and using \cite[Th 6.18]{HaagerupAndo}, the ultraproduct of a type $III_0$ is never a factor (see corollary \ref{AHzero} for the uncountable variant in case of countably incomplete ultrafilters). Thus it  only remains to exclude $M$ to be a $II_\infty$ factor (which is explained in the separable case in \cite[Prop 6.1]{HaagerupAndo}). But if $M=N\otimes B(H)$ for $H$ separable and $N$ a $II_1$-factor, since the ultrapower does not depend on the state, we can realize the ultrapower with the geometric state so that $M$ satisfies $T_{\sigma W^*II_\infty geom}$, thus so does its ultrapower, and thus it is a $II_\infty$ factor too.
\end{proof} 

\begin{remark}\label{NonAxiom}[Non-Axiomatizability results of various classes]

Using \cite[Prop 6.3, Th 6.18]{HaagerupAndo}, one sees that factors, $III_0$ factors are not stable by ultrapowers, and $II_1$-factors are not stable by ultraproducts, thus none of these classes are axiomatizable in the language for $\sigma$-finite factors. However, the proof above shows that every other classes of factors namely $I_n, n\in \N^*\cup \{\infty\},II_1,II_\infty$ are local (i.e. stable by ultrapowers and ultraroots).
\end{remark}

\begin{corollary}\label{Qelim}
The theory of $B(H)$ described above as $T_{\sigma W^*I_\infty geom}$ in the language of $\sigma$-finite von Neumann algebras  with constants for matrix units is $\omega$-categorical  and admits quantifier elimination.
\end{corollary}
\begin{proof}
The theory has a unique model $B(H)$ and it is separable, implying a fortiori $\omega$-categoricity. The quantifier elimination test \cite[Prop 13.6]{BenYBHU} is obvious to check since for a substructure $M\subset N$ in this language with $N$ a model of $T_{\sigma W^*I_\infty geom}$, since the constant of the matrix units are in the language and generate a $d$ dense space, $M=N$ and the extension requirement is trivial.
\end{proof}

\subsection{An approximately minimal language enabling axiomatization of $\sigma$-finite $W^*$-probability spaces.}

A really natural question, first asked to us by Ita\"i Ben Yaacov, concerns the minimal language enabling to carry out the previous  axiomatization. At first sight, we thought that the modular group would be one piece of this language, and obtaining this axiomatization (even an explicit one) was easy but unreadable. But, in trying to answer another really natural question communicated to us by Ilijas Farah about the definability of the modular group, we realized that even this piece of data was not necessary. We summarize those results in this subsection.`

We consider the following countable language $\mathcal{L}_\nu$. Recall we have one sort $U$ with domains of quantification $D_n$ for the operator norm ball of radius $n$ of $M$.
\begin{itemize}
\item[$\bullet$]The constant 0 which will be in $D_1$. 
\item[$\bullet$] For every $\lambda\in \Q[i]$ a unary function symbol also denoted $\lambda$ to be interpreted as a scalar
multiplication. For simplicity we shall write $\lambda x$ instead of $\lambda (x)$.
\item[$\bullet$] A unary function symbol $*$ for involution on $U$, leaving stable all domains.
\item[$\bullet$] A binary function symbol $+:D_n\times D_m\to D_{n+m}$. and for $K,L\geq \nu,$ $m_{K,L}:D_{n}\times D_{m}\to D_{nm}$ (interpreted as $F_K^\varphi(.)F_L^\varphi(.)$ with modulus of continuity as obtained in lemma \ref{ContinuityProduct}).
\item[$\bullet$]The constant 1 in $D_1$.
\item[$\bullet$] Two unary relation symbols $\varphi_r$ and $\varphi_i$ for the real and imaginary parts of the state $\varphi$, on $U$. 
\end{itemize}

Using the (obvious multisorted/multidomain variant of the) definition 9.27 of \cite{BenYBHU}, we want to obtain the following result for the theory $T_{\sigma W^*}^\nu$ composed of logical consequences of $T_{\sigma W^*}$ in the language  $\mathcal{L}_\nu$, i.e. its so-called restriction to  $\mathcal{L}_\nu$.

\begin{theorem}\label{OcneanuThMinimal}
For any $\nu\in \N$, $T_{\sigma W^*}$ is an extension by definitions of $T_{\sigma W^*}^\nu$.
\end{theorem}

\begin{proof}
By definition, $T_{\sigma W^*}$ is a conservative extension of $T_{\sigma W^*}^\nu$. It thus suffices to check that any supplementary logical symbol is definable. The main part is to prove the modular group $\sigma_t$ is defined in $T_{\sigma W^*}$ over $\mathcal{L}_\nu$. 

\setcounter{Step}{0}

\begin{step}
Definability of $\sigma_t$ in $T_{\sigma W^*}$ over $\mathcal{L}_\nu$.
\end{step}

This will use several spectral theory results we didn't recall yet. We will thus explain this in several sub-steps. To explain the general idea, to define $\Delta^{it}$ as a form bellow the form corresponding to the metric with generator $\Delta(1+\Delta)^{-1}$, we will define $(u+\Delta)^{it}=e^{it\ln(u+\Delta)}$ since $\ln(u+\Delta)$ is bounded bellow, one can use the standard Hille-Yosida theory of semigroups to define this from the resolvants $(\beta+ \ln(u+\Delta))^{-1}$ which can be itself produced from the semigroup at positive times $(u+\Delta)^{-t}$ and this semigroup can be defined by composition and integral formulas for fractional powers of $(u+\Delta)^{-1}$. Thus, in reverse order, we will start by defining those as forms and check all the steps can be written as Riemann integrals with explicit Lispchitz constants and  bounds on our domains of quantification in order to get the final definability. The starting point is an infemum formula to define $(u+\Delta)^{-1}$ in a way similar to our first definition of the metric. However, we should obtain an explicit domain of quantification where our infemum will be reached.

We first show that the relation \[\psi_{K,L,u}(x,y)=\varphi(F_K(y^*).\Delta(u+\Delta)^{-1}(F_L(x)))\ \mathrm{for}\ u\in \Q, u>0, K,L\geq \nu\] is defined in $T_{\sigma W^*}$ over $\mathcal{L}_\nu$ (even equivalent to a formula of the language in any model). Formally, this formula in $\mathcal{L}$ has to be understood as using the form for $\Delta^{1/2}$ and $G_{\ln(u)}$ related to $\Delta^{1/2}(u+\Delta)^{-1}$ namely \[\psi_{K,L,u}(x,y):=\frac{1}{2\sqrt{u}}\mathscr{E}_{1/2,K,L}(y,G_{ln(u)}(x)).\]

 First note that by definition and polarization : $\psi_{K,L,u}(x,y)=\langle \Delta^{1/2}(u+\Delta)^{-1/2}(F_K(y)\xi_\varphi),\Delta^{1/2}(u+\Delta)^{-1/2}(F_L(x)\xi_\varphi)=\frac{1}{4}\sum_{k=0}^3(-i)^k\psi'_{K,L,u}(i^kx,y)$
with
\begin{align*}\psi'_{K,L,u}(x,y)&:=||\Delta^{1/2}(u+\Delta)^{-1/2}[(F_K(y)+F_L(x))\xi_\varphi]||^2\\&=||\Delta(u+\Delta)^{-1}[(F_K(y)+F_L(x))\xi_\varphi]||^2+u||\Delta^{1/2}(u+\Delta)^{-1}[(F_K(y)+F_L(x))\xi_\varphi]||^2\end{align*}

Now if $x,y\in D_1(M),$ $X=F_K(y)+F_L(x)$, one sees as in lemma \ref{NormG} that 
\begin{align*}\psi'_{K,L,u}(x,y)&=\inf_{z\in M}\left[{\varphi(z^*z)+\frac{1}{u}\varphi((X-z)(X-z)^*)}\right]\\&=\inf_{z=2F_{2K+2L+2}(z)-F_{K+L+1}(z)\in D_{m(K,L,u)}(M)}\left[{\varphi(z^*z)+\frac{1}{u}\varphi((X-z)(X-z)^*)}\right]\end{align*}
since the first infemum, in any $W^*$ probability space, is reached at $z=\Delta(u+\Delta)^{-1}(X)=X-u(u+\Delta)^{-1}(X)\in D_{m(K,L,u)}(M)$ with $m(K,L,u)=\frac{3(e^K+e^L)}{2\sqrt{u}}$ since $||\Delta^{1/2}(F_K(y))||_M\leq 3e^K$ (using again the proof of \cite[lemma 4.13]{HaagerupAndo}) and the bound on $G_s$ in lemma \ref{NormG} as operator on $M$. And the equality with the second infemum follows since we also have the infemum reached in $M(\sigma^\varphi,[-K-L,K+L])$ explaining the identity $z=2F_{2K+2L+2}(z)-F_{K+L+1}(z)$. Thus we also obtain the explicit formula in $\mathcal{L}_\nu.$
For, recall that $\tilde{m}$ is explicitly defined from $m$ by \begin{equation}\label{tildem}\tilde{m}_{(N,N)}(x,y)=4m_{(2N,2N)}(x,y)+
m_{(N,N)}(x,y)-2m_{(2N,N)}(x,y)-2m_{(N,2N)}(x,y)\end{equation} so that (since the infemum below corresponds to the first infemum above for elements of the form $z=2F_{2K+2L+2}(Z)-F_{K+L+1}(Z)$ and is thus a priori in between the two infema we showed equal) we obtained :
\begin{align*}\psi'_{K,L,u}(x,y)&=\inf_{Z\in D_{m(K,L,u)}(M)}\left[\varphi(\tilde{m}_{K+L+1,K+L+1}(Z^*,Z))\right.\\&\left.+\frac{1}{u}\varphi(\tilde{m}_{K+L+1,K+L+1}((m_{L,L}(x,1)+m_{K,K}(y,1)-Z),(m_{L,L}(x,1)+m_{K,K}(y,1)-Z)^*)\right].\end{align*}

We thus have our starting point to apply spectral theory and make definable the modular group.  Since the various spectral theory maps are only defined at $L^2$ level, we only consider for a while either the corresponding forms and show that they give definable relations, or their composition with $F_K$ for which we can get operator norm estimates. The intermediate formulas bellow won't be part of the language $\mathcal{L}$ but we will show they are definable universally in any $W^*$ probability space. 

First, we note that $\psi_{K,L,u_1,...,u_n}(x,y)=\varphi(F_K(y^*).\Delta(u_1+\Delta)^{-1}...(u_n+\Delta)^{-1}(F_L(x)))$   for $u_1\neq ...\neq u_n\in \Q, u_i>0, K,L\geq \nu$ are definable in $T_{\sigma W^*}$ over $\mathcal{L}_\nu$ in using the resolvent relations $(u_1+\Delta)^{-1}(u_2+\Delta)^{-1}=\frac{1}{u_2-u_1}[(u_1+\Delta)^{-1}-(u_2+\Delta)^{-1}]$ iteratively. 
Finally for general $u_1,...,u_n$ the definability comes from (the same equation implying with standard bounds on operator norms of resolvents of positive operators) : \[T_{\sigma W^*}\models \sup_{(x,y)\in D_m^2}|\psi_{K,L,u_1,...,u_i,...,u_n}(x,y)-\psi_{K,L,u_1,...,u_i+\epsilon,...,u_n}(x,y)|\leq \frac{\epsilon m^2}{u_1...u_n(u_i+\epsilon)} .\]

Then we want to show the definability of a map corresponding in any model to \[A_{u,K}:=(u+\Delta)^{-1}F_K, u\in \Q,u>0,K\geq \nu.\] First note that $A_{u,K}=\frac{1}{2\sqrt{u}}\Delta^{-1/2}G_{ln(u)}F_K$ and thus maps (using again the proof of \cite[lemma 4.13]{HaagerupAndo}) $D_m(M)\to D_{\lceil \frac{3me^K}{2\sqrt{u}}\rceil}(M).$ 
But more, we have the relation \[A_{u,K}=u^{-1}F_K-u^{-1}\Delta(u+\Delta)^{-1}F_K=u^{-1}F_K-\frac{1}{2\sqrt{u}u}\Delta^{1/2}G_{ln(u)}F_K\] so that we know it maps $D_m(M)\to D_{\frac{m}{u}+ \frac{3me^K}{2\sqrt{u}u}}(M).$ (This better estimate will be used to make converge an integral in norm bellow while the first would not have been sufficient).

Thus we want to estimate uniformly over models $||A_{u,K}(x)-y||_\varphi^*$ and thus since \[\sup_{(x,y)\in D_m^2}| \ ||A_{u,K}(x)-y||_\varphi^*-||A_{u,K}(x)-F_L(y)||_\varphi^*\ |\leq \sup_{y\in D_m}||y-F_L(y)||_\varphi^*\to_{L\to \infty} 0,\] (from the integral definition, the fact that Fejer's kernel is a positive mollifier  and lemma \ref{ContinuityModular}) it suffices to check definability, for $L$ large enough, $u\neq 1,$ of
\begin{align*}&(||A_{u,K}(x)-F_L(y)||_\varphi^*)^2=||\Delta^{1/2}(1+\Delta)^{-1/2}(A_{u,K}(x)-F_L(y))\xi_\varphi||^2\\&=\Re\langle2 F_L(y)+(u+\Delta)^{-1}F_K(x),\Delta(1+\Delta)^{-1}(u+\Delta)^{-1}F_K(x)\rangle+(||F_L(y)||_\varphi^*)^2
\\&=2\Re \psi_{L,K,1,u}(x,y)+ \psi_{K,K,1,u,u}(x,x)+(||F_L(y)||_\varphi^*)^2\end{align*}
Thus our previous computation gives the definability of $A_{u,K}$ for every $u\in \Q,u>0$ over $\mathcal{L}_\nu$.

We then check the same kind of definability for \[B_{u,t,K}=(u+\Delta)^{-t}F_K, u,t\in \Q,u,t>0,K\geq \nu,t<1/2\]

One uses the relation in lemma \ref{EquationForms}, $B_{u,t,K}=\frac{\sin(t \pi)}{\pi}\int_0^\infty dv A_{u+v,K}v^{-t}$. From the bound before for $A$ one deduces for $y\in D_m$, (the second inequality by cuting the integral at $u$ and standard bounds such as $\frac{\sin(t \pi)}{\pi}\leq t$ and the last inequality valid only for $u<1,t<1/2$, case we will use later) \begin{align*}||B_{u,t,K}(y)||&\leq \frac{\sin(t \pi)}{\pi}m\int_0^\infty dv (\frac{v^{-t}}{u+v}+\frac{3e^Kv^{-t}}{2(u+v)\sqrt{u+v}})
\\&\leq m(\frac{u^{-t}}{(1-t)}+
\frac{3t(4-t)e^Ku^{-1/2-t}}{2(1+2t)(1-t)})\\&\leq m
8e^K u^{-1/2-t}\ \ \mathrm{for }\ u<1,t<1/2.\end{align*}
This gives a domain of value for $B_{u,t,K}$ and the definability is easy since the integral is a Riemann integral with uniform Lipschitz bound  in $v$ for $A_{u+v,K}$ coming again from the resolvent equation. We then extend  $B_{u,t,K}=(2B_{u,t/n,2K}-B_{u,t/n,K})^{\circ (n-1)}\circ B_{u,t/n,K}$ for other $n/2>t>0$ which is thus definable by composition and has also the same formula as the original $B_{u,t,K}$ in any model. Note also (the last inequality for $u<1,t<n/2$) $||B_{u,t,K}(y)||\leq 8^n.3^{n-1}e^{2(n-1)K+K}u^{-n/2-t}\|y\|.$

For a supplementary $\beta\in \Q, $ large enough, we then show similarly the definability of \[C_{u,\beta,K}:=(\beta+\ln(u+\Delta))^{-1}F_K=\int_0^\infty dve^{-\beta v} B_{u,v,K},\]
since for $u<1,y\in D_m(M)$ we have the bound: \[||C_{u,\beta,K}(y)||\leq m\sum_{k=0}^\infty\int_{(k-1)/2}^{k/2} dv e^{-(\beta+\ln(u)) v} 8^k.3^{k-1}e^{2(k-1)K+K}u^{-k/2}
\leq\frac{e^{(\beta+\ln(u)) /2}}{3(\beta+\ln(u))(1-e^{-1})}e^{-K}m.\]
by a crude bound by a geometric series $Ce^{-k}$ for $\beta+2\ln(u)>4K+2\ln(24)+2.$  The definability is again easy by Riemann integration. Indeed, for $v>t$, we have the lipshitzness bound deduced from spectral theory:  \[\|(B_{u,v,K}-B_{u,t,K})(y)\|_\varphi^*
\leq u^{-t}|v-t|\|\ln(u+\Delta)F_K(y)\|_\varphi\leq u^{-t}|v-t|\ln(u+e^K)\|y\|.\]

We now want to exponentiate (the map in front of $F_K$) in  \[\beta \ln(u+\Delta)C_{u,\beta,K}=\beta (F_K-\beta C_{u,\beta,K})\] of course in a definable way to get for $t\in \Q,$ $E_{u,\beta,t,K}:=e^{it \beta \ln(u+\Delta)(\beta+\ln(u+\Delta))^{-1}}F_K$, in any model, we have the formula (using all the maps leave stable the spectral algebras with spectrum $[-K,K]$, reached by the first $F_K$, so that all the $2F_{2(K+1)}-F_{K+1}$ appearing can be replaced by identity):

\[E_{u,\beta,t,K}=F_K+\sum_{k=0}^\infty\frac{(it)^k}{k!} [2\beta (F_{2K+2}-\beta C_{u,\beta,2K+2})-\beta (F_{K+1}-\beta C_{u,\beta,K+1})]^{\circ k}\circ [\beta (F_K-\beta C_{u,\beta,K})].\]
From this, for $\beta$ large enough as above, it is easy to get the domain of value of the map and the definability over $\mathcal{L}_\nu$. We now want to make $\beta\to\infty,u\to 0,K\to\infty.$

For this final step, we won't need anymore any composition of maps, thus we come back to definition as forms. From the above definability, one gets the same for the relation (from the equality in any model) : \[F_{K,L,u,\beta,t}(x,y):=\varphi(F_L(y^*).\Delta(1+\Delta)^{-1}E_{u,\beta,t,K}(x))=2\psi_{L,2K,1}(E_{u,\beta,t,K}(x),y)-\psi_{L,K,1}(E_{u,\beta,t,K}(x),y).\]

Now, we consider the potential limit for $\beta\to\infty$ : \[ F_{K,L,u,t}(x,y)=\varphi(F_L(y^*).\Delta(1+\Delta)^{-1}e^{it\ln(u+\Delta)}F_K(x))\]
One can use Duhamel's formula to get :
\begin{align*}&F_{K,L,u,t}(x,y)-F_{K,L,u,\beta,t}(x,y)=\int_0^tds[\varphi(F_L(y^*)\\&.\Delta(1+\Delta)^{-1}e^{i(t-s) \beta \ln(u+\Delta)(\beta+\ln(u+\Delta))^{-1}}\ln(u+\Delta)(1-\beta (\beta+\ln(u+\Delta))^{-1})e^{is\ln(u+\Delta)}F_K(x))].\end{align*}
Now since $\beta \ln(u+\Delta)(\beta+\ln(u+\Delta))^{-1},\ln(u+\Delta)$ are self-adjoint, one gets :
\begin{align*}&|F_{K,L,u,t}(x,y)-F_{K,L,u,\beta,t}(x,y)|\\&\leq \frac{t}{\beta} \|\Delta(1+\Delta)^{-1}\ln(u+\Delta) F_L(y)\|_\varphi\|\ln(u+\Delta)\beta (\beta+\ln(u+\Delta))^{-1}F_K(x)\|_\varphi
\\&\leq \frac{t}{\beta} 
\ln(u+e^L) \ln(u+e^K)\|y\|_\varphi\|x\|_\varphi.\end{align*}

One thus deduces from the limit $\beta\to\infty$, the definability of $F_{K,L,u,t}.$
Similarly we show the definability over $\mathcal{L}_\nu$ of \[ F_{K,L,t}(x,y)=\varphi(F_L(y^*).\Delta(1+\Delta)^{-1}\sigma_t^\varphi(F_K(x)))\]
in using Duhamel's formula again to get the inequality 
\begin{align*}&|F_{K,L,u,t}(x,y)-F_{K,L,t}(x,y)|
\leq \frac{tu}{u+e^{-K}} \|y\|_\varphi\|x\|_\varphi.\end{align*}

Finally, we make $K,L\to \infty$ to get the interesting limit $ F_{t}(x,y)=\varphi(y^*.\Delta(1+\Delta)^{-1}\sigma_t^\varphi(x))$ in noting that in any model, by Schwarz inequality : 
\begin{align*}&|F_{t}(x,y)-F_{K,L,t}(x,y)|
\leq \sqrt{\varphi([F_L(y)-y]^*.\Delta(1+\Delta)^{-1}[F_L(y)-y]))\varphi(\sigma_t^\varphi(x))^*.\Delta(1+\Delta)^{-1}\sigma_t^\varphi(x))}.\\&+\sqrt{\varphi([F_L(y)]^*.\Delta(1+\Delta)^{-1}[F_L(y)]))\varphi(\sigma_t^\varphi(F_K(x)-x))^*.\Delta(1+\Delta)^{-1}\sigma_t^\varphi(F_K(x)-x))}
\\&\leq \|F_L(y)-y\|_\varphi^*\|x\|_\varphi^*+ \|F_K(x)-x\|_\varphi^*\|y\|_\varphi^*\end{align*}
with the last inequality coming from the identification in the proof of lemma \ref{NormG} again.
One deduces the uniform convergence in $x,y$, uniformly over models, as $K,L\to\infty$ from lemma \ref{ContinuityModular}.

Finally, we are ready to prove the expected definability of $\sigma_t$ in $T_{\sigma W^*}$ over $\mathcal{L}_\nu$. For, we need by definition to check the definability of $d(\sigma_t(x),y)^2=\|\Delta^{1/2}(1+\Delta)^{-1/2}(\sigma_t(x)-y)\xi_\phi\|^2$ by the proof of lemma \ref{NormG}. But by commutation of $\sigma_t$ with $\Delta$ and invariance of $\varphi,$ we have the alternative formula :\[d(\sigma_t(x),y)^2=\Re(F_0(x,x)+F_0(y,y)-2F_t(x,y)),\]
And this gives definability by our previous work.

\begin{step}
Definability of other maps from $\mathcal{L}$ in $T_{\sigma W^*}$ over $\mathcal{L}_\nu$.
\end{step}
The inequality in axiom (16) then gives the axiomatization of $G_t$ from the one of $\sigma_t$ (using the triangular inequality to make $d(G_t(x),y)$ close to the corresponding distance function with $G_t$ replaced by Riemann sums), similarly for $F_{N,l}$ with (18), $\mathscr{E}_{\alpha,K,L}$ with (19). $\tau_{p,\lambda,N}$ is obvious by axiom (11).

The definability of $\lambda(.)$ for $\lambda\in \C-\Q[i]$ is of course easy by density. 

Finally, it remains to check the definability of $m_{K,L}$ for $K,L\in \N^*$ (if we don't have $K,L\geq \nu$). From the formula $m_{K,L}(x,y)=\tilde{m}_{O,P}(F_K(x),F_L(y))$ for $O\geq\max(K+1,\nu),P\geq\max(L+1,\nu)$ and with the notation of \eqref{tildem}, it suffices to show definability of $F_K$ (for $K<\nu$) and this is a special case of $F_{m,l}$ above. This concludes the definability of all supplementary maps.
\end{proof}

\begin{corollary}
For any $\nu\in \N$, $\sigma$-finite $W^*$ probability spaces are axiomatizable  in the language  $\mathcal{L}_\nu$ (by the not so explicit theory $T_{\sigma W^*}^\nu$).\end{corollary}
\begin{proof}
Since $T_{\sigma W^*}$ axiomatizes $\sigma$-finite $W^*$ probability spaces, $T_{\sigma W^*}^\nu$ is exactly the theory considered in the already used consequence of Keisler-Shelah theorem \cite[Prop 5.14]{BenYBHU}, namely the set of all $\mathcal{L}_\nu$-conditions satisfied in all $\sigma$-finite $W^*$ probability spaces. We thus check that this class of models is stable by ultraproducts and ultraroot. Of course, this class of spaces is stable by ultraproducts since the language is a restriction of the previous one. Conversely, if $(M,\varphi)$ is a $\mathcal{L}_\nu$-structure, such that $(M,\varphi)^\omega$ is a $\sigma$-finite $W^*$-probability space. By \cite[Corol 9.31]{BenYBHU}, $(M,\varphi)^\omega$ has a unique expansion to a model $[(M,\varphi)^\omega]_{\mathcal{L}}$.
By definition, the extension by definitions result implies that each symbol in $\mathcal{L}$ not in $\mathcal{L}_\nu$ gives a definable constant, function or predicate in $(M,\varphi)^\omega$ as an $\mathcal{L}_\nu$-structure in the sense of \cite[Def 9.1]{BenYBHU}. Thus by \cite[Prop 9.7,9.25]{BenYBHU} and the theorem of ultraproducts, one deduces from the elementary embedding $(M,\varphi) \preceq(M,\varphi)^\omega$ as $\mathcal{L}_\nu$ theory, the elementary embedding $(M,\varphi)_{\mathcal{L}}\preceq[(M,\varphi)^\omega]_{\mathcal{L}}$ for the restriction to $(M,\varphi)$ of the above extensions (note that for functions one uses \cite[Prop 9.25]{BenYBHU} for the restriction to be well-defined). Especially, $(M,\varphi)_{\mathcal{L}}$ is a model of $T_{\sigma W^*}$ thus a $\sigma$-finite $W^*$-probability space as expected.

\end{proof}
\begin{remark}
It is not difficult to see (in combining (16)-(19)-(20) in an awful formula corresponding to a discretization of a sixfold integral to get the form $\mathscr{E}_1$, with one integral to express it in terms of $\mathscr{E}_{2/3}$, one to express this one in terms of $\mathscr{E}_{1/3}$ and another one in terms of $\mathscr{E}_{0}$ and 3 more to express each $G_s$ in those formulas) that $\sigma$-finite $W^*$ probability spaces are axiomatizable by a $\forall\exists$-theory in the language $\mathcal{L}_\nu$ with all modular group $\sigma_t,t\in\Q$ added. If we even add $\tau_{p,\lambda,n}$ in the language (say with $N_i\geq \nu$ in $N$) one can even get back the universal axiomatization. The author conjectures that there is  no $\forall\exists$-axiomatization in the language $\mathcal{L}_\nu$ and no universal axiomatization (even with $\tau_{p,\lambda,n}$ added).
\end{remark}

\section{The Groh theory for preduals of von Neumann algebras}
\subsection{The explicit theory for preduals jointly with a weak-* dense subalgebra of its dual}
We will use the notion of 
 matrix-ordered operator space in the sense of \cite{Werner}, that includes duals of $C^*$ algebras (and preduals of $W^*$-algebras). We call language for matrix-ordered  operator spaces the union of the language for operator spaces and operator systems in \cite{GoldbringSinclair} without any unit symbol (and thus without their map $h_n$ containing a unit) but with an extra $\oplus$ operation on positive cones.
Matrix-ordered operator spaces are easily axiomatized in this language if we remove Werner's axiom $(M_0)$ (saying that a sum of positive is $0$ if and only if each is $0$). We will recover this axiom later, we only keep the axiom saying that diagonal direct sum of positive are positive $(M_1)$ and conjugation by matrices keeps positivity $(M_2)$. More precisely
we also add in any cone sort ,for any polynomial in one variable positive on $\R_+$, the function $\tau_p$ with an axiom saying $\tau_p(x)=p(d(x,0))x$ with domains $\tau_p:D_n\to D_{\lceil\sup_{t\in [0,n]} |p(x)x|\rceil}$ so that approximating $x/\max(1,d(x,0))$ by such polynomials, one obtains as in \cite{FarahII} that the norm unit ball is indeed $D_n.$

We call tracial matrix-ordered operator spaces $X$ a matrix-ordered operator space with a completely positive linear functional $tr$ such that for $\phi\in M_n(X)$ positive, $||tr(\phi)||_{M_n(\C)}=||\phi||_{M_n(X)}.$ For instance, by Stinespring's theorem, the predual of a von Neumann algebra  is such a space with Haagerup's trace : $tr(\phi)=\phi(1).$ Note that, since $M_n(X)$ is a metric space, any $x,y\in M_n(X)$ positive then $tr(x+y)$ positive and is thus $0$ only if $tr(x)=tr(y)=0$ and by the norm relation $x=y=0,$ especially one recovers for free the axiom $(M_0)$ so that tracial matrix-ordered operator spaces are easily (explicitly) axiomatizable.

Groh's construction of a von Neumann algebra ultrapower in \cite[Prop 2.2]{Groh}, as dual of a Banach space ultrapower of preduals, suggests that predual of von Neumann algebras should be an axiomatizable class. Moreover, the proof of the construction of the ultrapower gives a theory for a pair $(M, X)$ of a $C^*$ algebra and a predual of a von Neumann algebra with $M\subset X^*$ weak-* dense. We will thus explicitly axiomatize the theory of these pairs.

We now explain the following theory but we consider $X$ as a tracial matrix-ordered operator space to get our stronger axiomatizability result bellow for preduals alone.

We consider a theory with five (families of) sorts : $M_n(V),$ each with the language of $C^*$-algebras from \cite{FarahII} with domains of quantification $D_m(M_n(V))$, with also the language of operator spaces (we will call this the language of $C^*$-algebras as operator spaces and in the corresponding theory the identification of the product with the matrix product will be written) and the sorts $(M_{p,q}(\mathcal{S}),\mathcal{C}_n,M_{m,n}(\C),\R_{\geq 0})$ with the language of 
tracial matrix-ordered operator spaces explained above, with domains of quantification written $D_m(.)$ with dot replaced by the sort. We add the following data :

\begin{itemize}
\item[$\bullet$]Unary function symbols $|.|: D_m(M_{n,n}(\mathcal{S}))\to D_m(\mathcal{C}_n).$ 
\item[$\bullet$] Binary functions symbols : \[m_{N,n}^{l}:(D_l(M_N(V)))\times D_m(M_n(S))\to D_{lm}( M_{Nn}(S)),\]  \[m_{n,N}^{r}:D_m(M_n(S))\times (D_l(M_N(V)))\to D_{lm}( M_{Nn}(S)),\] interpreted as module (tensorial at matrix level) action of $M_N(U)$ on $M_n(\mathcal{S}))$ and $m^{i}_1:=m^{i}_{1,1}.$
\item[$\bullet$] a binary function symbol $B:D_m(M_n(S))\times D_l(M_n(V))\to D_{n^2ml}(\C) $ (see axiom (30) for its meaning)
\item[$\bullet$] Unary function symbols $\pi^{(n)}_{ij}:D_m(M_{2n}(S))\to D_m(M_{n}(S)), i,j\in \{1,2\} $ (meaning a shorter notation for block projections), and $\exp i\Re:D_m(V)\to D_1(V), \arctan\Re:D_m(V)\to D_2(V), r:D_m(V)\to D_1(V)$ (meaning $\exp i\Re(a)=exp(i(a+a^*)/2), \arctan\Re(a)=\arctan((a+a^*)/2), r(a)=\frac{1}{1+(a+a^*)^2/4}$).
\item[$\bullet$] Unary function symbols $tr_n:D_m(M_{n}(S))\to D_m(M_{n}(\C)) $ (for $tr\ot id_n$).
\end{itemize}

Of course, some maps as $\exp i\Re,arctan\Re$ are not necessary, but will be convenient to produce cheap unitaries. We will define them by classical series valid in any $C^*$-algebra (Euler's series converging on $\R$ uniformly on segments will be convenient for $arctan(t)=\frac{t}{1+t^2}\sum_{n=0}^\infty\frac{(n!)^24^n}{(2n+1)!}\left(\frac{t^2}{1+t^2}\right)^n $).

Before stating any axiom, we have to explicit uniform continuity bounds.
Note that for $\phi,\varphi\in M_n(X)$ we use the definition for $B(\phi,.)=\hat{\phi}$ in (30) bellow :\begin{align*}|||\varphi|-|\phi|||_{cb}=&\sup_{k,||a||_{M_k(X^*)}\leq 1}||(|\varphi|-|\phi|)(a)||_{M_{kn}(\C)}\\&=\sup_{T\in D_1(TC_{nk})}|B(i_n(|\varphi|)-i_n(|\phi|),(\sum_{j,J=1}^kT_{(i,j),(I,J)}a_{j,J})_{iI})|\\&\leq \|\ |\widehat{\varphi}|-|\widehat{\phi}|\ \|\ \sup_{k,||a||_{M_k(X^*)}\leq 1}\sup_{T\in D_1(TC_{nk})}||(\sum_{j,J=1}^kT_{(i,j),(I,J)}a_{j,J})_{iI}||_{M_n(M)}\
\\&\leq c_n \left(2\sqrt{||\widehat{\varphi}||||\widehat{\varphi}-\widehat{\phi}||}+||\widehat{\varphi}-\widehat{\phi}||\right)\end{align*}

with  the last inequality coming e.g. from the proof of \cite[Prop III.4.10]{TakesakiBook} and with $c_n=\sup_{||a||_{TC_n\widehat{\ot}M}\leq 1}||a||_{M_n(M)}$ is a universal constant\footnote{this is quite standard it is finite since for instance from \cite[section 9.3]{EffrosRuan} in terms of row and column Hilbert spaces :$TC_n\widehat{\ot}M\simeq R_n\ot_hM\ot_hC_n\to C_n\ot_hM\ot_hR_n\simeq M_n(M)$  coming from identity $CB$ maps $R_n\to C_n$, $C_n\to R_n$ so that it is even known that $c_n\leq n$}  used thanks to the bound\footnote{coming from the canonical completely contractive map $(TC_n\widehat{\ot}TC_k)\widehat{\ot}(M_k\ot_{min} M)\simeq TC_n\widehat{\ot}(TC_k\widehat{\ot}(M_k\ot_{min} M))\to TC_n\widehat{\ot}((TC_k\widehat{\ot}M_k)\ot_{min} M))\to TC_n\widehat{\ot} M$  using in the middle the shuffle map $TC_k\widehat{\ot}(M_k\ot_{min} M))\simeq R_k\ot_h(M_k\ot_{min} M)\ot_hC_k\to( R_k\ot_hM_k\ot_hC_k)\ot_{min} M$ from \cite[Th 5.15]{PisierBook}}
\[||(\sum_{j,J=1}^kT_{(i,j),(I,J)}a_{j,J})_{iI}||_{M_n(M)}\leq c_n ||(\sum_{j,J=1}^kT_{(i,j),(I,J)}a_{j,J})_{iI}||_{TC_n\widehat{\ot}M}\leq c_n ||T||_{TC_n\widehat{\ot}TC_k}||a||_{M_k(M)}.\]
  Using then $||\widehat{\varphi}||\leq n ||\varphi||_{cb}$ one thus gets an explicit uniform continuity function for $|.|$ that we take in the language setting. This was the only non-obvious uniformly continuous map.

We consider the following axioms :

\begin{enumerate}
\item[(28)] For $x,y\in D_l(M_N(V)),z,t\in D_m(M_n(S)),\lambda\in\C$, any $N,n\in \N^*$ \[m_{N,n}^l(\lambda x+y,z)= \lambda m_{N,n}^l( x,z)+m_{N,n}^l( ,z), m_{n,N}^l(x,\lambda z+t)= \lambda m_{n,N}^l(x,z)+m_{n,N}^l(x,t),\] %
\[m_n^r(\lambda z+t,x)= \lambda m_n^r( z,x)+m_n^r( t,x), m_n^r(z,\lambda x+y)= \lambda m_n^r( z,x)+m_n^r( z,y),\] 
and $m_{1,n}^l(1,z)=z=m_{n,1}^r( z,1),$
\[\pi_{nN}^{(i,j),(I,J)}(m_{N,n}^l(x,z))=m_1^l(\pi_N^{iI}(x),\pi_n^{jJ}(z))\]
\[m_1^l( xy,z)=m_1^l( x,m_1^l( y,z)),
 m_1^r( z,xy)=m_1^r( m_1^r( z,x),y), 
m_1^l( x,m_1^r( z,y))=m_1^r(  m_1^l( x,z),y), \]
\[ \sup_{a\in D_m(V)}\max(d(r(a)[1+(a+a^*)^2/4],1),d([1+(a+a^*)^2/4]r(a),1))=0,\]
\[ \sup_{a\in D_m(V)}\max(0,d(\exp i\Re(a),\sum_{k=0}^N\frac{i^k}{2^kk!}(a+a^*)^k)-exp(m)+\sum_{k=0}^N\frac{m^k}{k!})=0,\]
\begin{align*} \sup_{a\in D_m(V)}\max(0,&d(\arctan\Re(a),\frac{a+a^*}{2}r(a)\sum_{n=0}^N\frac{(n!)^24^n}{(2n+1)!}\left(\frac{(a+a^*)^2}{4}r(a)\right)^n \\&-(1+m^2)\arctan(m)+m\sum_{n=0}^N\frac{(n!)^24^n}{(2n+1)!}\left(\frac{m^2}{1+m^2}\right)^n )=0.\end{align*}
\item[(29)] $tr(m_1^l( x,z))=tr(m_1^r( z,x))$ and 
\[\sup_{x\in D_n(V)}|d(x,0)-\sup_{y\in D_1(S)}|tr(m_1^l( x,y))|\ |=0,\]
\[\sup_{x\in D_n(\mathcal{S})}|d(x,0)-\sup_{y\in D_1(V)}|tr(m_1^r( x,y))|\ |=0,\]
\item[(30)] For $x\in C_n, y\in M_n(S), a\in D_l(M_n(V))$, $|i_n(x)|=x$, \[B(y,a)=\sum_{i,j=1}^ntr(m_1^l(\pi_n^{ij}(a),\pi_n^{ij}(y))),\]
\[\sup_{y\in D_l(M_n(S))}\sup_{a\in D_k(M_n(V))}\max\left[0,|B(y,a)|^2- \sup_{b\in D_1(M_n(V))}|B(y,b)| \ B(i_n(|y|),aa^*)\right],\]
\[ \sup_{y\in D_l(M_n(S)}\left|\sup_{b\in D_1(M_n(V))}|B(y,b)|- \sup_{b\in D_1(M_n(V))}|B(i_n(|y|),b)|\right|=0,\]

\item[(31)] Axioms for $tr_n,\pi^{(n)}_{ij}$ and :

\[ \sup_{y\in D_l(M_n(\mathcal{S}))}\left|d(y,0)- \inf_{x\in D_{5l}(\mathcal{C}_{2n})}\left[\max(||tr_n(\pi^{(n)}_{11}(x))||,||tr_n(\pi^{(n)}_{22}(x)||)+d(y,\pi^{(n)}_{12}(x))\right]\right|=0,\]
\[ \sup_{y\in D_l(M_n(\mathcal{S}))}\sup_{a\in D_k(M_N(V))}\max\left(0,||(tr_{nN}(m_{N,n}^l(a,y)))||_{M_{nN}(\C)}-d(y,0)d(a,0)\right)=0.\]
\end{enumerate}

We will also consider a last axiom depending on a parameter $d\in[0,2]$ for the diameter of state space to recover the results of \cite[Section 6.2]{HaagerupAndo} in our continuous logic setting.
\begin{enumerate}
\item[(32)] For any $m\in\N^*$, and write for short $f_r(a)=\exp i\Re(2 \arctan\Re(a))$ \begin{align*}&\left|\sup_{x,y\in D_1(\mathcal{C}_1)}\inf_{a\in D_m(V)} d(i_1(x)tr(i_1(y)),m_1^l(f_r(a),m_1^r(i_1(y)tr(i_1(x)),[f_r(a)]^*))-d\right|\leq 4\sin(2\arctan(m))\end{align*}

For any $N\in\N^*$, \[\inf_{(x_1,...,x_N)\in (D_1(V))^N}\left[\max(\max_{i=1,...,N}(\max(d(x_ix_i^*,(x_ix_i^*)^2)), |d(x_ix_i^*,0)-1|), \max_{i\neq j=1,...,N,}(d(0,x_ix_i^*x_jx_j^*))\right]=0.\]
\end{enumerate}

Let us now explain how  a pair $(M,X)$ of a weak-* dense unital $C^*$ subalgebra $M$  in a von Neumann algebra $X^*$ with predual $X$ gives rise to a model of this theory. Of course $m_1^l, m_1^r$ are induced by the usual actions of $X^*$ on $X$ for $x,y\in X^*,\varphi\in X$,  $m_1^l(x,\varphi)(y)=\varphi(yx)$ and then $m_n^l$ is defined entry-wise. $B(\varphi,.)=\hat{\varphi}$ is the map defined in (30) or \cite[p 16]{BrownOzawa} that converts $M_n(X)=NCB(X^*,M_n(\C))$ into maps in $(M_n(X^*))^*$. Being a $CP$ map i.e. in the cone $C_n$ is then equivalent to have $\hat{\varphi}$ positive as a functional by their proposition 1.5.14. The axiom (30) then states that $\widehat{|\varphi|}=|\widehat{\varphi}|$ in the sense of absolute values in $(M_n(X^*))_*$ characterized by \cite[Prop III.4.6]{TakesakiBook}. By strong-* density of $M$ in $X^*$ the defining relation is only checked on $M$ and since $\hat{.}$ is a bijection $|.|$ is defined uniquely and has indeed its range in completely positive maps and all of (30) is satisfied.  (28),(29) are then obvious and (31) comes from the identification of the completely bounded norm with the decomposable norm (see e.g. \cite[lemma 5.4.3]{EffrosRuan}).
Indeed, we have from this lemma  \[d(y,0)=||y||_{cb}\leq ||y-x||_{cb}+||x||_{dec}=||y-x||_{cb}+\max(||\psi_1||,||\psi_2||)\] when $\left(\begin{array}{cc}\psi_1&x\\x^*&\psi_2\end{array}\right)$ is completely positive i.e. in $\mathcal{C}_{2n}$.
Thus taking the infemum \[d(y,0)\leq \inf_{x\in (\mathcal{C}_{2n})}\left[\max(||tr_n(\pi^{(n)}_{11}(x))||,||tr_n(\pi^{(n)}_{22}(x)||)+d(y,\pi^{(n)}_{12}(x))\right]\]
and since $||y||_{cb}=||y||_{dec}$ the infemum is reached when $\pi^{(n)}_{12}(x)=y$ and we can take $||\psi_i||=||y||$ giving especially the equality stated in (31). The second part states $||y||_{cb}\leq d(y,0)$ which is satisfied when we have equality.

The second equation in (32) states that there are $N$ orthogonal non-zero projections and thus axiomatizes infinite dimensionality. Let us also interpret the first equation in axiom (32).

\begin{lemma}
For a pair $(M,X)$ as above the first equation in (32) is equivalent to the equalty on the state space diameter  of $X^*$: $d(X^*)=d$. 
\end{lemma} \begin{proof}Note that, in our formula, we only look  at unitaries $u=\exp i\Re(2 \arctan\Re(a))$ so that for $x,y\in D_1(X)$: \begin{align*}\delta_m(x,y)&:=\inf_{a\in D_m(M)} d(i_1(x),m_1^l(\exp i\Re(2 \arctan\Re(a)),m_1^r(i_1(y),\exp i\Re(-2 \arctan\Re(a)))))\\&\geq \delta(x,y):=\inf_{u\in D_1(X^*), uu^*=1=u^*u} d(i_1(x),m_1^l(u,m_1^r(i_1(y),u^*)))\end{align*}
But, since by functional calculus, any unitary $u$ is close to a unitary $u_m=\exp i\Re(2 \arctan\Re(a_m))$ with $a_m\in  D_m(X^*)$ and explicitly with $||u-u_m||\leq 2\sin(2\arctan(m)),$ one gets using Kaplansky density theorem for the first equality and $f(m)=4\sin(2\arctan(m)):$ \[\delta_m(x,y)=\inf_{a\in D_m(X^*)} d(x,\exp i\Re(2 \arctan\Re(a)).y.\exp i\Re(-2 \arctan\Re(a)))))\leq \delta(x,y)+f(m).\]
One thus obtains that \[\left|\sup_{x,y\in D_1(\mathcal{C}_1)}\delta_m(i_1(x)tr(i_1(y)),i_1(y)tr(i_1(x)))-\sup_{x,y\in D_1(\mathcal{C}_1)}\delta(i_1(x)tr(i_1(y)),i_1(y)tr(i_1(x)))\right|\leq f(m).\]
Finally, it is easy to see that the state space diameter of $X^*$ is  \[\sup_{x,y\in D_1(\mathcal{C}_1)}\delta(i_1(x)tr(i_1(y)),i_1(y)tr(i_1(x)))=\sup_{x,y\in D_1(\mathcal{C}_1), tr(x)=tr(y)=1}\delta(i_1(x),i_1(y)))\]
 and since $f(m)\to_{m\to \infty} 0$ (32) is equivalent to the computation of the state space diameter to be equal to $d$. 
\end{proof} 
 We are now ready to prove our axiomatization result. We consider the category whose objects are pairs $(M,X)$ of a weak-* dense $C^*$ algebra $M$ of $X^*$ and a predual of a von Neumann algebra $X$, and morphism are any  pair of a $*$-homomorphism $M_1\to M_2$ extending to a weak-* continuous $*$-homomorphism $X_1^*\to X_2^*$ onto a subalgebra having a normal conditional expectation from $X_2^*$ which is the dual map of a completely positive tracial map of preduals $X_1\to X_2$.

\begin{theorem}\label{GrohTh}
The class of couples $(M,X)$ of a weak-* dense $C^*$ algebra $M$ of $X^*$ and a predual of a von Neumann algebra $X$ is axiomatizable by the theory $T_{C^*,(W^*)_*}$ consisting of axioms of $C^*$ algebras as operator space for $M$, tracial matrix-ordered operator spaces for $X$ and (28)-(31).
The subclass with $X$ such that $X^*$ is a type $III_\lambda$-factor  for a fixed $0<\lambda\leq 1$ is axiomatizable by the theory $T_{C^*,(III_\lambda)_*}$ consisting of $T_{C^*,(W^*)_*}$ and (32) with $d=2\frac{1-\sqrt{\lambda}}{1+\sqrt{\lambda}}.$
\end{theorem}

\begin{proof}
We explained how a couple gives a model of $T_{C^*,(W^*)_*}$. 
Moreover, a model homomorphism $(i,j):(M_1,X_1)\to (M_2,X_2)$, then $j^*:X_2^*\to X_1^*$ is a completely positive unital contraction and when restricted to $M_2$ satisfies for all $\phi\in X_2, c \in M_1$ and $j(c.\phi)=i(c).j(\phi)$, so that for $d\in X_2^*$ \[tr(j^*(d.i(c)).\phi)=tr(d.i(c).j(\phi))=tr(dj(m(c,\phi)))=tr(dj(m(c,\phi)))=tr((j^*(d).c).\phi)\] and thus, since $\phi$ arbitrary in $X_1$ , $j^*(d.i(c))=j^*(d).c$ in $X_1$ for $c\in M_1$ and similarly $j^*(i(c).d)=c.j^*(d)$.

 Consider $N$ the von Neumann algebra generated by $i(M_1)\subset X_2^*$ the restriction of $j^*$ satisfies $j^*(i(x).i(y))=x.(j^*(i(y)))=(j^*(i(x))).(j^*(i(y)))$ so that $j^*$ is a weak-* continuous $*$-homomorphism on $N$. Its image is a von Neumann subalgebra of $X_1^*$ containing $j^*(i(M_1))=M_1$ thus $X_1^*$ so that $j^*|_N:N\simeq X_1^*$ and its inverse $\alpha$ is thus a von Neumann algebra isomorphism extending $i$ and $\alpha\circ j^* $ is a completely positive weak-* continuous projection $X_2^*\to N\simeq X_1^*.$ Similarly, consider  any $*$-homomorphism $i:M_1\to M_2$ extending to a weak-* continuous $*$-homomorphism $X_1^*\to X_2^*$ onto a subalgebra having a normal conditional expectation from $X_2^*\to X_1^*$ with predual map $j$.  It gives rise to a model morphism since $j^*$ is completely positive tracial and the conditional expectation property implies $j^*(i(c).d.i(c'))=c.j^*(d).c'$ which implies $j(c.\phi.c')=i(c).j(\phi).i(c')$, $c,c'\in M_1,\phi\in X_1$. The remaining structure is easily preserved by $(i,j)$, for instance $j(|\phi|)$ satisfies the characterizing properties of $|j(\phi)|$ in (30).

Conversely, let $(M,X)$ the underlying sets for such a model, so that $M_n(M)$ are matrix $C^*$-algebras over $M$ with their norm and operator norm balls as domains of quantification and $X$ is tracial matrix-ordered operator space with its norm balls as domains of quantification. (29) gives a duality pairing and thus isometric embeddings $M\to X^*, X\to M^*.$ $M^*$ is the predual of the von Neumann algebra $M^{**}$ and using (28), for $x\in X, a,b\in M$ the duality pairing \[\langle ab, x\rangle=tr(m_1^l( ab,x))=\langle a, m_1^l( b,x)\rangle\]
and since $a$ is arbitrary we can replace it by density by $a\in M^{**}$ so that computed in $M^*$, the product $b.x=m_1^l( b,x)$ and similarly $x.b=m_1^r( x,b)$. Since $X$ is complete (as any model, which has complete balls, and for which metric balls are domains of quantifications) it is closed in $M^*$ and both a right and left ideal, thus \cite[Th III.2.7]{TakesakiBook}, $X= M^* e$ for a central projection in $M^{**}$ and thus $X^*=M^{**}e$ is a von Neumann subalgebra of $M^{**}$.
Note that  the previously isometric inclusion $M\subset X^*$ is a $*$-algebra homomorphism by definition of the $*$-algebra structure on $X^*$ induced from $M^{**}$. Indeed we have using (28) the relations for $a\in M, x\in X$  : \[\langle a^*, x\rangle=tr(m_1^l( a^*,x))=\overline{\langle a, x^*\rangle}\] defines $x^*\in M^*$ and reading backwards the formula define $a^*\in M^{**}$ thus in $X^*$.  Moreover the canonical map  $M\to M^{**}\to M^{**}e$ is clearly a $*$ algebra homomorphism since $e$ central projection, and for $x\in X=M^{*}e$, $\langle ae,x\rangle= \langle a,xe\rangle=\langle a,x\rangle$ so that the above map coincides with the original map $M\to X^*$ we defined, which is thus as expected an algebra homomorphism. Note also that if $a\in X^*\subset M^{**}$ we have by Goldstine lemma a net $a_n\to a$ weak-* in $M^{**}$ with $a_n\in M_{||a||}$, thus weak-* in $X^*$ since this topology is weaker. Thus $M$ is a weak-* dense subalgebra of $X^*$ as expected.

Let us finally show that the tracial matrix ordered structure of $X$, as well as the extra structure of the language, is the one induced as predual of $X^*$. This is the ``new" (but standard) part with respect to Groh's proof. The map $B$ is defined unambiguously by (30). If we write $B(\varphi,.)=\hat{\varphi}$ as before the map defined in (30) or \cite[p 16]{BrownOzawa} that converts $M_n(X)=NCB(X^*,M_n(\C))$ into maps in $(M_n(X^*))^*$.  Then the (in)equalities implied by (30) extend by strong-* density from $a\in M$ to $a\in X^*$ so that, 
by the characterization of \cite[Prop III.4.6]{TakesakiBook} $\widehat{i_n(|\varphi|)}=|\widehat{\varphi}|$ with the second absolute value computed in $(M_n(X^*))^*$. This determines $|\varphi|$ as the value we expected and moreover, by the bijection in \cite[Prop 1.5.14]{BrownOzawa}, $i_n(|\varphi|)$ is always  in $NCP(X^*,M_n(\C)).$ Thus since the image of $i_n(|.|)$ is exactly $i_n(\mathcal{C}_n),$ by the first relation in (30), one deduces that $i_n(\mathcal{C}_n),$ is exactly the set of completely positive maps and thus the expected positive cone of $X$ induced by the von Neumann algebra $X^*$. The fact that $tr$ is also the expected Haagerup trace is obvious. It remains  to  use (31) to identify the norms on $M_n(\mathcal{S}).$ The last formula implies the inequality for $y\in M_n(\mathcal{S})$,  $||y||_{cb}\leq d(y,0).$ Moreover, 
from the equation $||tr_n(\phi)||_{M_n(\C)}=||\phi||_{M_n(X)}$ in the theory of tracial matrix-normed operator spaces for $\phi\in \mathcal{C}_n$ we have $||\phi||_{cb}=d(y,0)$ in this case. Finally, we have from the first identity  for $y\in D_l(M_n(\mathcal(S)))$ (thus with $||y||_{cb}\leq l$)

\begin{align*}d(y,0)&= \inf_{x\in D_{5l}(\mathcal{C}_{2n})}\left[\max(||tr_n(\pi^{(n)}_{11}(x))||,||tr_n(\pi^{(n)}_{22}(x)||)+d(y,\pi^{(n)}_{12}(x))\right]\\&\leq \inf_{x\in \mathcal{C}_{2n}, ||x||_{cb}\leq 5l, \pi^{(n)}_{12}(x)=y}\left[\max(||tr_n(\pi^{(n)}_{11}(x))||,||tr_n(\pi^{(n)}_{22}(x)||)\right]=||y||_{dec}
\end{align*}
with the last equality coming from identification with $C_{2n}$ as the right cone of $CP$ maps (and of $d$ on positive elements) and thus by \cite[lemma 5.4.3]{EffrosRuan}, since $||y||_{dec}=||y||_{cb}$ this equals $d(y,0)$ and this concludes.
All  the extra structure is determined by the various equations.

Since the state space diameter is $2$ as soon we don't have a factor, the computation recalled in \cite[Theorem 6.6]{HaagerupAndo} characterizes $III_\lambda$ factors by infinite dimensionality and $d=2\frac{1-\sqrt{\lambda}}{1+\sqrt{\lambda}}.$ This concludes the second axiomatization.
\end{proof}
\subsection{The abstract axiomatization result for preduals of von Neumann algebras}

We now obtain a more abstract result.

\begin{theorem}\label{GrohThPredual}
The class of preduals of von Neumann algebras is axiomatizable in the language of tracial matrix-ordered operator spaces and so are the classes of preduals of $III_\lambda$-factors  for each fixed $0<\lambda\leq 1.$
\end{theorem}
\begin{proof}
We of course use the model theoretic result \cite[Prop 5.14]{BenYBHU}. Thus, it suffices to check those classes are stable by ultraproducts and ultraroots.

First, by \cite[Th 3.24]{HaagerupAndo} Banach space ultraproducts of preduals is a predual (the result is due to Groh for ultrapowers). We can actually deduce this result and the general ultraproduct case even in the sense of tracial matrix-ordered operator spaces from the previous theorem \ref{GrohTh}. Indeed, for any $X_n$ preduals, we associate a couple $(X_n^*,X_n)$ which gives a model $\mathcal{M}(X_n^*,X_n)$ of $T_{C^*,(W^*)_*}$. And so is the model-theoretic ultraproduct $\mathcal{M}(X_n^*,X_n)^\omega$ which has as second space the ultraproduct $(X_n)^\omega$ as tracial matrix-ordered operator space. This implies this is indeed a predual of a von Neumann algebra with the corresponding structure of tracial matrix-ordered operator space. The $III_\lambda$-factor case is similar.


It remains to check the class is stable by ultraroot. For let $(X,tr)$ be a tracial matrix-ordered operator space such that $((X^\omega)^*,1= tr^\omega)$ is a von Neumann algebra. Recall that by definition $M_n(X^\omega)=(M_n(X))^\omega$ (first as Banach space and then as operator space) and ultraproduct of positive cone is the cone of the ultraproduct. Obviously we have a completely isometric completely positive injection  $i:X^*\hookrightarrow (X^\omega)^*$ given by \[[i(\varphi)]((x_n)_\omega)=\lim_{n\to \omega}\varphi(x_n).\] Note that by definition $i(tr)=1.$ Moreover, we have a completely positive map $E:(X^\omega)^*\to X^*$ dual to the canonical injection and $E\circ i=id.$ Thus  $P=i\circ E:(X^\omega)^*\to (X^\omega)^*$ is a completely positive projection with image $i(X^*)\simeq X^*.$ Note also $E(1)=E(i(tr))=tr$ so that $P(1)=1.$ By the result of \cite{ChoiEffros} (see also \cite{EffrosStormer}), the image of $P$, i.e. $P((X^\omega)^*)=i(X^*)\simeq X^*$ becomes a $C^*$-algebra for the product $P(x).P(y)=P(xy)=P(P(x)P(y)).$ Thus $X^*$ is a $C^*$-algebra which is a dual of a Banach space, this is thus a von Neumann algebra. Note that $\varphi$ is positive in the $C^*$ structure if and only if there exists $h$ such that $\varphi=P(hh^*).$ This implies  for $x_n=x\in X$ positive  \[\varphi(x)=P(hh^*)(x)=\lim_{n\to \omega}[E(hh^*)](x)=[E(hh^*)](x)=(hh^*)(x_n)\geq 0\] since $(x_n)$ is by definition positive in the ultraproduct thus positive on positive elements of the $C^*$ algebra dual to this ultraproduct. Thus $\varphi$ is a positive element of $X^*$ for the duality (using also matrix variants to obtain complete positivity). Conversely, if $\varphi$ is such a positive element, then, for $(x_n)\in M_k(X^\omega)$ consisting of positive elements $i(\varphi)(x_n)=\lim_{n\to \omega}\varphi(x_n)$ is positive, thus $i(\varphi)$ is positive in $(X^\omega)^*$ thus of the form $hh^*$ and $i(\varphi)=P(i(\varphi))=P(hh^*)$ is positive in our $C^*$ algebra structure on $i(X^*)\simeq X^*$. Thus $i$ from $X^*$ with dual matrix-ordered norm structure and with unit $1=tr$ to $i(X^*)$ is a unital complete order isomorphism (complete since the reasoning above applies also on $M_n(X^*)$). Thus one deduces that the order structure and trace on $X$ are those as predual of the von Neumann algebra $X^*,$ as expected. Thus $(X^*,X)$ satisfies $T_{C^*,(W^*)_*}$. If its ultraproduct satisfies $T_{C^*,(III_\lambda)_*}$ as this is the case if $X^\omega$ is the predual of a $III_\lambda$ factor, so does $(X^*,X)$ and thus we also get stability by ultraroot of preduals of $III_\lambda$ factors for $\lambda\in ]0,1].$
\end{proof}

\subsection{The Groh-Haagerup-Raynaud theory for standard forms}
We write down in this short subsection the model theory for standard forms implicit in \cite{HaagerupAndo} and strongly based on the original works of \cite{Haagerup,Raynaud}. We thus call Groh-Haagerup-Raynaud theory the resulting theory. We add to the previous language for Groh theory the language of complex Hilbert spaces, namely the (complex variant) of \cite[section 15]{BenYBHU} and the following supplementary data :
\begin{itemize}
\item[$\bullet$]Unary function symbols $\Pi_P: D_{m}(\mathcal{H})\to D_m(\mathcal{H}), J: D_{m}(\mathcal{H})\to D_m(\mathcal{H})$ for the projection onto the positive cone $P$ of the standard form and the modular operator.
\item[$\bullet$]Unary function symbols $\xi: D_{m^2}(\mathcal{C}_1)\to D_m(\mathcal{H}),$ 
 $\omega: D_m(\mathcal{H})\to D_{m^2}(\mathcal{C}_1),$ for the two inverse bijections between normal states and the positive cone of the standard form. (we may also write $i_1\omega$ as $\omega$)
\item[$\bullet$]Binary function symbols $\omega_2: D_m(\mathcal{H})^2\to D_{m^2}(\mathcal{S})$ the bivariate variant of the previous one. 
\item[$\bullet$] Binary functions symbols : \[\pi:D_l(\mathcal{V})\times D_m(\mathcal{H})\to D_{ml}(\mathcal{H})\] interpreted as an action.
\end{itemize}
The expected uniform continuity modulus are obvious, $\Pi_P,J,\omega_2,\pi$ are contractions or are well-known as Powers-Störmer inequality ( see e.g. \cite[lemma 2.10]{Haagerup} or \cite[Th IX.1.2 (iv)]{TakesakiBook}) for $\xi,\omega$.
\[||\xi(\phi)-\xi(\psi)||\leq \sqrt{||\phi-\psi||}\]
\[||\omega(h)-\omega(k)||\leq||h+k||||h-k||\leq 2m||h-k||\]
if $h,k\in D_m(H).$ We need one more axiom :
\begin{enumerate}
\item[(33)] 
 $\pi$ bilinear, $\omega_2$ sesquilinear , $J=J^2$ antilinear. \[\pi(a,\pi(b,h))=\pi(ab,h),\ \pi(1,h)=h,\  \langle \pi(a,h),k\rangle = \langle h,\pi(a^*,k)\rangle,\] $\xi(\omega(\Pi_P(h)))=\Pi_P(h), \xi(\psi)=\Pi_P(\xi(\psi)),\omega(\xi(\psi))=\psi$, $\omega_2(h,h)=i_1(\omega(h)),$ \[ \Pi_P(\Pi_P(h))=h,\ \forall \lambda>0, \Pi_P(\lambda h)=\lambda \Pi_P(h),\ \ \Pi_P(\Pi_P(h)+\Pi_P(k))=\Pi_P(h)+\Pi_P(k),  J(\Pi_P(h))=\Pi_P(h)\] \[tr(m_1^l(a,\omega_2(\xi,\eta)))
 =\langle \pi(a,\xi), \eta\rangle\]
\[\Pi_P(\pi(a,J\pi(a,J(\Pi_P(h)))))=\pi(a,J\pi(a,J(\Pi_P(h)))),\]
 \[\sup_{(x,y)\in (D_m(\mathcal{H}))^2}\max(0,-\Re\langle\Pi_P(x),\Pi_P(y)\rangle)=0\]
\[\sup_{x\in D_m(\mathcal{H})}d(\Pi_P(x+J(x))-x-J(x),\Pi_P(\Pi_P(x+J(x))-x-J(x)))=0\]
 \begin{align*}&\sup_{x\in D_m(\mathcal{H})}\sup_{(a,b)\in D_{m}(V)} d(J\pi(b,J\pi(a,x)),\pi(a,J\pi(b,Jx))))=0\end{align*}
\begin{align*}\sup_{(x_1,...,x_n,y_1,...,y_n)\in D_m(\mathcal{H})^{2n}}&\max(0,\inf_{a\in D_{1}(V)} \sum_{i=1}^n||J\pi(a,Jx_i)-y_i||^2\\&-\sup_{b\in D_1(M_n(V))} \left(\sum_{i,j,k=1}^n\langle \pi(b_{ki},y_i),\pi(b_{kj},y_j)\rangle-\langle \pi(b_{ki},x_i),\pi(b_{kj},x_j)\rangle\right))=0\end{align*}
\end{enumerate}

\begin{theorem}\label{RaynaudTh}
The class of quintuples $(M,X,H,J,\Pi_P)$ of a weak-* dense $C^*$ algebra $M$ of $X^*$ and a predual of a von Neumann algebra $X$, and a standard form $(X^*,H,J,P)$ for $X^*$ acting on $H$, with modular conjugation $J$ and positive cone $P$ (given by its projection $\Pi_P$)  is  axiomatizable by the theory $T_{C^*,(W^*)_*,SF}$ consisting of $T_{C^*,(W^*)_*}$, the axioms of complex Hilbert spaces for $H$ and (33).
\end{theorem}
As in Groh theory, the morphisms for the category of axiomatization are the $C^*$ algebra homomorphism, inducing  normal homomorphisms of $X^*$, with image a von Neumann subalgebra having a conditional expectation onto it. We will see that this is enough to obtain a unique structure preserving morphism of the theory and we will especially obtain for free a standard form homomorphism (namely a supplementary Hilbert space isometry commuting with $J,\Pi_P$ and the action).
\begin{proof}
We first check that a standard form satisfies axiom (33). Most equations are obvious consequences of the definition. The bijectivity of $\xi,\omega$ comes e.g. from \cite[Th IX.1.2]{TakesakiBook}. The ante-penultimate equation stating a real vector can be decomposed as a sum of orthogonal positive and negative parts comes from \cite[lemma IX.1.7]{TakesakiBook}. The proof of the last equation comes from \cite[lemma 3.2, 3.1]{HaagerupWinslow} as in \cite[lemma 3.21]{HaagerupAndo}. Indeed take $x=(x_1,...,x_n)\in H^n,y=(y_1,...,y_n)\in H^n $ and let $\epsilon =\sup_{b\in D_1(M_n(V))} \left(\sum_{i,j,k=1}^n\langle \pi(b_{ki},y_i),\pi(b_{kj},y_j)\rangle-\langle \pi(b_{ki},x_i),\pi(b_{kj},x_j)\rangle\right))$ then for $a$ positive in $M_n(X^*)$ \[\omega_y(a)\leq \omega_x(a)+\epsilon ||a||\]
The quoted lemma 3.2 then implies there is $Y\in H^n$ with $||Y-y||\leq \sqrt{\epsilon}$ and $\omega_Y\leq \omega_x$ as state on $M_n(X^*)$. Thus by their quoted lemma 3.1, there is $c\in (M_n(X^*))'=(X^*)'I_n=JX^*JI_n$ with $c=JdJI_n$ with $||d||\leq 1$ such that $Y_i=JdJx_i$ and thus \[\sum_{i=1}^n||JdJx_i-y_i||^2\leq \epsilon.\]
By Kaplansky's density theorem giving strong-* density of the unit ball  $M$ in the unit ball of $X^*$, this concludes to the last equation in (33).

We already saw that a model morphism gives rise to the expected kind of morphism since preserving the structure of Groh theory is enough for that. Conversely, consider a $C^*$-algebra morphism extending to a normal $*$ homomorphism $i:X_1^*\to X_2^*$ with image a von Neumann subalgebra with conditional expectation $iE$ obtained from $E:X_2^*\to X_1^*$. To build the standard form morphism giving the structure preserving map, we use the uniqueness theorem for standard form \cite[Th IX.1.14]{TakesakiBook} and can assume $H_1=L^2(i(X_1^*),\phi), H_2=L^2(X_2^*,\phi\circ E)$ for some faithful semi-finite normal weight so that we have an isometric inclusion $u:H_1\to H_2$. From the criteria for existence of conditional expectations \cite[Th IX.4.2]{TakesakiBook}, the modular theory of $\phi$ in $H_1$ is computed by restriction of the one of $H_2$, so that $J_2u=uJ_1$ and $u(P_1)\subset P_2$ but from this and self-duality of the cones it is easy to see that $u(P_1)=P_2\cap H_1$ and thus $u$ commutes with the projections on the cones. Finally, since $\pi(i(x),u(y))=u(\pi(x,y))$ and $\omega_2(u(h),u(k))=\omega_2(h,k)\circ E$ by definition, one deduces $u$ preserves all the other data in the structure of $T_{C^*,(W^*)_*,SF}$ which is derived : $\omega_2(h,k)=\omega_2(u(h),u(k))\circ i$,  $u\xi_\omega=\xi_{\omega\circ E}.$ Moreover a standard form morphism $u$ as above part of a structure preserving morphism satisfy $u\xi_\omega=\xi_{\omega\circ E}.$ Thus $u$ is determined on the positive cone thus by linearity on the Hilbert space concluding to the bijection between structure preserving morphisms and the morphisms of the considered category.

Assume given a model $(M,X,H,J,\Pi_P)$ of $T_{C^*,(W^*)_*,SF}$. We already know $M\subset X^*$ is weak-* dense $C^*$-algebra of a von Neumann algebra and $H$ is an Hilbert space. First, for $f\in X^*$, $f(\omega_{2}(\xi,\eta))$ defines a sesquilinear map on $H$ and thus from Riesz representation 
\[f(\omega_{2}(\xi,\eta))=\langle A(f)\xi,\eta\rangle.\]
Moreover, if $f\in M\subset X^*,$ one obtains $A(f)\xi=\pi(f,\xi)$ and the linear $A$ is weak-* to weak operator topology continuous, thus weak-* continuous on bounded sets so that one extends the action property by weak-* density of $M$ in $X^*$, so that $A$ is a $*$-homomorphism. It is one-to-one since if $A(f)=0$, from the bijection $\omega,\xi$ between $Im(\Pi_P)$ and the positive cone of $X$, $f$ vanishes on this cone and $f=0$. Thus $A$ is a $*$-isomorphsim onto its image and is thus weak-* continuous \cite[Corol III.3.10]{TakesakiBook}. $\pi$ has thus been extended to a normal action of $X^*$.

From the equations in (33), the image of $\Pi_P$ is a convex cone, and since $\Pi_P$ contractive, it is closed.  To check it is self-dual, first note that we know $\langle\Pi_P(x),\Pi_P(y)\rangle$ is positive and conversely, consider if  $h\in H$ such that $\langle h, \Pi_P(x)\rangle \geq 0$ for all $x\in H$ and consider from the ante-penultimate formula in (33) $a=h+J(h),b=\Pi_P(a)-a$ with $\langle \Pi_P(a),\Pi_P(b)\rangle =\langle \Pi_P(a),\Pi_P(a)-a\rangle=0$ (since it is $\geq 0$ from the recalled positivity and since $0\in Im(\Pi_P)$ $\Re(\langle a-\Pi_P(a),0-\Pi_P(a)\rangle)= \langle a-\Pi_P(a),-\Pi_P(a)\rangle \leq 0$  from characterization of the projection) and $ h+J(h)=\Pi_P(a)-\Pi_P(b)$. Similarly one can get $c,d$ with $\langle \Pi_P(c),\Pi_P(d)\rangle=0, h-J(h)=i(\Pi_P(c)-\Pi_P(d))$ so that $2h =\Pi_P(a)-\Pi_P(b)+i(\Pi_P(c)-\Pi_P(d))$. But by assumption we have :$$\langle 2h, \Pi_P(x)\rangle=\langle \Pi_P(a), \Pi_P(x)\rangle-\langle \Pi_P(b), \Pi_P(x)\rangle-i\langle \Pi_P(c), \Pi_P(x)\rangle+i\langle \Pi_P(d), \Pi_P(x)\rangle\geq 0$$

Thus in taking the imaginary part $\langle \Pi_P(c), \Pi_P(x)\rangle-\langle \Pi_P(d), \Pi_P(x)\rangle=0, $ and from $x=c$ one gets $\Pi_P(c)=0$ and from $x=d$ one gets $\Pi_P(d)=0.$ Similarly from $x=b$ on gets $-\langle \Pi_P(b), \Pi_P(b)\rangle\geq 0$ and thus $\Pi_P(b)=0$ implying $2h=\Pi_P(a)$ so that the cone is indeed self dual.

By now we consider $X^*$ as a von Neumann algebra on $H$ via $\pi.$
From \cite[lemma 3.19]{HaagerupAndo}, in order to check we have a standard form $(X^*,H,J,Im(\Pi_P))$, it suffices to check for $x\in Im(\Pi_P)$ $J(x)=x$, which is contained in (33), $aJaJ(Im(\Pi_P))\subset Im(\Pi_P)$ which is also contained in the case $a\in M$ and extends to $a\in X^*$ by strong density of the unit ball of $M$ in the form of Kaplansky's density theorem and finally the key $JX^*J=(X^*)'.$ 

For that last statement, one uses the next-to-last equation in (33)  to see that $JMJ\subset M'=(X^*)'$  and thus by strong-* density  $J(X^*)J\subset (X^*)'$. Conversely, one uses the idea in the proof of \cite[Th 3.22]{HaagerupAndo}, let $a\in (X^*)'$ with $||a||\leq 1$ and show that $a\in (JMJ)''$. For take $x_1,...,x_n\in H$. It is well-known that $M_n(X^*)$ acts on $H^n$ with commutant $(X^*)' Id_n$.  Let $y_i=ax_i\in H.$ According to \cite[lemma 3.1]{HaagerupWinslow} if $x=(x_1,...,x_n)\in H^n, y=(y_1,...,y_n)$ and $\omega_z(.)=\langle .z,z\rangle$ the canonical state on $N'$ with $z\in\{x,y\}$ we have $\omega_y\leq \omega_x$ and this is relevant to apply the last equation in (33) since then \[\sup_{b\in D_1(M_n(V))} \left(\sum_{i,j}^n\langle (b^*b)_{ji}y_i,y_j\rangle-\langle (b^*b)_{ji}x_i,x_j\rangle\right)=0,\]

and thus by this equation \[\inf_{a\in D_{1}(M)} \sum_{i=1}^n||JaJx_i-y_i||^2=0.\]
Thus since $x_1,...,x_n$ are arbitrary, $(X^*)'$ is in the strong operator topology  closure of $JMJ$ i.e. on gets the claimed  $ (X^*)'\subset J(X^*)J.$  Finally, all the data of the model is determined as the expected data since $J,\pi,\Pi_P$ are part of a standard form data, $\omega_2$ is then determined by $\pi$ and determines $\omega$ which determines $\xi$ as its inverse.
\end{proof}

\section{The Ando-Haagerup theory gathering Groh and Ocneanu theories}
We now describe a theory for a $C^*$ algebra $C$ weak-* dense in $X^*$, von Neumann algebra with predual $X$, and a $\sigma$-finite von Neumann algebra $M\simeq eX^*e$ for $e$ the support projection of $\phi\in X,$ thus inducing a faithful  state $\varphi$ on $M$. $(C,X)$ is described by the Groh theory of the previous section with corresponding language and even we will rather use for convenience the language of the Groh-Haagerup-Raynaud theory, $(M,\varphi)$ is described by the Ocneanu theory of section \ref{Ocneanu}. This is the kind of setting Haagerup and Ando used in \cite{HaagerupAndo} to relate the Ocneanu and Groh-Raynaud ultraproducts. Even if we also use Raynaud's viewpoint using standard forms instead of Groh's viewpoint, our axiomatization will look quite different from their proofs, contrary to our axiomatization of the Groh theory that was similar and a generalization of Groh's argument in a better setting.  To correct axiom (38) from a previous preprint that contained a meaningless copy-pasted formula from section \ref{Ocneanu}, we think it was more convenient to use the standard form framework to get a canonical way of computing $\Delta^{1/2}$ in $(eX^*e,\phi)$ to state in a readable enough way it agrees with the expected one for $(M,\varphi)$.

To have a model associated to $(C,X,H,J,P,\phi,M)$ in the setting above, we now introduce the following supplementary data in the language.
\begin{itemize}
\item[$\bullet$] A unary function symbol $P:D_n(V)\to D_n(U)$ for $e.e:C\to eX^*e\simeq M$
\item[$\bullet$] A constant $\phi\in \mathcal{C}_1$ for a state.  
 \item[$\bullet$] Binary Relation Symbols 
  $\mathscr{E}_{P,\beta,r},\mathscr{E}_{P,\beta} :(D_m(V))^2\to \C$ for $\beta\in]0,1[\cap \Q,r\in \Q$ 
  meaning \[\E_{P,\beta}(x,y)=\E_{\beta}(P(x),P(y)),
  \E_{P,\beta,r}(x,y)=\E_{\beta}(G_r(P(x)),P(y)).\]
  \item[$\bullet$] Binary  Relation Symbols $\mathscr{E}_{\beta,N,\infty}:(D_m(U))^2\to \C$ for $\beta\in\Q\cap[0,1[,N\in \N$ 
  meaning $\E_{\beta}(F_N^\varphi(.),.)$
  \item[$\bullet$] A Ternary  Relation Symbol $\mathscr{E}_{0,N,\infty,M}:(D_m(U))^3\to \C$ for 
  meaning $\varphi(F_N^\varphi(.)^*.F_M^\varphi(.)^*)$, $\mathscr{E}_{0,\infty,M}(.,.)=\mathscr{E}_{0,1,\infty,M}(1,.,.)$.
\end{itemize}

Of course, the last relation symbols on $D_m(U)$ could have been introduced in the Ocneanu theory from the very beginning but it will be crucial only in this section.  Note that the uniform continuity in the first variable is obvious, and the one in the second variable is obtained as the one for $\varphi$ in noting that for any $x=y+z, t$ and using the same bounds from \cite[lemma 4.13]{HaagerupAndo} as in lemma \ref{ContinuityProduct} :
\[|\mathscr{E}_{\beta}(F_N^\varphi(t),x)|\leq ||\sigma_{-i\beta}^\varphi(F_N^\varphi(t))||_\varphi||y||_\varphi+||(\sigma_{-i\beta}^\varphi(F_N^\varphi(t))z)^*||_\varphi
\leq 3\sqrt{2}e^{(2\beta+1) N}||t||\sqrt{||y||_\varphi^2+||z^*||_\varphi^2},\]
\[|\mathscr{E}_{\beta}(F_N^\varphi(x),t)|\leq e^{\beta N}||(F_N^\varphi(x))||_\varphi||t||_\varphi
\leq e^{(\beta+1/2) N}(1+e^N)^{1/2}||x||_\varphi^*||t||_\varphi,\]

since from spectral theory and the proof of lemma \ref{NormG}, we used the fact that we have :\[||(F_N^\varphi(x))||_\varphi=||\Delta^{-1/2}(1+\Delta)^{1/2}\Delta^{1/2}(1+\Delta)^{-1/2}(F_N^\varphi(x))\xi_\varphi||\leq 
e^{N/2}(1+e^N)^{1/2}||(F_N^\varphi(x))||_\varphi^*\]
We are now ready to introduce our supplementary axioms :

\begin{enumerate}
\item[(34)]For any $N,M\in\N^*,\beta\in\Q\cap[0,1],$ $\mathscr{E}_{\beta,N,\infty}(y,F_M(x))=\mathscr{E}_{\beta,N,M}(y,x),$
\item[(35)]$P(1)=1,P(x^*)=(P(x))^*,$ $tr(m^l_1(a,\phi))=\varphi(P(a)),$
\[\sup_{a\in D_n(V)}\max(0,d(P(a),0)^2- tr(m^l_1(aa^*,\phi)))=0.\]
\item[(36)]For any $l,m,N\in\N^*,$
\[\sup_{(x,z)\in (D_l(V))^2}\sup_{y\in D_m(U)}\max(0,\mathscr{E}_{0,N,\infty}(y,P(x))+\overline{\mathscr{E}_{0,N,\infty}(y^*,P(z^*))}-m d(m^l_1(x,\phi)+m^r_1(\phi,z),0)=0,\]
\item[(37)]For any $m,N\in\N^*,$
\[\sup_{y\in D_m(U)}\inf_{z\in D_m(V)}\sup_{x\in D_1(V)}\max(|\mathscr{E}_{0,N,\infty}(y,P(x))-tr(m_1^l(zx,\phi)|,|\mathscr{E}_{0,N,\infty}(y^*,P(x))-tr(m_1^l(z^*x,\phi)|)=0,\]
\begin{align*}&\sup_{(y,Y)\in (D_m(U))^2}\inf_{(z,Z)\in (D_m(V))^2}\sup_{x\in D_1(V)}\\&\max\left(|\mathscr{E}_{0,N,\infty}(y^*,P(x))-tr(m_1^l(z^*x,\phi)|,|\mathscr{E}_{0,M,\infty}(Y,P(x))-tr(m_1^l(Zx,\phi)|,\right.
\\&\ \ \ \ \ \ \ \ \ \ \ |(NM+2)\mathscr{E}_{0,NM+2,\infty}(m_{N,M}(y,Y),P(x))\\&\ \ \ \ \ \ \ \ \ \ \left.-(NM+1)\mathscr{E}_{0,NM+1,\infty}(m_{N,M}(y,Y),P(x))-tr(m_1^l(Zzx,\phi)|
\right)=0,\end{align*}
\item[(38)]For any $l,m,N\in\N^*,\beta\in \Q\cap[0,1[,r\in\Q$
\begin{align*}\sup_{x\in (D_l(V))}\sup_{y\in (D_m(V))} \max(&0,|\E_{P,\beta,r}(x,y)-\mathscr{E}_{\beta,N,\infty}(G_r(P(x)),P(y))|\\&-4me^{-r/2}(e^r+|e^r-1|)d(F_N(P(x)),P(x)) )=0\end{align*}
For $\alpha,\beta\in \Q\cap]0,1[,0<\alpha<1/2,\alpha+\beta<1,\epsilon= \min(1/2-\alpha,1-\beta-\alpha),\delta= \min(\epsilon,\alpha) ,m,K,L,n\in\N^*$  \begin{align*}\sup_{(x,y)\in D_{m}^2}\max(0&,\left|\E_{P,\alpha+\beta}(x,y)-\frac{1}{n^2}\frac{\cos(\alpha \pi)}{2\pi}\sum_{k=-n^3}^{n^3-1}e^{\alpha k/n^2}\E_{P,\beta,k/n^2}(x,y)\right|\\&- \frac{4e^{-n\delta}m^2}{\pi\delta}-|(1-e^{1/n^2})|\frac{2(3+e^{\epsilon/n^2})m^2}{\pi\delta})=0 \end{align*}
For any $n,m,N,M,N_i,N_{i,j},L,l,K\in\N^*,\lambda_i,\lambda_{i,j}\in \Q\cap[0,1],u\in \Q, u>0$ with $\sum_{i=1}^n\lambda_i=1$
\[\sup_{x\in (D_m(U))}\sup_{(z_1,...,z_n)\in (D_L(U))^n}\inf_{y\in (D_m(V))}\max_{j=1,...,n} |\sum_{i=1}^K\lambda_{i,j}[\E_{0,N_{i,j},M}(z_j,x)-\E_{0,N_{i,j},\infty}(z_{j},P(y))]|=0,\]
\begin{align*}\sup_{x\in (D_m(V))}\max&\left(0,
\inf_{z\in (D_l(V))}\Re(tr(\|\pi(z,\xi_\phi)\|^2+\frac{1}{u}(\langle\pi((z-x)^*,\xi_\phi) ,J(\pi(z-x,\xi_\phi))\rangle-\frac{1}{u}\E_{P,1/2}(x,x)\right.\\&-\inf_{Z\in (D_l(U))}\left[\sum_{i=1}^n\lambda_i\sum_{j=1}^n\lambda_j(\E_{0,N_i,N_j}(Z,Z)+\frac{1}{u}\E_{1/2,N_i,N_j}(Z,Z))\right.\\&\left.\left.+\frac{1}{u}\sum_{i=1}^n\lambda_i(\E_{1/2,\infty,N_i}(P(x),Z)+\E_{1/2,N_i,\infty}(Z,P(x)))\right]\right)=0\end{align*}

\end{enumerate}

The two most technical conditions (37)-(38) will be explained in the proof bellow. (37) will enable us to check our expected $j:M\to eX^*e$ is a $*$-homomorphism. (38) first gives the weak-* density of $P(C)$ in $M$ and will be related to the surjectivity of $j$.

\begin{theorem}\label{AH}
The class of septuples $(C,X,H,J,\Pi_P,\phi,M)$ of a weak-* dense $C^*$ algebra $C$ of $X^*$ and a predual of a von Neumann algebra $X\ni \phi$ a state and a von Neumann algebra $M\simeq eX^*e$ for $e$ the support projection of $\phi$ with $(X^*,H,J,Im(\Pi_P))$ a standard form as above is axiomatizable by the theory $T_{AHW^*}$ consisting $T_{C^*,(W^*)_*,SF}$ for $(C,X,H,J,\Pi_P),$  $T_{\sigma W^*}$ for $M$, and  (34)-(38).
\end{theorem}

\begin{proof}
We already know that an element in the class produces a model except for the verification of (37)-(38).

Seeing $F_N(y)\in eX^*e$ and $P(x)=exe$ we have \[\mathscr{E}_{0,N,\infty}(y,P(x))=\phi(F_N(y^*)exe)=\phi(F_N(y^*)x)=\langle F_N(y)\xi_\phi, x\xi_\phi\rangle .\] But we know that $C$ is strong-* dense in $X^*$ thus there is a net $a_n\in C$ bounded in $D_m(C)$ (for $y\in D_m(M)$) with $a_n\to F_N(y^*), a_n^*\to F_N(y)$ strongly, thus $\sup_{x\in D_1(V)}|\mathscr{E}_{0,N,\infty}(y,P(x))-\langle a_n^*\xi_\phi, x\xi_\phi\rangle | \leq || (F_N(y)-a_n^*)\xi_\phi||\to 0,$ giving the first statement in (37) about the  infemum \[\inf_{z\in D_m(V)}\sup_{x\in D_1(V)}\max(|\mathscr{E}_{0,N,\infty}(y,P(x))-tr(m_1^l(zx,\phi)|,|\mathscr{E}_{0,N,\infty}(y^*,P(x))-tr(m_1^l(z^*x,\phi)|)=0.\]

Similarly, one takes $b_n\to F_M(Y^*)$ strongly so that 
\begin{align*}\langle a_n^*b_n^*\xi_\phi, x\xi_\phi\rangle&\to\langle F_N(y)F_M(Y)\xi_\phi, x\xi_\phi\rangle= \phi(F_M(Y^*)F_N(y^*)x)\\&=(NM+2)\mathscr{E}_{0,NM+2,\infty}(m_{N,M}(y,Y),P(x))-(NM+1)\mathscr{E}_{0,NM+1,\infty}(m_{N,M}(y,Y),P(x))
\end{align*} and the convergence can even be uniform in $x$ giving the last equality in (37).

For the first inequality in (38),  using $\Delta^\beta\leq 1+\Delta$, the resolvent equation $G_r-G_0e^{r/2}=e^{r/2}(1-e^r)G_0(\Delta+e^r)^{-1}$, $||(\Delta+e^r)^{-1}||\leq e^{-r}$, $||P(y)||_\varphi^\#\leq 2||y||\leq 2m$ and lemma \ref{NormG}, one gets :
\begin{align*}|\E_{P,\beta,r}(x,y)-\mathscr{E}_{\beta,N,\infty}(G_r(P(x)),P(y))|&=|\mathscr{E}_{\beta}(G_r(P(x))-F_N(G_r(P(x))),P(y))|\\&\leq || G_r(P(x))-F_N(G_r(P(x)))||_\varphi^\# ||P(y)||_\varphi^\#\\&\leq 4me^{-r/2}(e^r+|e^r-1|)d(F_N(P(x))-P(x))),0) \end{align*}

The second equation in (38) is a substitute to (19) with $F_L,F_K$ replaced by $P$. Note we could not have done the same with (20), {this is the main correction from a previous preprint version.}

The third statement in (38) comes from the weak-* density of $P(C)=eCe$ in $M=eX^*e$. The last inequality is more technical. 
Recall $e$ denotes the support projection of $\phi$.
First note the following computation (where we write for short $\pi(y,\xi_\phi)=y.\xi_\phi$ and use commutation of $JyJ$ with $e\in X^*, J^2=J,J\xi_\phi=\xi_\phi=e\xi_\phi$):
\begin{equation}\label{Inserte}\langle z.\xi_\phi ,J(y.\xi_\phi)\rangle=\langle z.\xi_\phi ,JyJe\xi_\phi\rangle=\langle ez.\xi_\phi ,Jy\xi_\phi\rangle=\langle y.\xi_\phi ,Jez\xi_\phi\rangle=\langle ey.\xi_\phi ,Jez\xi_\phi\rangle=\langle z.\xi_\phi ,eJe(y.\xi_\phi)\rangle\end{equation}
We deduce that considering the restriction of the inf on $z$ to those elements of the form $eze$ and since the inf over $D_l(V)$ is the same as  one over $D_l(X^*)$ by strong-$*$ continuity of the expression and density. 

\begin{align*}&\inf_{z\in (D_l(V))}\Re(tr(\|\pi(ez,\xi_\phi)\|^2+\frac{1}{u}(\langle\pi((z-x)^*,\xi_\phi) ,J(\pi(z-x,\xi_\phi))\rangle\\&\geq
\inf_{z\in (D_l(V))}\Re(tr(\|\pi(z,\xi_\phi)\|^2+\frac{1}{u}(\langle\pi((z-x)^*,\xi_\phi) ,J(\pi(z-x,\xi_\phi))\rangle\end{align*}
Now we express this formula in terms of $\varphi$ on $eX^*e\simeq M$. First note for instance that $\langle\pi((ze)^*,\xi_\phi) ,J(\pi(ez,\xi_\phi))\rangle=\langle J_\phi\Delta_\phi^{1/2}(P(z)\xi_\phi) ,eJe(P(z)\xi_\phi)\rangle=||\Delta_\phi^{1/4}(P(z)\xi_\phi)||^2$ since (using commutation of $e$, $JeJ$ and $J^2=1$) $(eJeJ)J(eJeJ)=eJeJeJ=JeJeJ=eJeJJ=eJe=J_\phi$ for $e$ the support projection of $\phi$ by \cite[lemmas 2.6,2.9]{Haagerup} and using the isomorphism of his corollary 2.5 and the previous computation to replace $e$ by $q=eJeJ$ $\langle P(z^*)\xi_\phi) ,eJe(P(z)\xi_\phi)\rangle=\langle qz^*q\xi_\phi ,eJeqzq\xi_\phi\rangle$ (one also uses $\xi_\phi=q\xi_\phi$ cyclic and separating on $qX^*q$ acting on $qH$), so that one gets using \eqref{Inserte}: 
\begin{align*}&\Re(tr(\|\pi(ez,\xi_\phi)\|^2+\frac{1}{u}(\langle\pi((z-x)^*,\xi_\phi) ,J(\pi(z-x,\xi_\phi))\rangle\\&=\Re(tr(\|\pi(ez,\xi_\phi)\|^2+\frac{1}{u}(\langle\pi(e(z-x)^*,\xi_\phi) ,J(\pi(ez-ex,\xi_\phi))\rangle\\&=\Re\left(\varphi(P(z^*)P(z))+ \frac{1}{u}\langle \Delta^{1/2}_\varphi P(z-x)\xi_\varphi,P(z-x)\xi_\varphi\rangle\right).\end{align*}

By strong-* density of $P(C)$ the infemum over $P(z),z\in D_l(V)$ is the same with $Z=P(z)$ replaced by $Z\in D_l(U)$ and one gets a huger infemum if one restricts to convex combinations of the form $\sum_{i=1}^n\lambda_i F_{N_i}(Z),Z\in D_l(U)$. In that form, this gives the inequality at the source of the second statement in (38), namely for $||x||\leq m$:
\begin{align*}\inf_{z\in (D_l(V))}&\Re(tr(\|\pi(z,\xi_\phi)\|^2+\frac{1}{u}(\langle\pi((z-x)^*,\xi_\phi) ,J(\pi(z-x,\xi_\phi))\rangle\leq\frac{1}{u}\E_{1/2,P}(x,x)+\\&\inf_{Z\in (D_l(U))}\left[\sum_{i=1}^n\lambda_i\sum_{j=1}^n\lambda_j(\E_{0,N_i,N_j}(Z,Z)+\frac{1}{u}\E_{1/2,N_i,N_j}(Z,Z))\right.\\&\left.+\frac{1}{u}\sum_{i=1}^n\lambda_i(\E_{1/2,\infty,N_i}(P(x),Z)+\E_{1/2,N_i,\infty}(Z,P(x)))\right]\end{align*}

One gets exactly the last formula in (38).

Conversely, take $(C,X,H,J,\Pi_P)$ model of the Groh-Haagerup-Raynaud theory and $M$ model of the Ocneanu theory. Since we checked $d_U(F_N(x),x)\to 0$ we indeed obtain from (34) the expected $\E_{\beta,M,\infty}(y,x)=\E_{\beta}(F_N^\varphi(y),x).$
Note also that the inequality in (35) implies that $P$ is uniformly continuous from the strong operator topology on $C$ induced from $X^*$ to the topology of the metric $d$. Thus, by 
\cite[\S 5.4.(4)]{Kothe}, it extends to a uniformly continuous map we still call $P:X^*\to M.$ Moreover, this extension is weak-* continuous by standard relation to the topology of $d$ and the strong topology.

Consider the (norm) closures $L,V$ of the spaces spanned respectively by $m^l_1(x,\phi)+m^r_1(\phi,z)\in X, x,z\in C, $ and  $m^l_1(x,\phi)\in X, x\in C$. It is known from the proof of \cite[lemma III.3.6]{TakesakiBook} that $V$ is the smallest left invariant (by $C$, thus $X^*$ by weak-* density and usual applications of Hahn-Banach to identify weak-closures and norm closures) subspace containing $\phi$, which is $Xe$ for $e$ the support projection of $\phi$. Moreover $L=V+V^*=Xe+eX$ so that the dual $L^*=eX^*e.$ The  equality in (36) then states that the map :\[i(F_N^\varphi(y^*)):m^l_1(x,\phi)+m^r_1(\phi,z)\mapsto \varphi(F_N^\varphi(y^*)P(x))+\varphi(P(z)F_N^\varphi(y^*))\] extends to a continuous linear form on $L$ giving a contractive linear map $i:A\to eX^*e$ for the operator norm (since, recall $A=Vect(F_N(x), x\in M, N\in \N^*)$). Note that $i(1)=e$ and also that $i$ is uniformly continuous on bounded sets from the strong-* topology (given by the metric $||.||_\varphi^{\#}$) to the weak-* topology on $eX^*e$. 
By \cite[\S 5.4.(4)]{Kothe} again and Kaplansky's density theorem, it extends to a uniformly continuous map \[j:M\to eX^*e\] which is especially weak-* continuous (by standard compactness and Hahn-Banach arguments
).
Notice that the interpretation of terms in the two first equations in (37) that we want to use : \[(\phi.i(F_N(y^*)))(x)+(i(F_N(y^*).\phi))(z)=(x.\phi+\phi.z)(i(F_N(y^*)))=\varphi(F_N^\varphi(y^*)P(x)+P(z)F_N^\varphi(y^*))\] and 
\begin{align*}&(NM+2)\mathscr{E}_{0,NM+2,\infty}(m_{N,M}(y,Y),P(x))-(NM+1)\mathscr{E}_{0,NM+1,\infty}(m_{N,M}(y,Y),P(x))\\&=(\phi.i(H_{NH+1}(F_M(Y^*).F_N(y^*))))(x)=(\phi.i(F_M(Y^*).F_N(y^*)))(x).\end{align*}

Now, we can use the first equation in (37) which says exactly that for $y\in A$, there is a sequence $z_n\in C$ such that $||\phi.z_n-\phi.(i(y))||_X,||\phi.z_n^*-\phi.(i(y^*))||_X\to 0.$ Taking a subnet such that $z_n\to z$ ultraweakly in $X^*$ which implies $\phi.z_n$ converges weakly in $X$ to $\phi.z$, one gets using the first equation we just noticed :\[\phi.z=\phi.(i(y)),\ \ \ \  \phi.z^*=\phi.(i(y^*))\] and thus $ez=i(y),e.z^*=i(y^*).$ Thus, extending from $y\in A$ to $x\in M$ the relation $[i(y)e]^*=[eze]^*=i(y^*)e$, we have obtained \[[j(x^*)]=[j(x^*)]e=e[j(x)]^*=[j(x)]^*.\]

Similarly, the second equation in (37) says that for $y,Y\in A$, there are sequences $z_n,Z_n\in C$ such that $||\phi.z_n^*-\phi.(i(y^*))||_X\to 0,||\phi.Z_n-\phi.(i(Y))||_X\to 0, ||\phi.(Z_nz_n)-\phi.(i(Y.y))||_X\to 0.$ Arguing as before, one gets for limit points $z,Z$ of $z_n,Z_n$ in $X^*$, \[ez^*=i(y^*), \ \ \  eZ=i(Y),   \ \ \ \  eZz=i(Y.y)\]
 and thus $i(Y.y)e=eZze=i(Y)[i(y^*)]^*$ so that extending by ultraweak continuity (separately in $y,Y$):
 \[j(Y.y)= j(Y.y)e=j(Y)[j(y^*)]^*=j(Y)j(y)\]
and $j$ is thus a $*$-homomorphism.

 Note that making $N\to \infty$ in the relation $(x.\phi+\phi.z)(i(F_N(y^*)))=\varphi(F_N^\varphi(y^*)P(x)+P(z)F_N^\varphi(y^*))$
one gets:
\[(x.\phi+\phi.z)(j(y))=\varphi(yP(x)+P(z)y)\]

Thus $\phi(j(y))=\varphi(y)$ 
and the isometry relation $\phi(j(y)^*j(y))=||y||_\varphi^2$ implying that $j$ is one-to-one (since $\varphi$ faithful on $M$).

Moreover $\phi(j(P(a)))=\varphi(P(a))=\phi(a)$ by the last equality in (35) for $a\in C$. Note also that for $x\in C, y\in M$ $\phi(xj(y))=\varphi(P(x)y)$ and this extends to $x\in X$, since $P$ is strong to metric continuous on bounded set by (35), thus weak-* continuous. 
The relation thus extends to $x\in X^*$ thus \[\varphi(P(j(P(x)))y)=\phi(j(P(x))j(y))=\phi(j(P(x)y))=\varphi(P(x)y),\] 
and since $\varphi$ is faithful on $M$, $P\circ j\circ P=P$ on $X^*$ and thus $j\circ P$ is a projection on $X^*$ to $j(M).$

As a consequence we also deduce :\begin{align}\label{IneqPhiVarphi}\varphi(P(x^*)P(x))=\phi(x^*j(P(x)))\leq\|exe\xi_\phi\|\|j(P(x))\xi_\phi\|=\|exe\xi_\phi\|\|P(x)\xi_\varphi\|\end{align}

Note also that this relation implies $j(1)=e$ since we know it is in $eX^*e$ and it has the expected formula for $e$ on the predual of $eX^*e.$ As a consequence since from (35) $P(1)=1$, $P(e)=P\circ j\circ P(1)=P(1)=1.$

Let us finally check that $j$ is onto : $j(M)=eX^*e$. 
The two first equation in (38) determine uniquely $\mathscr{E}_{P,\beta,r},\mathscr{E}_{P,\beta}$ so that we will be able to use $\mathscr{E}_{P,1/2}$ in the last equation with its expected interpretation.
 
 One uses the third equation in (38) to note that $P(C)$ is weak-* dense in $M$ and take $x\in D_m(U)$, $\epsilon>0$ and $z_1,...,z_n\in M$, and find convex combinations, from the proof of Theorem \ref{OcneanuTh}, with $\|\sum_{i=1}^K\lambda_{i,j}F_{N_{i,j}}(z_j)-z_j\|_\varphi \leq \epsilon/4m$, then find $y\in D_m(U)$ given by the first equation in (38) such that for all $j=1,...,n$: $|\varphi((\sum_{i=1}^K\lambda_{i,j}F_{N_{i,j}}(z_j))^*(F_M(x)-P(y)))|\leq \epsilon/2$ and thus \[|\varphi[z_j^*(F_M(x)-P(y))] |\leq 2m\epsilon/4m+\epsilon/2=\epsilon\]
Since $M$ is dense in $L^1(M,\varphi)$ this concludes to the stated weak-* density of $P(C)\subset M$. Note that as a consequence $j\circ P$ is a state preserving norm 1 projection from $(eXe,\phi)$ onto $(j(M),\varphi\circ j^{-1}).$

We can now use the last equation of (38) and recall that we already computed the first infemum in a more explicit form  and after taking an infemum over $l$:
\begin{align*}&\inf_{z\in C}\Re(tr(\|\pi(z,\xi_\phi)\|^2+\frac{1}{u}(\langle\pi((z-x)^*,\xi_\phi) ,J(\pi(z-x,\xi_\phi))\rangle\\&=\inf_{z\in C}\Re(tr(\|\pi(ez,\xi_\phi)\|^2+\frac{1}{u}(\langle\pi(e(z-x)^*,\xi_\phi) ,J(\pi(e(z-x),\xi_\phi))\rangle\\&=\inf_{z\in X^*}\Re\left(\|\pi(eze,\xi_\phi)\|^2+ \frac{1}{u}\langle \Delta^{1/2}_\phi e(z-x)\xi_\phi,e(z-x)\xi_\phi\rangle\right)\\&=||\Delta_\phi^{1/4}(u+\Delta_\phi^{1/2})^{-1/2}(exe)\xi_\phi||^2\end{align*}
with the last equation coming from a variant of lemma \ref{NormG} for $\Delta^{1/2}$ instead of $\Delta$.

Reasoning similarly with the right hand side and approximating $Z$ by convex combinations of $F_{N_i}(Z)$, one gets from (38):
\[||\Delta_\phi^{1/4}(u+\Delta_\phi^{1/2})^{-1/2}(exe)\xi_\phi||^2\leq ||\Delta_\varphi^{1/4}(u+\Delta_\varphi^{1/2})^{-1/2}(P(x))\xi_\varphi||^2.\]
Using the variant of lemma \ref{EquationForms} with again $\Delta^{1/2}$ instead of $\Delta$, for $\alpha\in]0,1/2[$, $\Delta^{1/2-\alpha}=\frac{\sin(2\alpha\pi)}{\pi}\int_0^\infty u^{-2\alpha}\Delta^{1/2}(u+\Delta^{1/2})^{-1}$ so that integrating the previous inequality and letting $\alpha\to 1/2$, one gets:
\[||exe\xi_\phi||^2\leq ||P(x)\xi_\varphi||^2.\]
Since the converse was already obtained in \eqref{IneqPhiVarphi}, one gets equality and using the various equations already obtained for $j,P:$
\[||(j(P(x))-exe)\xi_\phi||^2=||j(P(x))\xi_\phi||^2+||exe\xi_\phi||^2-2\Re\phi(x^*j(P(x)))=||exe\xi_\phi||^2-||P(x)\xi_\varphi||^2=0,\]
so that $j(P(x))=exe$ for any $x\in X^*$, thus $eX^*e\subset j(M)$ as expected implying that $j:M\to eX^*e$ is onto.
Finally, via this isomorphism, $P(x)=exe$ as expected, and this defines $P$ uniquely. This concludes the identification of the structures.

\end{proof}

As a consequence, one obtains the stabillity by ultraproduct from \cite[Th 6.11]{HaagerupAndo} in the case of a general non-principal ultrafilter $\omega$ not necessarily on $\N$. We already used it in section 1 to obtain the existence of a first order axiomatization of  $III_\lambda$ factors.

\begin{corollary}\label{AHlambda}
Let $\lambda\neq 0$ fixed. Let $M_n$ be $\sigma$-finite factors of type $III_\lambda$,  and consider faithful normal states $\varphi_n$, then $(M_n,\varphi_n)^\omega$ is also a factor of type $III_\lambda$.
\end{corollary}

\begin{proof}
$(M_n,(M_n)_*,L^2(M,\varphi_n), J, \Pi_P,\varphi_n,M_n)$ satisfies $T_{AHW^*}$ and moreover $(M_n,(M_n)_*)$ satisfies $T_{C^*,(III_\lambda)_*}$ thus so does $((M_n)^\omega,((M_n)_*)^\omega,L^2(M,\varphi_n)^\omega, J^\omega, \Pi_P^\omega,\varphi_n^\omega,(M_n,\varphi_n)^\omega)$ (with the first three spaces being the Banach space ultraproducts and the last one being the Ocneanu ultraproduct). Thus $(M_n,\varphi_n)^\omega\simeq p[((M_n)_*)^\omega]^*p$ is a corner of a $III_\lambda$ factor and thus especially, a $III_\lambda$ factor.
\end{proof}

\section{Axiomatization of cross-products appearing in Connes' description of type $III_0$ factors}

Finally, it seems interesting to axiomatize the crossed-product decomposition of $III_0$-factors from \cite[Th 5.3.1]{ConnesThesis}. We won't axiomatize in this way $III_0$ factors that are not axiomatizable by the above cited result of \cite{HaagerupAndo}, but we will obtain a class useful to study them. This section will thus mostly contain a model theory version of the proof of their proposition 6.23. It gives in corollary \ref{AHzero} the missing non-stability by ultraproduct needed in the proof of axiomatizability of type $III_\lambda$ factors, for fixed $\lambda>0.$ This is the most technical work hidden in this natural result. This result could also probably have been written in operator algebraic style as in  \cite{HaagerupAndo}, but having an explicit axiomatization will be crucial for later model theoretic investigation of $III_0$ factors and thus deserves a first model theoretic study here. As for $II_1,II_\infty$ factors, we don't expect axiomatizability without fixing the state, matrix unit and extra data in the language.
 
We want to axiomatize cross products of $\Z$ (with implementing unitary $U$) by a $II_\infty$ von Neumann algebra $N$ that will be the centralizer for a lacunary normal faithful semifinite weight $\psi$ (thus with spectrum included in $(\R-]\log(\lambda_0),-\log(\lambda_0)[)\cup\{0 \}$ for some $\lambda_0\in]0,1[$). We will use a matrix unit $(w_{i,j})_{i,j\geq 0}$ for $B(H)$ with some geometric state $\varphi$ that will be our basic state in our previous language. We will model the semifinite weight by Connes cocycle derivative with respect to our faithful state (see e.g. \cite{TakesakiBook}) and various compression by well chosen projections related to the matrix unit. 
We will also need some of the modular theory  for $\psi$.

Unfortunately, saying that $U$ is a unitary is not obvious in our non-tracial setting, and we will thus use what we called before Ando-Haagerup theory, to have the $C^*$-algebra in Groh's theory to require $u$ in it (or rather $u=eUe+1-e$ or any other unitary projecting to $eue=U$ and the theory will depend on the choice of this unitary in the $C^*$-algebra). 
This is the key part restricting the choice of the language and which is crucial to get some stability results by ultraproducts and not only ultrapowers as in \cite{HaagerupAndo}.

This will however enable to express easily enough the unitarity of $eue$ since more maps can be defined on the $C^*$-algebra. {With our new section 3 corrected from a previous preprint version (especially axiom (38))}, this will not be completely easy since $\varphi(P(.)P(.))$ has not been defined in section 3 (and could not have been defined). 
Our solution won't be completely satisfactory since it will axiomatize only a class of cross-product depending on parameters such as $\lambda_0$ above and we won't axiomatize every cross product with the same theory which would require to axiomatize the union over all parameters. This will be enough though to obtain the stability by ultrapower we are aiming at to apply in  corollary \ref{AHzero}, thus extending \cite{HaagerupAndo}. 


We will fix $\lambda_{j}, j\geq 0$ increasing in $j$ satisfying : $\lim_{j\to\infty} \lambda_{j}=1$ and such that for $i\geq 0:$
\begin{equation}\label{wLmn}\langle \Delta(e^{-j}+\Delta)^{-1}(u\xi_\varphi),(u\xi_\varphi)\rangle\geq \lambda_{j};\ \ \   \sum_{l=0}^j\varphi(uw_{l,l}u^*)\geq \lambda_{j}.\end{equation}

Of course, in any cross-product (for instance in any $III_0$-factor), there is a choice of parameters satisfying this if we take $\lambda_{j}$ the minimum of the two values it has to bound, which is increasing in $j$ and converges as expected to $1$.

We are now ready to state what we need to add in the language of the Ando-Haagerup theory.

\begin{itemize}
\item[$\bullet$]The constant $u$  in $D_1(V)$ for a preimage of the implementing unitary. 
\item[$\bullet$] Constant symbols $W_{i,j}\in D_1(V),i,j\in\N$ with $w_{i,j}=P(W_{i,j})\in D_1(U),i,j\in\N$ for a matrix unit
. Unary function symbols $P_{i,j}:D_1(U)\to D_1(U),i,j\in\N$ meaning $P_{i,j}(x)=w_{i,i}xw_{j,j}.$
\item[$\bullet$] Constant symbols $u_t,t\in\Q$ in $D_1(U)$ for Connes' cocycle derivatives $(D\varphi:D\psi)_t$ and unary function symbols $\Sigma_t$ (for $\sigma_t^\psi$).
\item[$\bullet$] Unary function symbols for $(m,l)\in\Q^*,m> 0$ $\Psi_{m,l}:D_n\to  D_n, \Psi_{N,0}=\Psi_N$ (for Fejer's map $F_{m,l}^\psi$)
\item[$\bullet$] Binary function symbols for $N\in\N$ $m_{N,\infty,P},m_{\infty,N,P}$ (meaning $F_{N}^\varphi(.)P(.),P(.)F_{N}^\varphi(.)$)
\item[$\bullet$] Unary function symbols for $r\in\Q$ $\Gamma_{r}:D_n(U)\to  D_n(U), $ (meaning $G_r^\psi$)
\item[$\bullet$] Binary relation symbols $\psi_{r,n,m,P}
$ for $(n,m)\in\N^2,r\in\Q$
  meaning \[\psi_{r,n,m,P}(x,y)= \langle G_r^\psi(w_{m,m}P(x)w_{n,n})\xi_\psi,P(y)w_{n,n}\xi_\psi\rangle,
  \]
  Binary relation symbols $\mathscr{E}_{\beta,P}^{\psi,n,m}
  $
   for $(n,m)\in\N^2,\beta\in
   \Q\cap]0,1/2[
   ,$
  meaning \[\mathscr{E}_{\beta,P}^{\psi,n,m}(x,y)= \langle \Delta_\psi^{\beta }(w_{m,m}P(x)w_{n,n}\xi_\psi),P(y)w_{n,n}\xi_\psi\rangle
  ,\]
\item[$\bullet$] Unary function symbols for $N,k\in \N^*$ \[\Pi_{N,k},E_{N,k},EU_{N,k},\theta EP,\overline{\theta} EP,u EP,u^* EP:D_m(V)\to D_m(U)\] for \[\Pi_{N,k}(x)=P(u^k)\Psi_N(P(x)),E_{N,k}(x)=E_{M_\psi}(P(u^k)\Psi_N(P(x))),\] \[EU_{N,k}(x)=P((u^k)^*)E_{M_\psi}(P(u^k)\Psi_N(P(x))),uEP(x)=P(u)E_{M_\psi}(P(x)),\]  \[\theta EP(x)=P(u)E_{M_\psi}(P(x))P(u^*), u^*EP(x)=P(u^*)E_{M_\psi}(P(x)),\overline{\theta}EP(x)=P(u^*)E_{M_\psi}(P(x))P(u) .\]
\end{itemize}

The uniform continuity constants are similar to previous sections. 
For instance, we bound $||\Sigma_t(x)||_\varphi^*=||\sigma_t^\varphi(u_txu_t^*)||_\varphi^*\leq 9||x||_\varphi^*$ from lemma \ref{stability} since $u_t$ is in the centralizer. From the uniform continuity bounds in section 3, one gets: $||F_{N}^\varphi(y)P(x)||_\varphi^*\leq ||(F_{N}^\varphi(y)P(x))^*||_\varphi\leq ||x||\ ||(F_{N}^\varphi(y))^*||_\varphi\leq e^{N/2}\sqrt{1+e^{N}}||x||\ ||y||_\varphi^*$.

We define, for $N=(N_1,...,N_n),\lambda=(\lambda_1,...,\lambda_n)$ :
\begin{align*}\epsilon_N(\lambda)^2=&2-\sum_{i=1}^n\overline{\lambda_i}\mathscr{E}_{0,N_i,\infty,N_1}(P(u),P(u),1)-\sum_{j=1}^n{\lambda_j}\mathscr{E}_{0,N_1,\infty,N_j}(1,P(u^*),P(u^*))\\&+\sum_{i,j=1}^n\lambda_i\overline{\lambda_j}\mathscr{E}_{0,N_i,\infty,N_j}(P(u),1,P(u^*))+\sum_{i,j=1}^n\lambda_j\overline{\lambda_i}\mathscr{E}_{0,N_i,\infty,N_j}(P(u^*),1,P(u))\\&-\sum_{i=1}^n{\lambda_i}\mathscr{E}_{0,N_i,\infty,N_1}(P(u^*),P(u^*),1)-\sum_{j=1}^n\overline{\lambda_j}\mathscr{E}_{0,N_1,\infty,N_j}(1,P(u),P(u))
\end{align*}
meaning $\epsilon_N(\lambda)=||P(u)-\sum_{i=1}^n\lambda_iF_{N_i}^\varphi(P(u))||_\varphi^\#$ if we know $\varphi(P(u)P(u^*))=1=\varphi(P(u^*)P(u)).$ Our axiom (42) based on the control of the spectrum of $u$ will insure those equalities  thanks to the various maps we will define  in axioms (43),(49). 

\medskip

We finally consider the following supplementary axioms beyond Ando-Haagerup theory and (24), depending on parameters $\lambda_0<1, \lambda_{j}$:
\medskip

\begin{enumerate}
\item[(39)] 
For $N\in\Q\cap]0,\infty[,l\in\Q^*,m\in\N^*$ with $]l-N,l+N[\subset ]\ln(\lambda_0),-\ln(\lambda_0)[-\{0\}$  
\[ \sup_{x\in D_{m}}d_U(\Psi_{N,l}(x)),0)=0,\]

and if $]l-N,l+N[\subset ]\ln(\lambda_0),\infty[$  
\[ \sup_{x\in D_{m}}d_U(\Psi_{N,l}(P(u))),0)=0,\]
and (25) with $x,y$ replaced by $2\Psi_{2n}(x)-\Psi_{n}(x),2\Psi_{2n}(y)-\Psi_{n}(y)$ with $2n\in \Q\cap]0,|\log(\lambda_0)|[,$ \[2\Psi_{2n}(w_{i,j})-\Psi_{n}(w_{i,j})=w_{i,j}.\]
For all $k,l\in \N,$\[tr(m_{1}^l(\phi,xW_{k,l}))=2\mathscr{E}_{0,\infty,2}(P(x),w_{k,l})-\mathscr{E}_{0,\infty,1}(P(x),w_{k,l})\]
\item[(40)]With the notation in (25) \[\sup_{x\in D_m(U)}\max(0,d(P_{k,j}(x),w_{k,k}F_N(x)w_{j,j})-9d(x,F_N(x)))=0\]

For $t\in\Q$, $P_{k,k}(u_t)= 2^{(k+1)it}w_{k,k}$
,$ P_{k,j}(u_t)=0,k\neq j$ and $2N\in\Q\cap]0,|\log(\lambda_0)|[$, \[2F_{2N}(u_t)-F_{N}(u_t)=u_t,2\Psi_{2N}(u_t)-\Psi_{N}(u_t)=u_t,\] \[2\Psi_{2N}(\theta EP(x))-\Psi_{N}(\theta EP(x))=\theta EP(x),2\Psi_{2N}(\overline{\theta} EP(x))-\Psi_{N}(\overline{\theta} EP(x))=\theta EP(x).\]
\item[(41)]For $N\in\Q\cap]0,\infty[, l\in\Q^*,M,m\in\N^*, K,L\in\N, L\geq 2K$  
\[\Sigma_{l}(F_M(x))=\sigma_{l}(M_{(M+4,0)}(M_{(0,M+1)}(u_{l},F_M(x)),u_{l}^*)\]
\begin{align*}&\sup_{x\in D_{m}}\max(0,-\frac{4m}{n^2}-\frac{16m}{\pi Nn^3}-\frac{(10+|l|N)m}{n\pi}-\frac{mN^3}{2n}+\\&d_U\left(\Psi_{N,l}(F_M(x)),\frac{N}{2\pi n^2}x+\sum_{k=-n^3,k\neq0}^{n^3-1}\frac{e^{il\frac{k}{n^2}}}{n^2}\frac{1-\cos(N\frac{k}{n^2})}{\pi N \frac{k^2}{n^4}}\Sigma_{k/n^2}(F_M(x))\right)=0,\end{align*}


\item[(42)]For any $K,m,n\in\N^*, \beta\in \Q\cap]0,1/2[$
\[
2^{-1}e^{K/2}\E_{P,1/2,-K}(u,u)\geq \lambda_{K}\]\[\sum_{l=0}^{m}\sum_{k=0}^n2^{-k-1}\E^{\psi,k,l}_{\beta,P}(u^*,u^*)\geq \lambda_0^{-\beta}[\lambda_{m}-2^{-n-1}(m+1)].\]
\item[(43)]For any $N,M,L\in\N^*$ $m_{\infty,N,P}(y,x)=m_{N,\infty,P}(x^*,y^*)^*$ \[\mathscr{E}_{0,N,\infty,L}(x,P(y),z)=\mathscr{E}_{0,\infty,L}(m_{N,\infty,P}(x^*,y),z)\]
For $r\in \Q,n,m,M\in\N$
:   \[\sup_{x\in D_{M}(U)}\max(0,d_U\left(\Gamma_s(x),\frac{1}{n^2}\sum_{k=-n^3}^{n^3-1}\frac{2e^{-is\frac{k}{n^2}}}{e^{\pi \frac{k}{n^2}}+e^{-\pi \frac{k}{n^2}}}\Sigma_{k/n^2}(x)\right)- \frac{8e^{-\pi n}M}{\pi}-\frac{16M}{n^2}-\frac{4(\pi+s)M}{n})=0.\]
 \begin{align*}&\sup_{(x,y)\in D_M(V)^2}\max(0,-\frac{8}{\pi\sqrt{N}}M^2-\frac{2(n+m+2
 +ch(r))}{\sqrt{N}}M^2+\\&\ \ \ \ \ \ \ \left|\psi_{r,n,m,P}(x,y)-2^{n+1}\mathscr{E}_{0,N,\infty}\left[\Gamma_r(2m_{2,\infty,P}\left(w_{m,m},2m_{\infty,2,P}(x,w_{n,n})-m_{\infty,1,P}(x,w_{n,n}))\right)\right.\right.\\&\ \ \ \ \ \ \ \left.\left.-m_{1,\infty,P}\left(w_{m,m},2m_{\infty,2,P}(x,w_{n,n})-m_{\infty,1,P}(x,w_{n,n}))\right),(2m_{\infty,2,P}(y,w_{n,n})-m_{\infty,1,P}(y,w_{n,n}))\right]\right|)=0 \end{align*}
 For $\alpha, 
 \Q\cap]0,1/2[,
 \epsilon= 1/2-\alpha,
 \delta= \min(\epsilon,\alpha) ,m,K,L,n\in\N^*,r\in \Q$ we have 
 \begin{align*}\sup_{(x,y)\in D_{m}^2(V)}\max(0&,\left|\E_{\alpha,P}^{\psi,n,m}(x,y)-\frac{1}{n^2}\frac{\cos(\alpha \pi)}{2\pi}\sum_{k=-n^3}^{n^3-1}e^{\alpha k/n^2}\psi_{k/n^2,n,m,P}(x,y)\right|\\&- \frac{4e^{-n\delta}m^2}{\pi\delta}-|(1-e^{1/n^2})|\frac{2(3+e^{\epsilon/n^2})m^2}{\pi\delta})=0 \end{align*}
 
\item[(44)] For $N=(N_1,...,N_n),N_i,M,k,m\in \N^*,\lambda=(\lambda_1,...,\lambda_n),\lambda_i\in \C$

\[\sup_{x\in D_m(U)}\max(0,|\mathscr{E}_{0,M,\infty,N_1}(x,1,1)-\sum_{i=1}^n\lambda_i\mathscr{E}_{0,M,\infty,N_i}(x,P(u),P(u))|-m\epsilon_N(\lambda) )=0,\]

\[\sup_{x\in D_m(U)}\max(0,|\mathscr{E}_{0,N_1,\infty,M}(1,1,x)-\sum_{i=1}^n\lambda_i\mathscr{E}_{0,N_i,\infty,M}(P(u),P(u),x)|-m\epsilon_N(\lambda) )=0,\]


\item[(45)]For $M=(M_1,...,M_n),M_i,m,p, n,L,N\in \N,\lambda=(\lambda_1,...,\lambda_n),2Q\in\Q\cap]0,|\log(\lambda_0)|[$, 
\[\sup_{\tiny\begin{array}{c} y\in D_p(V)\\x\in D_m(U)\end{array}}\max(0,|\sum_{i=1}^n\lambda_i\mathscr{E}_{0,M_i,\infty,L}(P(u^*),(2\Psi_{2Q}-\Psi_{Q})(P(y)),x)-\mathscr{E}_{0,\infty,L}(uEP(y),x)|- mp\epsilon_M(\lambda))=0\]
\[\sup_{\tiny\begin{array}{c} y\in D_p(V)\\x\in D_m(U)\end{array}}
\max(0,|\mathscr{E}_{0,L,\infty}(x,\theta EP(y))-\sum_{i=1}^n\lambda_i\mathscr{E}_{0,L,\infty,M_i}(x,uEP(y),P(u))|- mp\epsilon_M(\lambda))=0\]
\[\sup_{\tiny\begin{array}{c} y\in D_p(V)\\x\in D_m(U)\end{array}}
\max(0,|\sum_{i=1}^n\lambda_i\mathscr{E}_{0,M_i,\infty,L}(P(u),(2\Psi_{2Q}-\Psi_{Q})(P(y)),x)-\mathscr{E}_{0,\infty,L}(u^*EP(y),x)|- mp\epsilon_L(\lambda))=0\]
\[\sup_{\tiny\begin{array}{c} y\in D_p(V)\\x\in D_m(U)\end{array}}
\max(0,|\mathscr{E}_{0,L,\infty}(x,\overline{\theta} EP(y))-\sum_{i=1}^n\lambda_i\mathscr{E}_{0,L,\infty,M_i}(x,u^*EP(y),P(u))|- mp\epsilon_M(\lambda))=0\]


\item[(46)]
$2\Psi_{2M}(\Pi_{N,k}(y))-\Psi_M(\Pi_{N,k}(y))=E_{N,k}(y)$ for $M<|\log(\lambda_0)|/2$ \begin{align*}&\sup_{x\in D_m(U),y\in D_l(V)}\max(0,|\sum_{i=1}^n\lambda_i\mathscr{E}_{0,M_i,\infty,L}((P(u^k))^*,\Psi_N(P(y)),x)-\mathscr{E}_{0,M_i,\infty,L}(1,\Pi_{N,k}(y),x)|^2\\&-\sum_{i,j=1}^n\lambda_i\overline{\lambda_j}\varphi(m_{M_j,M_i}(P(u^k),P(u^k)^*)m^2l^2\\&+\sum_{j=1}^n\lambda_i\mathscr{E}_{0,M_j,\infty,L}(P(u^k)^*,P((u^k)^*),1)m^2l^2+\sum_{j=1}^n\overline{\lambda_j}\mathscr{E}_{0,L,\infty,M_j}(1,P(u^k),P((u^k)))m^2l^2-m^2l^2)=0\end{align*}

\begin{align*}&\sup_{x\in D_m(U),y\in D_l(V)}\max(0,|\sum_{i=1}^n\lambda_i\mathscr{E}_{0,M_i,\infty,L}((u^k),E_{N,k}(y),x)-\mathscr{E}_{0,M_i,\infty,L}(1,EU_{N,k}(y),x)|^2\\&-\sum_{i,j=1}^n\lambda_i\overline{\lambda_j}\varphi(m_{M_j,M_i}(P(u^k)^*,P(u^k))m^2l^2\\&+\sum_{j=1}^n\lambda_i\mathscr{E}_{0,M_j,\infty,L}(P(u^k),P((u^k)),1)m^2l^2+\sum_{j=1}^n\overline{\lambda_j}\mathscr{E}_{0,L,\infty,M_j}(1,P(u^k)^*,P((u^k)^*))m^2l^2-m^2l^2)=0\end{align*}
For $K=\lfloor\frac{N}{|\log(\lambda_0)|}\rfloor+1$
\[\Psi_N(P(x))=E_{N,0}(P(x))+\sum_{k=1}^K EU_{N,k}(x)+[EU_{N,k}(x^*)]^*\]
\item[(47)] For $2Q\in\Q\cap]0,|\log(\lambda_0)|[$ and with for short $P_{\leq n}(x)=\sum_{i,j=0}^nW_{i,i}xW_{j,j}$
 \begin{align*}&\sup_{x\in D_1(V)}\max(0,\sum_{j=0}^n2^j\varphi(P_{j,j}(\theta EP(P_{\leq m}(xx^*))))-\lambda_0\sum_{j=0}^m2^j\varphi(P_{j,j}[2\Psi_{2Q}(P(xx^*))-\Psi_{Q}(P(xx^*))])=0\end{align*}
\item[(48)]For any $K,N\in\N$, \[\varphi(E_{N,0}(P(x))+\sum_{k=1}^K EU_{N,k}(x)+[EU_{N,k}(x^*)]^*)=\varphi(E_{N,0}(P(x)).\]
\end{enumerate}
Stated in words, (39) expresses the various spectral properties of constants, spectral gap and traciality of $\varphi$ on the compression of the centralizer of $\psi$ by $w_{0,0}.$ (40) defines $u_t$ and expresses preservation of the centralizer by the actions. (41) , (43) define various spectral theory for $\psi$. (44) expresses the unitarity of $u$ based on the inequalities obtained in (42) using \eqref{wLmn}. (46) expresses the cross-product decomposition based on maps defined in (45). (47) states the relation $\psi(u.u^*)\leq \lambda_0\psi(.).$ (48) identifies $\varphi$ as the dual state on the cross product of the geometric state on the centralizer of $\psi$ which was already identified in (24) and (39). 

For a geometric state $\varphi$ on $N=N_0\otimes B(H)$ for $N_0$ finite, we call associated trace of $\varphi$ the trace : $\tau=\varphi|_{N_0}\otimes Tr$.
\begin{theorem}\label{AHIII0}
We will fix  $\lambda_0\in ]0,1[$ and $\Lambda=(\lambda_{j})_{ j\geq 0}$, increasing in $j$ and such that $\lim_{j\to\infty} \lambda_{j}=1.$

The class of nonuples $(C,X,H,J,\Pi_P,\phi,M,u,W)$ of a weak-* dense $C^*$ algebra $C$ of $X^*$, a von Neumann algebra in standard form $(X^*,H,J,Im(\Pi_P))$ with predual $X$ containing a state $\phi\in X$ and a von Neumann algebra $M\simeq eX^*e$ for $e$ the support projection of $\phi$ with $M=N\rtimes_\theta\Z$ for a von Neumann algebra $N$ of type $II_\infty$ 
with $\phi$ the dual weight of a geometric state on $N$ for the matrix unit $w=(w_{i,j}=P(W_{i,j}))\in N,W_{i,j}\in C$ satisfying \eqref{wLmn} and $\theta$ an automorphism of $N$ implemented by the image $U:=eue\in M$ of $u\in C$ decreasing the trace $\tau$ associated to $\phi|_N$ by a factor $\lambda_0<1$ fixed (i.e. $\tau(\theta(x))\leq \lambda_0 \tau(x)$) 
and such that the image $eue\in M(\sigma^{\hat{\tau}},]-\infty,\log(\lambda_0)])$ in the modular theory of the dual weight $\psi=\hat{\tau},$ is axiomatisable in the language above by the theory $T_{III_0}(\lambda_0,\Lambda)$ consisting of $T_{AHW^*}$, (24) and (39)-(48).
\end{theorem}
By Connes' result \cite[Th 5.3.1]{ConnesThesis}, every type $III_0$-factor is of this type with  $\theta$ ergodic on the center of $N$ which is diffuse and conversely by his  \cite[Prop 5.1.1]{ConnesThesis}. We leave to the reader the identification of the category of axiomatization from structure preserving morphisms.

\begin{proof}
Let us start with such a nonuple and check this gives a model of $T_{III_0}(\lambda_0,\Lambda)$ with the interpretation suggested in the language description with $\psi=\hat{\tau}$. We thus define all the data in this way.
First note an equation we will use several times
\[||\sigma_t^\psi(x)-x||_\varphi^*\leq ||\sigma_t^\phi(x)-x||_\varphi^*+2||x||||u_t-1||_\varphi^\#\leq c|t|\ ||x||,\] with $c=6\ln(2)+2\leq 8$ since $(||u_t-1||_\varphi^\#)^2=\sum_{k=0}^\infty|2^{(k+1)it}-1|^22^{-k-1}\leq \sum_{k=0}^\infty\ln(2)^2(k+1)^2|t|^22^{-k-1}= C|t|^2, C=6\ln(2)^2$
Similarly, since $\sigma_t^\phi=\sigma_t^\psi(u_t^*.u_t)$, $\Gamma_r(w_{m,m}xw_{n,n})=(w_{m,m}\Gamma_r(x)w_{n,n}$ and  $w_{n,n}u_t=w_{n,n}2^{(n+1)it}$, \begin{align*}&|\langle[\sigma_t^\phi(\Gamma_r((w_{m,m}xw_{n,n}))-\Gamma_r((w_{m,m}xw_{n,n})].\xi_\psi,yw_{n,n}.\xi_\psi\rangle|\\&\leq |\langle[\sigma_t^\psi(\Gamma_r((w_{m,m}xw_{n,n}))-\Gamma_r((w_{m,m}xw_{n,n})].\xi_\psi,yw_{n,n}.\xi_\psi\rangle|\\&+|\langle[u_t^*\Gamma_r(w_{m,m}xw_{n,n})u_t-\Gamma_r(w_{m,m}xw_{n,n})].\xi_\psi,\sigma_{-t}^\psi(yw_{n,n}.\xi_\psi)\rangle|\\&\leq \sqrt{e^r\langle\frac{|\Delta_\psi^{it}-1|^2}{\Delta_\psi+\Delta_\psi^{-1}}(\Delta_\psi^2+1)(e^r+\Delta_\psi)^{-2}(w_{m,m}xw_{n,n})\xi_\psi,w_{m,m}xw_{n,n}\xi_\psi \rangle}\|yw_{n,n}.\xi_\psi\|\\& 
+|2^{(m+1)it}-1|||x||\ ||y||+|2^{(n+1)it}-1|||x||\ ||y||\\&\leq \sqrt{ (e^r+e^{-r})(2t)^2e^{-2}\|xw_{n,n}\xi_\psi\|^2}\|yw_{n,n}\xi_\psi\|
+(n+m+2)t\ln(2)\|x\|\|y\|
\\&\leq 2(n+m+2+ch(r))t\|x\|\|y\|
\end{align*}
The second inequality in (43) is a straightforward consequence, once one cuts the integral defining $F_N^\varphi$ from the modular group at $t=\pm1/\sqrt{N}$.

We already have a model of $T_{AHW^*}$ and (24) that expresses we have a matrix unit. The first part of (39) means that $M=M(\sigma^{\psi},]-\infty,\log(\lambda_0)]\cup \{0\}\cup [-\log(\lambda_0),\infty[)$. This is indeed the case since, by definition, $N$ is in the centralizer of $\psi$ and by assumption $U:=eue\in M(\sigma^{\hat{\tau}},]-\infty,\log(\lambda_0)])$ so that $NU^k\in M(\sigma^{\hat{\tau}},]-\infty,k\log(\lambda_0)])$ thus we see that any finite sum of terms in $N, NU^k$ and $(U^*)^kN$ is in $M(\sigma^{\psi},]-\infty,\log(\lambda_0)]\cup \{0\}\cup [-\log(\lambda_0),\infty[)$ and since those finite sums are dense, one obtains the equality with $M$ (we also see from the crossed-product decomposition that $N$ is exactly the centralizer for $\psi$). The modified (25) in (39) means (25) for the centralizer $N=M_\psi=N_0\otimes B(H)$ (especially $N_0$ in the centralizer for $\varphi$) and the end of (39) means the matrix unit in this centralizer and all are satisfied since $N$ is $II_\infty$.

The second part of (40) then means that $u_t$ is in the centralizer for both $\varphi,\psi$ the first part identifies its components on the matrix unit. The first part uses the notation of (25) and defines $P_{k,j}$ as expected 
Thus the first equation for $u_t$ means it is in $B(H)$ diagonal with the expected expression. (To see this is indeed the expected cocycle between the trace and the geometric state, cf. \cite[lemma VIII.2.10]{TakesakiBook}). The two last equations in (40) express the invariance of the centralizer $M_\psi$ by $u.u^*$ and $u^*.u$.
Then (41) means that $\Sigma_t$, the modular group for $\psi,$ is indeed related as it should to $\sigma_t$ using the cocycle, and $\Psi_{N,l}$ is indeed the corresponding modular map.

Note that conversely, (24) implies that $w_{i,j}$ is a matrix unit so that $M=M_0\otimes B(H)$ and then (40) implies $u_t$ is diagonal in $B(H)$. We can thus compute from $\varphi$ a (not necessarily unique) weight $\psi$ \cite[Th VIII.3.8]{TakesakiBook} with corresponding cocycle derivative and (41) implies we have the corresponding modular theory and Fejer map.
(39) then describes the centralizer as some $N_0\otimes B(H)$ and $\psi$ as lacunary, and the spectral condition on $P(u)$.

Let us come back to checking that our algebra gives a model. (42) is satisfied since $P(u)=eue$ is a unitary, based on \eqref{wLmn}, (43) defines $\Gamma_s$, $\psi_{r,n,m,P},\mathscr{E}_{\beta,P}^{\psi,n,m}$ as explained at the beginning of the proof or in a way similar to section \ref{Ocneanu} (for instance we use one more variant of (19)).

(44)-(45) are easy to check from the definitions that they express closely by duality. For instance, by Cauchy-Schwartz and unitarity of $P(u)$:
\begin{align*}|\varphi(F_N^\varphi(x)^*1F_M^\varphi(1)^*)&-\varphi(F_N^\varphi(x)^*P(u)F_M^\varphi(P(u))^*)|=|\varphi(F_N^\varphi(x)^*P(u)(P(u^*)-F_M^\varphi(P(u))^*))|\\&\leq \sqrt{\varphi(F_N^\varphi(x)^*P(u)P(u^*)F_N^\varphi(x))}||P(u)-F_M^\varphi(P(u))||_\varphi^\#.\end{align*}



The three first formulas in (46) are straightforward from the definitions and Cauchy Schwarz inequality. The last formula is the crucial one and follows from the proof of 
\cite[Proposition 6.23]{HaagerupAndo}. Let us explain it for the reader's convenience. First since \[p_n:=\bigvee_{k=0}^\infty(u^k(w_{0,0}+...+w_{n,n})u^{-k})\in M_\psi\] we can use $\sigma^\psi$ and thus $\Psi_N$ are $M_\psi$ bimodular :  $\Psi_N(p_nP(x)p_n)=p_n\Psi_N(P(x))p_n$. Now if $U=P(u),$ $Up_nUp_n=Up_n\theta(p_n)U=U\theta(p_n)U=U^2p_n$ since by definition $\theta(p_n)\leq p_n$ and similarly $[Up_n]^k=U^kp_n\in M(\sigma^{\hat{\tau}},]-\infty,k\log(\lambda_0)])$, $[p_nU^*]^k=p_n(U^*)^k\in M(\sigma^{\hat{\tau}},[-k\log(\lambda_0),\infty[)$ and thus $U^k\Psi_N(p_nP(x)p_n)\in M(\sigma^{\hat{\tau}},]-\infty,k\log(\lambda_0)+N])$

Thus for $k\geq K$, $k\log(\lambda_0)+N<0$ so that $E_{N,k}(p_nxp_n)=E_\psi(U^k\Psi_N(P(p_nxp_n)))=0$ and similarly $E_{N,k}(p_nx^*p_n)=0$ the crossed product expansion for $M_\psi\rtimes_\theta\Z$ thus gives \[\Psi_N(p_nxp_n)=E_{N,0}(p_nxp_n)+\sum_{k=1}^\infty UE_{N,k}(p_nxp_n)+[UE_{N,k}(p_nx^*p_n)]^*\] and is actually the expected finite sum stopping at index $K$. Since $\sigma_t^\psi$ is strongly continuous, so are the maps in the sum and $p_n\to 1$ so that taking the limit, we have thus finished checking (46).

Let us write $p_{n,0}=w_{0,0}+...+w_{n,n}.$
(47) is equivalent to \[\tau(p_{n,0}\theta (p_{m,0}E_{M_\psi}(P(xx^*))p_{m,0}))\leq \lambda_0\tau(p_{m,0}E_{M_\psi}(P(xx^*)))\]
which comes from $p_{n,0}\leq 1$ and $\tau\circ \theta\leq \lambda_0\tau$. (48) is then standard for a dual weight (see e.g. after polarization \cite[Th X.1.17 (i)]{TakesakiBook}).

Conversely, assume given a model. We already noted we have a lacunary weight $\psi$ with $II_\infty$ centralizer. We have to check  that $M$ is indeed a cross-product by using \cite[Th 5.3.1]{ConnesThesis} in the variant of \cite[Lemma 6.25]{HaagerupAndo}. We know that $E_{M_\psi}=2\Psi_{2N}-\Psi_{N}$ is a conditional expectation. (42) implies that $\varphi(P(u)P(u)^*)=1=\varphi(P(u)^*P(u))$. (The first using $\beta\to 0$, $n\to\infty$  and finally $m\to \infty$) This uses the spectral theory maps have the right interpretation from (41), (43).

As explained before, equation (42) guaranties $\epsilon_N(\lambda)=||P(u)-\sum_{i=1}^n\lambda_iF_{N_i}^\varphi(P(u))||_\varphi^\#$ and then from Hahn-Banach and the proof of theorem \ref{Ocneanu}, there is a net $u_n$ of convex conbinations of the form $\sum_{i=1}^{m_n}\lambda_iF_{N_i}^\varphi(P(u))$ such that $||P(u)-u_n||_\varphi^\#\to 0.$ This is the starting point to use (44) which then gives by taking a limit \[\varphi(F_M(x^*)(1-P(u)P(u^*)))=0,\ \ \ \ \ \varphi((1-P(u^*)P(u))F_M(x^*))=0\]  so that by density since $x,M$ are arbitrary, $1=P(u)P(u^*)=P(u^*)P(u)$ and thus $P(u)$ is unitary as expected.

 The last relations in (40) also implies that $P(u).P(u)^*,P(u)^*.P(u)$ leave stable $M_\psi$ in using the weak-* density of the image of $E_{M_\psi}P.$ This gives the automorphism $\theta(x)=P(u)xP(u)^*$ of $M_\psi.$ From the equivalent version of (47) above, and letting $n\to \infty,m\to \infty$, one gets $\tau(\theta (E_{M_\psi}(P(xx^*))\leq \lambda_0\tau(E_{M_\psi}(P(xx^*)))$ and we can replace by density (Kaplansky density theorem e.g. \cite[Th I.4.24]{TakesakiBook})  $E_{M_\psi}(P(xx^*))$ by any positive y in $M_\psi$. Then from \cite[Prop 5.1.1]{ConnesThesis}, one deduces $p(\theta^k)=0$ for $k\neq 0$ and of course $U=P(u)$ such that $\theta(x)=UxU^*.$ The last equation of (46) implies $\Psi_N(P(x))$ is in the algebra generated by $N,U$ so that by  density of image of $P$ one gets $\Psi_N(M)$ is in the von Neumann algebra generated by  $N,U.$
Taking $N\to \infty$ and using $||\sigma_t^\psi(x)-x||_\varphi^*\leq  c|t|\ ||x||,$ as already noted and thus $||\Psi_N(x)-x||_\varphi^*\to_{N\to \infty} 0$ for any $x\in M$ so that  one gets $M$ is generated by $N,U$. It thus only remains to check that $M_\psi'\cap M\subset M_\psi.$ But for $x\in M_\psi'\cap M$, we have $\Psi_N(x)\in M_\psi'\cap M$ (since $\Psi_N$ is $M_\psi$-bimodular). Moreover, it has a finite decomposition $\Psi_N(x)=x_0+\sum_{k=1}^Kx(k)U^k+x(-k)(U^*)^k$ from the uniqueness of the decomposition (formulas for $x(k)$), one deduces for any $a\in M_\psi$, $x(k)\theta^k(a)=ax(k),\theta^k(a)x(-k)=x(-k)a$ for $k>0$ and by \cite[Rmq 1.5.3 (a)]{ConnesThesis} one deduces from $p(\theta^k)=0$ for $k\neq 0$ that $x(k)=0$ for $k\neq 0.$ Thus $\Psi_N(x)\in M_\psi$ and taking $N\to \infty$ one deduces $x\in M_\psi$ as expected. This concludes the check of the assumptions of \cite[Lemma 6.25]{HaagerupAndo} and thus $M$ is indeed a cross product as expected.

It remains to check the data is uniquely determined as it should. 
(43) is similar to previous sections.

Similarly, the 4 equations with a $\lambda$ in (45) gives the definitions of $\theta EP,\overline{\theta} EP,u EP,u^* EP$.



Reasoning as before the three first relations in (46) define $\Pi_{N,k},E_{N,k},EU_{N,k}.$

Equation (48) finally characterizes the state $\varphi$ as determined from its restriction on $M_\psi$ as it should be for a dual weight and all the data is determined as expected. At this stage we can also choose $\psi=\hat{\tau}$ the dual weight of the determined trace and it has the expected modular theory.
\end{proof}

We can thus finish with the analogue for uncountable ultraproducts of \cite[Th 6.16]{HaagerupAndo}

\begin{corollary}\label{AHzero}
Let $M$ be a $\sigma$-finite factor of type $III_0$ with faithful normal states $\varphi$ and assume $\omega$ is a countably incomplete ultrafilter, then $(M,\varphi)^\omega$ is not a factor.
\end{corollary}

\begin{proof}Using the crossed product decomposition in \cite{ConnesThesis} one gets $M\simeq (N\otimes B(H))\rtimes \Z$
with $N$ finite and there is $\lambda_0,\Lambda$ so that $(M,M_*,\varphi,M,w,u)$ satisfies $T_{III}(\lambda_0,\Lambda)$ and thus so does $(M^\omega,(M_*)^\omega,\varphi^\omega,(M_n,\varphi)^\omega,w,u)$. Thus $(M,\varphi)^\omega)\simeq (N^\omega \otimes B(H))\rtimes \Z$ and the center $(\mathcal{Z}(N))^\omega\simeq \mathcal{Z}(N^\omega)$ by \cite[Corol 4.3]{FarahI} and it suffices to find a non trivial element in the fixed point algebra of the action of $\Z$, since such an element will be a non trivial element in the center of $(M,\varphi)^\omega.$ This follows from \cite[lemmas 6.19, 6.22]{HaagerupAndo}. Indeed, since $\omega$ is countably incomplete on $I$, there is a sequence $J_n\in \omega$, with $\cap_{n\in \N}J_n=\emptyset, J_0=I$ and we can assume $J_n$ decreasing.  Define the net $k_i$ by $k_i=n$ if $i\in J_{n-1}-J_n$ so that $\{i\in I, k_i\geq n\}=J_{n-1}\in \omega$ and thus $\lim_{n\to \omega} \frac{1}{k_n}=0$. Take $p=(1_{B_{k_n}})^\omega$ with $B_n$ built in their lemma 6.22, this of course gives as in the proof of their lemma 6.19, a central element in $N$ with $\varphi^\omega(p)=1/2$
since $\varphi(B_n)=1/2$ and $||upu^*-p||\to 0$ since $\lim_{n\to \omega} \mu(TB_{k_n}\Delta B_{k_n})\leq \lim_{n\to \omega}2/k_n=0$ for $\mathcal{Z}(N)\simeq L^\infty(\mu)$ with $u.u^*$ acting via $T$ on the measure space.\end{proof}

\section{Proof of theorem 1} We apply \cite[Th 5.6]{FarahII}. We can do this since our Ocneanu theory (theorem \ref{OcneanuTh}) and Groh theory (theorem \ref{GrohThPredual}) are axiomatized in a separable language (in section 1.4 we even wrote an explicit countable language). A separable von Neumann algebra is exactly a separable structure (separable unit ball for $||.||_\varphi^*$, which is the same as separable for the strong topology or the weak-* topology by a standard application of Banach-Saks theorem, or separable predual) in those theories. Combining (1) and (2) in the theorem, whether the theory is stable or not, if we assume (CH) the model theoretic ultrapowers are isomorphic and thus, by identification with Ocneanu and preduals of Groh ultraproducts in our quoted theorems, we deduce the first point.
 
 Assume now that the continuum hypothesis fails. 
Since we assumed $M$ is a factor not of type $III_0$, we treat each remaining type of factors for $M$ separately. Types $I_n,II_1$ are known from \cite[Th 4.7]{FarahI}.  Type $I_\infty$ is a consequence of the canonical isomorphism $(B(H))^\omega=B(H)$ (see \cite[section 6.1]{HaagerupAndo} or the proof of our corollary \ref{Qelim}).

\medskip
Let $M$ of type $II_\infty.$
 Consider on $M$ a geometric state $\varphi$ whose theory is described in theorem \ref{MoreAxiom} in the language with a matrix unit added. This state is lacunary thus there exists $N\in \Q$ such that the De la Vallée Poussin map (obtained from Fejer's map $F_N^\varphi$) $2F_{2N}^\varphi-F_N^\varphi$ is the conditional expectation on the centralizer as explained in the use of axiom (22) and is in the theory with language of $\sigma$-finite $W^*$-probability spaces (without matrix unit added). Consider the formula of this theory  \[f(x_1,y_1,x_2,y_2)=\varphi(M_{(3N,3N)}([\tilde{m}_{(N,N)}(x_1,y_2)-\tilde{m}_{(N,N)}(y_2,x_1)],[\tilde{m}_{(N,N)}(x_1,y_2)-\tilde{m}_{(N,N)}(y_2,x_1)]^*)),\]
with the notations of subsection 1.2 and recall the definition \eqref{tildem} in subsection 1.4 
Recall also that $m_{(N,M)}(x,y)=F^\varphi_N(x).F^\varphi_M(y)$ is one of the smeared product we have in our theory.
In the lacunary case it means \[f(x_1,y_1,x_2,y_2)=||[E_{M_\varphi}(x_1),E_{M_\varphi}(y_2)]||_2^2\]
computed with the $||.||_2$ norm of the centralizer $M_\varphi$. This formula is thus the formula witnessing the order property in \cite{FarahI} for the centralizer. We thus obtain that the theory of $M$ with this geometric state in the language of $\sigma$-finite algebras is unstable. For, we use \cite[Th 5.5]{FarahII}, it suffices to check it has the order property as witnessed by the formula above and this is the case since the centralizer (which is of the form $N\ot Z,$ $N\ II_1$ factor and $Z$ a commutative algebra)  contains unitally $M_{2^n}(\C)$ so that one can use \cite[lemma 3.2]{FarahI}.  Thus, by \cite[Th 5.6]{FarahII} (we use the implication mostly coming from \cite{FarahShelah}),  there exists $2^{\mathfrak{c}}$ ultrafilters with the model theoretic ultraproducts not isomorphic. Assume  that, for two such ultrafilters $\mathcal{U},\mathcal{V}$ we have $(M,\varphi)^\mathcal{U}\simeq(M,\varphi)^\mathcal{V}$ as von Neumann algebras (recall we can compute Ocneanu ultrapower with any state and this gives the same result). We want to count how many ultrafilters of that type there can be. Since the models are non-isomorphic if and only if the states are non-isomorphic, we want to count how many non-isomorphic geometric states there can be on a same ultraproduct von Neumann algebra of a geometric state. Since we took $\varphi$ to be a geometric state described in the language with a matrix unit added, one obtains $\varphi^\mathcal{U},\varphi^\mathcal{V}$ are geometric state with the same matrix unit coming from the one in $M$ and by a standard result if $e=w_{00}$ the first projection in the matrix unit $(M,\varphi)^\mathcal{U}\simeq e(M,\varphi)^\mathcal{U}e\otimes B(H)$. Note that if the isomorphism class of $e(M,\varphi)^\mathcal{U}e$ is determined, there is only one isomorphic class of geometric state on $(M,\varphi)^\mathcal{U}$ and thus the model as $\sigma$-finite $W^*$-probability space is determined (all the remaining part of the theory, (smeared) product, modular theory is determined by the state), and thus there is at most one $\mathcal{U}$ within the family fixed before.
Fix a trace $Tr_N$ on $N\simeq (M,\varphi)^\mathcal{U}\simeq(M,\varphi)^\mathcal{V}.$ 
It is well-known that the equivalence classes of finite projections $e$ is characterized by $Tr(e)\in [0,\infty)$ and thus if $e_\mathcal{U},e_\mathcal{V}$ are the images in the common algebra $N$ we thus have $e_\mathcal{U}Ne_\mathcal{U}\simeq e_\mathcal{V}Ne_\mathcal{V}$ if 
$Tr(e_\mathcal{U})=Tr(e_\mathcal{V}).$

As a consequence the isomorphism invariance classes as von Neumann algebras among the family of $2^\mathfrak{c}$ non-isomorphic models have at most $\mathfrak{c}$ members, so that there are again $2^\mathfrak{c}$ non-isomorphic ultrapowers as von Neumann algebras.

\medskip

Now assume $M$ is of type $III_1$, let us show that $Th(M)=Th(M^\omega)$ has the order property as witnessed by the formula \[f(x_1,y_1,x_2,y_2)=\varphi(M_{(2,2)}([M_{(0,0)}(x_1,y_2)-M_{(0,0)}(y_2,x_1)],[M_{(0,0)}(x_1,y_2)-M_{(0,0)}(y_2,x_1)]^*)).\]
For, one uses from \cite[Th 4.20,Prop 4.24]{HaagerupAndo} according to which the centralizer $(M^\omega)_{\varphi^\omega}$ is a type $II_1$ factor. If $x_i,y_i$ are in the centralizer, the formula is interpreted by :
\[f(x_1,y_1,x_2,y_2)=||x_1y_2-y_2x_1||_2^2\]
 and this is again the formula that witnesses the order property in the type $II_1$ case by \cite[lemma 3.2]{FarahI} since such an algebra contains unitally $M_{2^n}(\C).$ Thus one can take the sequence for the order property for $M$ in the centralizer of $M^\omega.$ Again, from \cite[Th 5.5,5.6]{FarahII}, one gets two ultrafilters with $(M,\varphi)^\mathcal{U}\not\simeq(M,\varphi)^\mathcal{V}$ as models in the language of $\sigma$-finite von Neumann algebras. But if we had an isomorphism as $III_1$ factors, since those ultraproducts are strictly homogeneous by  \cite[Th 4.20]{HaagerupAndo} again, the two states $\varphi^\mathcal{U},\varphi^\mathcal{V}$ would be unitarily conjugated, establishing the isomorphism as models in the language of $\sigma$-finite $W^*$ probability spaces, a contradiction.
 
\medskip 
 
 Consider finally the case where $M$ is of type $III_\lambda,$ $0<\lambda<1$.
One uses \cite[Th 4.3.2]{ConnesThesis} and its proof. One can fix $\varphi$ a periodic state (a faithful normal state of period $T_0=2\pi/\log(\lambda)$ from the computation of the invariant $T$ and its alternative definition in his remark 1.3.3 in the $\sigma$-finite case). Then $M\simeq M\otimes B(H)$ and $\varphi\otimes Tr$ is Connes' construction of a generalized trace on $M$. Since $\varphi$ is lacunary and $M_\varphi$ is a $II_1$ factor (by \cite[Th 4.2.6]{ConnesThesis}) it is a finite factor and its tensor product with $B(H)$ is a $II_\infty$ factor from his corollary 4.3.3). We will use that the same formula as for the $II_\infty$ case gives the order property for the theory of $(M,\varphi)$ as a $\sigma$-finite $W^*$ probability space. But we won't only use \cite[Th 5.6]{FarahII} to get $2^\mathfrak{c}$ non-isomorphic models  $(M,\varphi)^\mathcal{U}$, we will rather deduce a huge non-isomorphic class from the $II_\infty$ case. 

Consider first two ultrafilters with $(M,\varphi)^\mathcal{U}\simeq (M,\varphi)^\mathcal{V}$ as von Neumann algebras. We know from the axiomatization result theorem \ref{MoreAxiom} that they are $III_\lambda$ factors with a periodic state. Thus as above, on $(M,\varphi)^\mathcal{U}\otimes B(H)$, $\varphi^\mathcal{U}\otimes Tr$ is a generalized trace, thus from \cite[Th 4.3.2]{ConnesThesis}, it would be proportional to a unitary conjugate of $\varphi^\mathcal{V}\otimes Tr.$ Especially, both would have unitarily conjugated  centralizers $((M,\varphi)^\mathcal{U})_{\varphi^\mathcal{U}}\otimes B(H)\simeq ((M,\varphi)^\mathcal{V})_{\varphi^\mathcal{V}}\otimes B(H).$
By the lacunarity again, the ultraproduct of centralizers is nothing but the centralizer of the ultraproduct \cite[section 4.3]{HaagerupAndo} and thus
\[(M_\varphi\otimes B(H))^\mathcal{U}\simeq (M_\varphi)^\mathcal{U}\otimes B(H)\simeq (M_\varphi)^\mathcal{V}\otimes B(H)\simeq (M_\varphi\otimes B(H))^\mathcal{V}.\]

From the $II_\infty$ factor case above, we deduce we have $2^\mathfrak{c}$ non-isomorphic von Neumann algebras $(M_\varphi\otimes B(H))^\mathcal{U}$ and thus $2^\mathfrak{c}$ non-isomorphic von Neumann algebras $(M,\varphi)^\mathcal{U}$.

This completes the proof of Theorem 1. 
 
We conclude by a consequence in the spirit of  \cite[Th 6.11]{HaagerupAndo}
 
\begin{corollary}
Let $M$ be a  $III_\lambda$ factor with separable predual and $0<\lambda\leq 1$ and assume the continuum hypothesis fails. Then there are two ultrafilters $\mathcal{U},\mathcal{V}$ on $\N$ such that $\prod^\mathcal{U} M\not\simeq \prod^\mathcal{V} M$ as von Neumann algebras.
 \end{corollary}
 \begin{proof}
 It suffices to take the same ultrafilters as in theorem 1. If the isomorphism $\prod^\mathcal{U} M\simeq \prod^\mathcal{V} M$ were true, since any $\sigma$-finite projection of a type $III$-factor are equivalent (see e.g. \cite[Prop V.1.39]{TakesakiBook}), one would deduce from \cite[Corol 3.28]{HaagerupAndo} that for some $\sigma$-finite (support) projections $q,p$: $(M,\varphi)^\mathcal{U}\simeq q(\prod^\mathcal{U} M)q\simeq p(\prod^\mathcal{V} M)p\simeq (M,\varphi)^\mathcal{V}.$
 This would contradict theorem 1.
  \end{proof}




\end{document}